\documentclass[12pt]{amsart}

\usepackage{amsthm}
\newtheorem{theorem}{Theorem}[section]
\newtheorem{lemma}[theorem]{Lemma}
\newtheorem{proposition}[theorem]{Proposition}
\theoremstyle{definition}
\newtheorem{definition}[theorem]{Definition}

\theoremstyle{remark}
\newtheorem{remark}[theorem]{Remark}

\counterwithin*{section}{part}

\usepackage{amssymb}
\usepackage{empheq}
\usepackage{stmaryrd}
\usepackage{enumerate}
\usepackage[shortlabels]{enumitem}
\usepackage{calc}
\usepackage{url}
\usepackage{xcolor}

\newcommand{\norm}[1]{\left\lVert#1\right\rVert}
\usepackage{mathtools}
\DeclarePairedDelimiter{\ceil}{\lceil}{\rceil}

\usepackage{mathabx}

\usepackage[left=2.0cm,%
                right=2.0cm,%
                top=2.5cm,%
                bottom=3.5cm,%
                headheight=12pt,%
                a4paper]{geometry}%

\begin{document}

\title[A proof-theoretic metatheorem for nonlinear semigroups and applications]{A proof-theoretic metatheorem for nonlinear semigroups generated by an accretive operator and applications}

\author{Nicholas Pischke}
\date{\today}
\maketitle
\vspace*{-5mm}
\begin{center}
{\scriptsize Department of Mathematics, Technische Universit\"at Darmstadt,\\
Schlossgartenstra\ss{}e 7, 64289 Darmstadt, Germany, \ \\ 
E-mail: pischke@mathematik.tu-darmstadt.de}
\end{center}

\begin{abstract}
We further develop the theoretical framework of proof mining, a program in mathematical logic that seeks to quantify and extract computational information from prima facie `non-computational' proofs from the mainstream mathematical literature. To that end, we establish logical metatheorems that allow for the treatment of proofs involving nonlinear semigroups generated by an accretive operator, structures which in particular arise in the study of the solutions and asymptotic behavior of differential equations. In that way, the here established metatheorems facilitate a theoretical basis for the application of methods from the proof mining program to the wide variety of mathematical results established in the context of that theory since the 1970's. We in particular illustrate the applicability of the new systems and their metatheorems introduced here by providing two case studies on two central results due to Reich and Plant, respectively, on the asymptotic behavior of said semigroups and the resolvents of their generators where we derive rates of convergence for the limits involved which are, moreover, polynomial in all data.
\end{abstract}

\section{Introduction}

Proof mining is a program in mathematical logic that seeks to extract computational information, like (uniform) witnesses or bounds, from prima facie ``non-computational'' proofs from the mainstream mathematical literature. Historically, this endeavor goes back conceptually to Georg Kreisel's program of \emph{unwinding of proofs} from the 1950's \cite{Kre1951,Kre1952}. In its modern form, which relies on the use of well-known proof interpretations like \emph{negative translations}, Kreisel's \emph{modified realizability} and G\"odel's \emph{functional (Dialectica) interpretation}, it has been successfully developed since the 1990s by the groundbreaking work of Ulrich Kohlenbach (see in particular the early works \cite{Koh1990,Koh1992}) and his collaborators.

\medskip

The main applications of these methods are today found in the areas of nonlinear analysis and optimization and in that context, typical additional quantitative information extractable from non-constructive proofs have proved to be intimately connected with other approaches to finitary analysis. In particular, these logical methods in very general situations guarantee the existence of so-called \emph{rates of metastability} (in the sense of Terence Tao \cite{Tao2008b,Tao2008a}) for a finite quantitative account of a convergence result.

\medskip
 
The development of this modern period of proof mining is detailed comprehensively up to the year 2008 in the monograph \cite{Koh2008}, with early progress surveyed in \cite{KO2003} and recent progress in applications to the fields of nonlinear analysis and optimization surveyed in \cite{Koh2019}.

\medskip

The whole methodology of proof mining therein crucially relies on so-called \emph{general logical metatheorems} which guarantee the existence and quantify the complexity of such additional quantitative information.\footnote{Examples of such metatheorems may be found in \cite{GeK2006,GeK2008,GuK2016,Koh2005,KL2012,KN2017,Leu2006,Leu2014,Pis2022,Pis2024,PS2022,Sip2019}, as well as \cite{Koh2008} for the metatheorems obtained via (modifications of) G\"odel's Dialectica interpretation, and \cite{FLP2019} for subsequent metatheorems obtained via the bounded functional interpretation \cite{FO2005} due to F. Ferreira and P. Oliva.} Further, besides merely guaranteeing existence, these metatheorems provide an algorithmic approach towards actually extracting these quantitative results.

\medskip

In the context of nonlinear analysis, one of proof minings most successful fields of application, it is the theory of differential equations and neighboring fields like evolution equations where one so far still lacks a variety of proof mining applications with there so far only being three real case studies \cite{KKA2015,PP2022,Pis2023}. Already since the pioneering studies of Browder \cite{Bro1964}, Kato \cite{Kat1967} and Komura \cite{Kom1967}, two major tools in the study of nonlinear evolution equations have been the theory of nonlinear semigroups and their generators as well as the theory of accretive operators together with their correspondence via analogs of the Hille-Yosida theorem.

\medskip

One of the most important basic results in that context is the representation theorem due to Crandall and Liggett \cite{CL1971b} of the solution semigroup associated with the Cauchy problem
\[
\begin{cases}
u'(t)\in -Au(t),\, 0<t<\infty\\
u(0)=x
\end{cases}
\tag{$\dagger$}
\]
over a Banach space $X$ for a given set-valued accretive operator $A:X\to 2^X$, i.e.\ $A$ satisfies
\[
\norm{x-y+\lambda(u-v)}\geq\norm{x-y}
\]
for all $(x,u),(y,v)\in A$ and $\lambda\geq 0$. It is straightforward to show that any solution\footnote{A function $u:[0,\infty)\to X$ is a solution of $(\dagger)$ if $u(0)=x$, $u(t)$ is absolutely continuous, differentiable almost everywhere in $(0,\infty)$ and satisfies $(\dagger)$ almost everywhere. Note that this is often called a strong solution, but we omit this prefix here.} is unique as $A$ is accretive and if the system is solvable\footnote{As shown by Crandall and Liggett \cite{CL1971b}, this is (for strong solutions) in general not the case even for $A$ m-accretive and $\mathrm{dom}A=X$.}, then one can consider the family of operators $S(t)x=u_x(t)$ on $\mathrm{dom}A$ induced by the solutions $u_x(t)$ to $(\dagger)$ with initial values $x\in\mathrm{dom}A$. As these operators are continuous in $x$, one can consider the resulting extensions to $\overline{\mathrm{dom} A}$ which in that way generate the semigroup $\mathcal{S}=\{S(t)\mid t\geq 0\}$ on $\overline{\mathrm{dom}A}$ associated with $(\dagger)$. As shown by Brezis and Pazy \cite{BP1970}, this solution semigroup, if existent, has a particular fundamental representation in terms of a so-called exponential formula: 
\[
u_x(t)=\lim_{n\to\infty}\left(\mathrm{Id}+\frac{t}{n}A\right)^{-n}x.
\]
As shown subsequently by Crandall and Liggett \cite{CL1971b}, this formula actually always generates a nonexpansive semigroup on $\overline{\mathrm{dom}A}$ and thus facilitates a general study of equations like $(\dagger)$ even in the absence of solutions.

\medskip

Since the 1970's, an extensive range of results have been established in the theory of these semigroups and the initial value problems in the sense of ($\dagger$) associated with them, in particular in regard to the asymptotic behavior of the solutions of these differential equations, their connection and use in the study of partial differential equations and their use in the study of zeros of accretive operators (see \cite{Bar1976, Bar2010, BC2017, Miy1992, Pav1987}, among many more).

\medskip

In this paper, we extend the state-of-the-art of the underlying logical approach to proof mining to be applicable to proofs which make use of nonlinear semigroups generated by an accretive operator via the exponential formula. In particular, we establish logical metatheorems in the vein of the previously discussed results that guarantee, quantify and allow for the extraction of the computational content of theorems pertaining to these nonlinear semigroups. For that, we introduce new underlying logical systems that extend those developed for the treatment of accretive operators on normed spaces \cite{Pis2022} by carefully selected additional constants and corresponding axioms such that proofs from the mainstream literature become formalizable. To that end, we show that the initial key properties of these semigroups can be formally proved in these systems. In particular, these systems and metatheorems then further elucidate the extend of the phenomenon of so-called \emph{proof-theoretic tameness} of modern (nonlinear) analysis, i.e.\ the empirical fact that most proofs in e.g.\ analysis, although in principle being subject to well-known G\"odelian phenomena, nevertheless ``seem to be tame in the sense of allowing for the extraction of bounds of rather low complexity" \cite{Koh2020} (see also \cite{Mac2005,Mac2011} for further discussions of these types of phenomenas and their implications for logic and mathematics).

\medskip

These logical results provide a formal basis for the previous proof mining application \cite{KKA2015} carried out in the context of systems like $(\dagger)$ (generated by a certain subclass of accretive operators) and thus remove the ad-hoc nature surrounding it. Even further however, these results are expected to lead to many new case studies for proof mining in the context of that theory. In that vein, this paper provides two case studies on results due to Plant \cite{Pla1981} and Reich \cite{Rei1981} for the asymptotic behavior of these semigroups and in that context, under suitable quantitative translations of the assumptions used in the respective results, we are able to extract rates of convergence for the limits involved which are moreover polynomial in all data. In particular, we want to note that full rates of convergence are obtained here despite the fact that the sequence in question is not monotone and that the original proof is classical. This is due to a logical particularity that will be discussed after the extractions. The applicability of the present results is further substantiated by the fact that they also provide the logical basis for the previously mentioned applications \cite{Pis2022,Pis2023} in the context of results on the asymptotic behavior of these systems of differential equations due to Pazy \cite{Paz1978}, Nevanlinna and Reich \cite{NR1979}, Xu \cite{Xu2001} as well as Poffald and Reich \cite{PR1986}. At last, we want to mention that the whole logical apparatus for the treatment of these semigroups as developed in this paper requires some new technical tools. We expect that also these new logical approaches to these various notions in nonlinear analysis developed here will be of use in other circumstances than the ones described in this paper.

\section{Preliminaries}

\subsection{Nonlinear semigroups and the Crandall-Liggett formula}

The main objects of concern in this paper are the aforementioned nonlinear (and in this paper in particular nonexpansive) semigroups:
\begin{definition}
Let $C$ be a closed subset of $X$. A function $S:[0,\infty)\times C\to C$ is a \emph{(nonexpansive) semigroup on $C$} if
\begin{enumerate}
\item $S(t+s)x=S(t)S(s)x$ for all $x\in C$ and all $t,s\geq 0$,
\item $S(0)x=x$ for all $x\in C$,
\item $S(t)x$ is continuous in $t\geq 0$ for every $x\in C$,
\item $\norm{S(t)x-S(t)y}\leq\norm{x-y}$ for all $t\geq 0$ and all $x,y\in C$.
\end{enumerate}
\end{definition}

As discussed in the introduction already, these semigroups frequently arise in the study of differential and evolution equations as is e.g.\ exemplified by the initial value problem ($\dagger$). In particular, by the results of Crandall and Liggett \cite{CL1971b}, the exponential formula discussed before always generates such a semigroup on $\overline{\mathrm{dom}A}$ which will be the main object of study of this paper. Concretely, the following result was established in \cite{CL1971b}:
\begin{theorem}[Crandall and Liggett \cite{CL1971b}]\label{thm:CL}
Let $X$ be a Banach space and $A$ an accretive operator on $X$ such that there exists a $\lambda_0>0$ with
\[
\overline{\mathrm{dom}A}\subseteq\mathrm{ran}(\mathrm{Id}+\lambda A)\text{ for all }\lambda\in (0,\lambda_0].
\]
Then
\[
S(t)x:=\lim_{n\to\infty}\left(\mathrm{Id}+\frac{t}{n}A\right)^{-n}x
\]
exists for all $x\in\overline{\mathrm{dom}A}$ and $t\geq 0$ and $\mathcal{S}=\{S(t)\mid t\geq 0\}$ is a nonlinear semigroup on $\overline{\mathrm{dom}A}$.
\end{theorem}
We call $\mathcal{S}$ as defined above the semigroup generated by $A$ (via the exponential or Crandall-Liggett formula).\footnote{In fact, a large part of the literature calls $-A$ the generator of $\mathcal{S}$ (see e.g.\ \cite{Bar1976} and the references therein) to emphasize that the generator is dissipative. As we want to emphasize the accretiveness of the operator, we here deviated slightly from this convention.} We introduce further results and notions on and around these semigroups and accretive operators as needed throughout the paper and otherwise refer to \cite{Bar1976} for background.

\medskip

In terms of a logical treatment of these semigroups generated by an accretive operator, all of the later logical considerations naturally depend on the underlying theory of accretive operators over Banach spaces. In that vein, we logically crucially rely on the basic systems introduced in \cite{Pis2022} for the treatment of those accretive operators in the context of the extended systems of finite type commonly used in proof mining and we thus detail those in the next subsection.

\subsection{The basic system}

The basic system for accretive operators on normed spaces relies on the system $\mathcal{A}^\omega[X,\norm{\cdot}]$ introduced in \cite{GeK2008,Koh2005} as an underlying system for classical analysis over abstract normed spaces in all finite types $T^X$ defined by
\[
\mathbb{N},X\in T^X,\quad \rho,\tau\in T^X\Rightarrow \rho\to\tau\in T^X.
\]
We refer to those works, and to \cite{Koh2008} in general, for a precise exposition on the definition and basic properties of this and related systems. Accordingly, we mostly follow the notation used there as well as in \cite{Pis2022} (besides of using the above notation for the types which is of a more intuitive form and, in that vein, we also write $\mathbb{N}^\mathbb{N}$ for $\mathbb{N}\to\mathbb{N}$).

\medskip

In the context of these finite type systems, real numbers are as usual represented as fast-converging Cauchy sequences of rationals with a fixed rate. These are encoded via number theoretic functions, i.e.\ objects of type $\mathbb{N}^\mathbb{N}$, and on the level of that representation, one the can introduce the usual arithmetic operations and relations. Concretely, the relations $=_\mathbb{R}$, $\leq_\mathbb{R}$ operating on these type $\mathbb{N}^\mathbb{N}$ codes can then be chosen to be $\Pi^0_1$-formulas while $<_\mathbb{R}$ can be chosen to be a $\Sigma^0_1$-formula. 

\medskip

In general, we will omit the type of real numbers for arithmetical operations to make everything more readable. In proofs, we will almost always omit most types as to not distract from the general ideas and patterns.

\medskip

The starting point for our new systems will then by the theory $\mathcal{V}^\omega_p$ for the treatment of accretive operators on normed spaces as introduced in \cite{Pis2022}. This theory is defined as the extension of the theory $\mathcal{A}^\omega[X,\norm{\cdot}]$ with the additional constants 
\begin{itemize}
\item $\chi_A$ of type $X\to (X\to\mathbb{N})$ for the graph of the operator $A$,
\item $J^{\chi_A}$ of type $\mathbb{N}^\mathbb{N}\to (X\to X)$ for the resolvent $J^A_\gamma:=(\mathrm{Id}+\gamma A)^{-1}$ of $A$, 
\item $c_X$, $\widetilde{\gamma}$, $m_{\widetilde{\gamma}}$ of types $X$, $\mathbb{N}^\mathbb{N}$, $\mathbb{N}$, respectively, for a technical purpose in the majorization of the constants later on,
\end{itemize}
together with the corresponding axioms
\begin{enumerate}
\item[(I)] $\forall x^X,y^X(\chi_A xy\leq_\mathbb{N} 1)$,
\item[(II)] $\forall \gamma^{\mathbb{N}^\mathbb{N}},x^X\left(\gamma>_\mathbb{R}0\land \exists y^X\left(\gamma^{-1}(x-_Xy)\in Ay\right)\rightarrow \gamma^{-1}(x-_XJ^A_{\gamma}x)\in A(J^A_{\gamma}x)\right)$,
\item[(III)] $\begin{cases}\forall x^X,y^X,u^X,v^X,\lambda^{\mathbb{N}^\mathbb{N}}\big(u\in Ax\land v\in Ay\\
\qquad\rightarrow \norm{x-_Xy+_X\vert\lambda\vert(u-_Xv)}_X\geq_\mathbb{R}\norm{x-_Xy}_X\big),\end{cases}$
\item[(IV)] $\widetilde{\gamma}\geq_\mathbb{R}2^{-m_{\widetilde{\gamma}}}$,
\item[(V)] $\forall\gamma^{\mathbb{N}^\mathbb{N}}\left(\gamma>_\mathbb{R}0\rightarrow\gamma^{-1}(c_X-_XJ^A_{\gamma}c_X)\in A(J^A_{\gamma}c_X)\right)$,
\end{enumerate}
where here, and in the following, we write $J^A_\gamma$ for $J^{\chi_A}\gamma$ as well as $y\in Ax$ or $(x,y)\in A$ for $\chi_Axy=_\mathbb{N}0$. Further, as in \cite{Pis2022}, we use the abbreviation
\[
x\in\mathrm{dom}J^A_\gamma:\equiv \exists y^X\left(\gamma^{-1}(x-_Xy)\in Ay\right)
\]
and in that way can recognize the second axiom as stating
\[
\forall \gamma^{\mathbb{N}^\mathbb{N}},x^X\left(\gamma>_\mathbb{R}0\land x\in\mathrm{dom}J^A_\gamma\rightarrow \gamma^{-1}(x-_XJ^A_{\gamma}x)\in A(J^A_{\gamma}x)\right)
\]
which thereby specifies the behavior of the resolvent on its domain as dictated by its defining equality $J^A_\gamma:=(\mathrm{Id}+\gamma A)^{-1}$. We refer to \cite{Pis2022} for an (extensive) discussion of the motivation for and the particularities of this axiomatization (in particular regarding the use of reciprocals of reals) and here just note the restriction put in place by axiom (V) that the constant $c_X$ designates a common element of the domains of all resolvents $J^A_\gamma$ for $\gamma>0$. As discussed in \cite{Pis2022}, this assumption is easily satisfied for most applications which in particular include those situations where one assumes a range condition like
\[
\mathrm{dom} A\subseteq\bigcap_{\gamma>0}\mathrm{ran}(\mathrm{Id}+\gamma A)
\]
which will be the case in this work in particular as will be discussed in the coming sections.

\medskip

As shown in \cite{Pis2022}, the main parts of the basic theory of accretive operators and in particular their resolvents can then be immediately formally derived in the system $\mathcal{V}^\omega_p$ and we give an indication of that in the following lemma. For that, we also formally introduce the Yosida approximate  $A_\gamma$ defined via
\[
A_\gamma:=\frac{1}{\gamma}(\mathrm{Id}-J^A_\gamma)
\]
in the context of the formal system by treating $A_\gamma x$ as an abbreviation for the term $\gamma^{-1}(x-_XJ^A_\gamma x)$.\footnote{Note the discussion given in \cite{Pis2022} on the subtleties of the reciprocal of real arithmetic in the context of these systems of finite types and the resulting subtleties of the above definition.} 

\begin{lemma}[\cite{Pis2022}]\label{lem:basicProp}
The system $\mathcal{V}^\omega_{p}$ proves:
\begin{enumerate}
\item $J^A_\gamma$ is unique for any $\gamma>0$, i.e.\ 
\[
\forall \gamma^{\mathbb{N}^\mathbb{N}},p^X,x^X\left(\gamma>_\mathbb{R}0\land\gamma^{-1}(x-_Xp)\in Ap\rightarrow p=_X J^A_{\gamma} x\right).
\]
\item $J^A_\gamma$ is firmly nonexpansive for any $\gamma>0$ (on its domain), i.e.
\begin{gather*}
\forall\gamma^{\mathbb{N}^\mathbb{N}},r^{\mathbb{N}^\mathbb{N}},x^X,y^X\Big(\gamma>_\mathbb{R}0\land x\in\mathrm{dom}J^A_\gamma\land y\in\mathrm{dom}J^A_\gamma\land r>_\mathbb{R}0\\
\rightarrow\norm{J^A_\gamma x-_XJ^A_\gamma y}_X\leq_\mathbb{R}\norm{r(x-_Xy)+_X(1-r)(J^A_\gamma x-_XJ^A_\gamma y)}_X\Big).
\end{gather*}
\item $J^A_\gamma$ is nonexpansive for any $\gamma>0$ (on its domain), i.e.
\begin{gather*}
\forall\gamma^{\mathbb{N}^\mathbb{N}},x^X,y^X\Big(\gamma>_\mathbb{R}0\land x\in\mathrm{dom}J^A_\gamma\land y\in\mathrm{dom}J^A_\gamma\\
\rightarrow \norm{x-_Xy}_X\geq_\mathbb{R}\norm{J^A_\gamma x-_XJ^A_\gamma y}_X\Big).
\end{gather*}
\item $J^A$ is extensional in both arguments (on its domain), i.e.
\begin{gather*}
\forall \gamma^{\mathbb{N}^\mathbb{N}}>_\mathbb{R}0,x^X,{x'}^X\big(x\in\mathrm{dom}J^A_\gamma\land x'\in\mathrm{dom}J^A_\gamma\land x=_Xx'\\
\rightarrow J^A_\gamma x=_XJ^A_{\gamma'}x'\big),\\
\forall \gamma^{\mathbb{N}^\mathbb{N}}>_\mathbb{R}0,{\gamma'}^{\mathbb{N}^\mathbb{N}}>_\mathbb{R}0,x^X\big(x\in\mathrm{dom}J^A_\gamma\land x\in\mathrm{dom}J^A_{\gamma'}\land \gamma=_\mathbb{R}\gamma'\\
\rightarrow J^A_\gamma x=_XJ^A_{\gamma'}x\big).
\end{gather*}
\item $J^A$ satisfies the resolvent identity, i.e.
\begin{gather*}
\forall \gamma^{\mathbb{N}^\mathbb{N}},\lambda^{\mathbb{N}^\mathbb{N}},x^X\Big(\gamma>_\mathbb{R}0\land\lambda>_\mathbb{R}0\land x\in\mathrm{dom}J^A_\lambda\\
\rightarrow J^A_\lambda x=_X J^A_{\gamma}\left(\frac{\gamma}{\lambda} x+_X\left(1-\frac{\gamma}{\lambda}\right)J^A_\gamma x\right)\Big).
\end{gather*}
\item $J^A$ has controlled displacement, i.e.
\begin{gather*}
\forall \gamma^{\mathbb{N}^\mathbb{N}},\lambda^{\mathbb{N}^\mathbb{N}},x^X\Big(\gamma>_\mathbb{R}0\land\lambda>_\mathbb{R}0\land x\in\mathrm{dom}J^A_\gamma\land x\in\mathrm{dom}J^A_\lambda\\
\rightarrow\norm{x-_XJ^A_\gamma x}_X\leq_\mathbb{R}\left(2+\frac{\gamma}{\lambda}\right)\norm{x-_XJ^A_\lambda x}_X\Big).
\end{gather*}
\item $A_\gamma$ is $2\gamma^{-1}$-Lipschitz continuous for any $\gamma>0$, i.e.
\begin{gather*}
\forall\gamma^{\mathbb{N}^\mathbb{N}},x^X,y^X\Big(\gamma>_\mathbb{R}0\land x\in\mathrm{dom}J^A_\gamma\land y\in\mathrm{dom}J^A_\gamma\\
\rightarrow\norm{A_\gamma x-_XA_\gamma y}_X\leq_\mathbb{R}2\gamma^{-1}\norm{x-_Xy}_X\Big).
\end{gather*}
\item $A_\gamma x$ is bounded by any $y\in Ax$ for any $\gamma>0$, i.e.
\[
\forall\gamma^{\mathbb{N}^\mathbb{N}},x^X,y^X\left(\gamma>_\mathbb{R}0\land y\in Ax\land x\in\mathrm{dom}J^A_\gamma \rightarrow\norm{A_\gamma x}_X\leq_\mathbb{R}\norm{y}_X\right).
\]
In particular $\norm{x-_XJ_\gamma x}_X\leq_\mathbb{R} \gamma\norm{y}_X$.
\end{enumerate}
\end{lemma}

\subsection{The basic bound extraction theorems}

The main result established in \cite{Pis2022} is the logical metatheorem on bound extractions for $\mathcal{V}^{\omega}_p$ (and related systems) akin to the usual metatheorems of proof mining. Throughout the paper, we do not go into explicit detail regarding the proof of any bound extraction theorem and only provide sketches for the relevant additions and changes as the proofs otherwise follow the usual standard outline of most bound extraction results in proof mining established in \cite{GeK2008,Koh2005}.

\medskip

The prevalent central ingredient that we will focus on later in the context of the new bound extraction results will be the majorizability of the new constants used to treat nonlinear semigroups. Under majorizability, if not explicitly stated otherwise, we will here understand the extension due to \cite{GeK2008,Koh2005} of the strong majorizability of Bezem \cite{Bez1985} (which in turn builds on Howard's majorizability \cite{How1973}) to the new types in $T^X$. A fundamental paradigm of the bound extraction results is then that one achieves uniform bounds even in the absence of compactness by majorizing the bounds extracted by the underlying functional interpretation.

\medskip

In that way, following \cite{GeK2008,Koh2005}, majorants of objects with types from $T^X$ will be objects with types from $T$ related by the following projection:

\begin{definition}[\cite{GeK2008}]
Define $\widehat{\tau}\in T$, given $\tau\in T^X$, by recursion on the structure via
\[
\widehat{\mathbb{N}}:=\mathbb{N},\;\widehat{X}:=\mathbb{N},\;\widehat{\xi\to\tau}:=\widehat{\xi}\to\widehat{\tau}.
\]
\end{definition}

The majorizability relation is then defined in tandem with the structure of all majorizable functionals in the sense of the following definition.

\begin{definition}[\cite{GeK2008,Koh2005}]
Let $(X,\norm{\cdot})$ be a non-empty normed space. The structure $\mathcal{M}^{\omega,X}$ and the majorizability relation $\gtrsim_\tau$ are defined by
\[
\begin{cases}
\mathcal{M}_\mathbb{N}:=\mathbb{N}, n\gtrsim_\mathbb{N} m:=n\geq m\land n,m\in\mathbb{N},\\
\mathcal{M}_X:= X, n\gtrsim_X x:= n\geq\norm{x}\land n\in \mathcal{M}_\mathbb{N},x\in \mathcal{M}_X,\\
x^*\gtrsim_{\xi\to\tau}x:=x^*\in \mathcal{M}_{\widehat{\tau}}^{\mathcal{M}_{\widehat{\xi}}}\land x\in \mathcal{M}_\tau^{\mathcal{M}_\xi}\\
\hphantom{x^*\gtrsim_{\xi\to\tau}x:=}\land\forall y^*\in \mathcal{M}_{\widehat{\xi}},y\in \mathcal{M}_\xi(y^*\gtrsim_\xi y\rightarrow x^*y^*\gtrsim_\tau xy)\\
\hphantom{x^*\gtrsim_{\xi\to\tau}x:=}\land\forall y^*,y\in \mathcal{M}_{\widehat{\xi}}(y^*\gtrsim_{\widehat{\xi}}y\rightarrow x^*y^*\gtrsim_{\widehat{\tau}}x^*y),\\
\mathcal{M}_{\xi\to\tau}:=\left\{x\in \mathcal{M}_\tau^{\mathcal{M}_\xi}\mid \exists x^*\in \mathcal{M}^{\mathcal{M}_{\widehat{\xi}}}_{\widehat{\tau}}\left(x^*\gtrsim_{\xi\to\tau}x\right)\right\}.
\end{cases}
\]
\end{definition}

At a high level, the proofs of most bound extraction theorems then proceed as follows: using a variant of G\"odel's functional interpretation \cite{Goe1958} and a negative translation (e.g.\ \cite{Kur1951}), realizers are extracted from classical proofs of (essentially) $\forall\exists$-theorems. These realizers are then majorized to provide respective bounds which are validated in a model based on the structure of all majorizable functionals $\mathcal{M}^{\omega,X}$. If the types of all objects are low enough, one can then recover to the truth of the respective bound in a model based on the usual set-theoretic standard structure $\mathcal{S}^{\omega,X}$ of the underlying language defined by $\mathcal{S}_\mathbb{N}:=\mathbb{N}$, $\mathcal{S}_X:= X$ and
\[
\mathcal{S}_{\xi\to\tau}:=\mathcal{S}_{\tau}^{\mathcal{S}_{\xi}}.
\]

The resulting metatheorem for the case of the theory $\mathcal{V}^\omega_p$ then takes the form of the following theorem.  Here, we followed the names and notational conventions established in \cite{GeK2008,Koh2005} (see also \cite{Koh2008}) regarding so-called ``admissible'' types. which provide a formal perspective of the previously mentioned vague notion of ``low enough''.

\begin{theorem}[\cite{Pis2022}]\label{thm:metatheoremC}
Let $\tau$ be admissible, $\delta$ be of degree $1$ and $s$ be a closed term of $\mathcal{V}^\omega_p$ of type $\delta\to\sigma$ for admissible $\sigma$. Let $B_\forall(x,y,z,u)$/$C_\exists(x,y,z,v)$ be $\forall$-/$\exists$-formulas of $\mathcal{V}^{\omega}_p$ with only $x,y,z,u$/$x,y,z,v$ free. Let $\Delta$ be a set of formulas of the form $\forall\underline{a}^{\underline{\alpha}}\exists\underline{b}\preceq_{\underline{\beta}}\underline{r}\underline{a}\forall\underline{c}^{\underline{\zeta}}F_{qf}(\underline{a},\underline{b},\underline{c})$ where $F_{qf}$ is quantifier-free, the types in $\underline{\alpha}$, $\underline{\beta}$ and $\underline{\zeta}$ are admissible and $\underline{r}$ is a tuple of closed terms of appropriate type. If
\[
\mathcal{V}^{\omega}_p+\Delta\vdash\forall x^\delta\forall y\preceq_\sigma s(x)\forall z^\tau\left(\forall u^\mathbb{N} B_\forall(x,y,z,u)\rightarrow\exists v^\mathbb{N} C_\exists(x,y,z,v)\right),
\]
then one can extract a partial functional $\Phi:\mathcal{S}_{\delta}\times \mathcal{S}_{\widehat{\tau}}\times\mathbb{N}\rightharpoonup\mathbb{N}$ which is total and (bar-recursively) computable on $\mathcal{M}_\delta\times\mathcal{M}_{\widehat{\tau}}\times\mathbb{N}$ and such that for all $x\in \mathcal{S}_\delta$, $z\in \mathcal{S}_\tau$, $z^*\in \mathcal{S}_{\widehat{\tau}}$ and all $n\in\mathbb{N}$, if $z^*\gtrsim z$ and $n\geq_\mathbb{R}\norm{c_X-_XJ^A_{\widetilde{\gamma}}c_X}_X, m_{\widetilde{\gamma}},\vert\widetilde{\gamma}\vert,\norm{c_X}_X$, then
\begin{gather*}
\mathcal{S}^{\omega,X}\models\forall y\preceq_\sigma s(x)\big(\forall u\leq_\mathbb{N}\Phi(x,z^*,n) B_\forall(x,y,z,u)\\
\rightarrow\exists v\leq_\mathbb{N}\Phi(x,z^*,n)C_\exists(x,y,z,v)\big)
\end{gather*}
holds whenever $\mathcal{S}^{\omega,X}\models\Delta$ for $\mathcal{S}^{\omega,X}$ defined via any (nontrivial) normed space $(X,\norm{\cdot})$ with $\chi_A$ interpreted by the characteristic function of an accretive $A$ such that $\bigcap_{\gamma>0}\mathrm{dom}J^A_\gamma\neq\emptyset$, $J^{\chi_A}$ by the corresponding resolvents $J^A_\gamma$ for $\gamma>0$ and the other constants accordingly such that the corresponding axioms hold. 

Further: If $\widehat{\tau}$ is of degree $1$, then $\Phi$ is a total computable functional. If the claim is proved without $\mathrm{DC}$, then $\tau$ may be arbitrary and $\Phi$ will be a total functional on $\mathcal{S}_\delta\times \mathcal{S}_{\widehat{\tau}}\times\mathbb{N}$ which is primitive recursive in the sense of G\"odel's T. In that latter case, also plain majorization can be used instead of strong majorization.
\end{theorem}

The recent work \cite{KP2022} introduced a semi-constructive variant, in the spirit of \cite{GeK2006}, of another system from \cite{Pis2022} dealing with the treatment of maximally monotone operators over Hilbert spaces, and we similarly also want to consider a semi-constructive variant of the above system $\mathcal{V}^\omega_p$ in this work. Conceptually similar to the circumstances in \cite{GeK2006, KP2022}, this variant $\mathcal{V}^\omega_{i,p}$ is defined as an extension of $\mathcal{A}^\omega_i[X,\norm{\cdot}]+\mathrm{IP}_\neg+\mathrm{CA}_\neg$ with the basic system $\mathcal{A}^\omega_i=\mathrm{E}\text{-}\mathrm{HA}^\omega+\mathrm{AC}$ (defined as in \cite{GeK2006}) in the same manner as indicated above. We refer to the discussion given in \cite{GeK2006} on the resulting differences of the properties of $\mathcal{V}^\omega_{i,p}$ compared to $\mathcal{V}^\omega_p$ and in particular on the additional strength of the non-constructive principles allowed in the context of $\mathcal{V}^\omega_{i,p}$ while aiming for bound extractions.

Similar to \cite{KP2022}, we can show the following result by adapting \cite[Theorem 4.11]{GeK2006} (where now $\gtrsim$ denotes (not necessarily strong) majorization interpreted in the model $\mathcal{S}^{\omega,X}$):
\begin{theorem}\label{thm:metatheoremI}
Let $\delta$ be of degree $1$ and $\sigma,\tau$ be arbitrary, $s$ be a closed term of suitable type. Let $\Gamma_\neg$ be a set of sentences of the form $\forall\underline{a}^{\underline{\alpha}}(E(\underline{a})\rightarrow \exists\underline{b}\preceq_{\underline{\beta}}\underline{t}\underline{a}\neg F(\underline{a},\underline{b}))$ with $\underline{\alpha},\underline{\beta}$ and $E,F$ arbitrary types and formulas respectively and where $\underline{t}$ is a tuple of closed terms. Let $B(x,y,z)$/$C(x,y,z,u)$ be arbitrary formulas of $\mathcal{V}^{\omega}_{i,p}$ with only $x,y,z$/$x,y,z,u$ free. If
\[
\mathcal{V}^{\omega}_{i,p}+\Gamma_\neg\vdash\forall x^\delta\,\forall y\preceq_\sigma s(x)\,\forall z^\tau\,(\neg B(x,y,z)\rightarrow\exists u^\mathbb{N} C(x,y,z,u)), 
\]
one can extract a $\Phi:\mathcal{S}_\delta\times\mathcal{S}_{\widehat{\tau}}\times\mathbb{N}\to\mathbb{N}$ which is primitive recursive in the sense of G\"odel's T such that for any $x\in\mathcal{S}_\delta$, any $y\in\mathcal{S}_\sigma$ with $y\preceq_\sigma s(x)$, any $z\in\mathcal{S}_{\tau}$ and $z^*\in\mathcal{S}_{\widehat{\tau}}$ with $z^*\gtrsim z$ and any $n\in\mathbb{N}$ with $n\geq_\mathbb{R}\norm{c_X-_XJ^A_{\widetilde{\gamma}}c_X}_X, m_{\widetilde{\gamma}},\vert\widetilde{\gamma}\vert,\norm{c_X}_X$, we have that
\[
\mathcal{S}^{\omega,X}\models\exists u\leq_\mathbb{N}\Phi(x,z^*,n)\,(\neg B(x,y,z)\rightarrow C(x,y,z,u))
\]
holds whenever $\mathcal{S}^{\omega,X}\models\Gamma_\neg$ for $\mathcal{S}^{\omega,X}$ defined via any (nontrivial) normed spaces $(X,\norm{\cdot})$ with $\chi_A$ interpreted by the characteristic function of an accretive operator $A$ such that $\bigcap_{\gamma>0}\mathrm{dom}J^A_\gamma\neq\emptyset$, $J^{\chi_A}$ by the corresponding resolvents $J^A_\gamma$ for $\gamma>0$ and the other constants accordingly such that the corresponding axioms hold. 
\end{theorem}

We emphasize that while the classical logical metatheorems are derived using a monotone variant of G\"odel's Dialectica interpretation due to Kohlenbach (see \cite{Koh1996}), the above semi-constructive version rests on the use of a monotone variant due to Kohlenbach \cite{Koh1998} of Kreisel's modified realizability.

\section{Treating the normalized duality map and the alternative notion of accretivity}

Before we concern ourselves with the treatment of the semigroups, we need to extend the systems for accretive operators discussed previously in order to adequately deal with the associated notions around Theorem \ref{thm:CL}. In particular, we need to provide logical treatments of an alternative notion of accretivity, an extended range condition and the quantification over elements from the closure of the domain of $A$. We begin with the first of these in this section.

\subsection{The duality map and selection functionals}

Recall that for a Banach space $X$ with its dual space $X^*$, its normalized duality mapping
\[
J:X\to 2^{X^*}, x\mapsto \left\{x^*\in X^*\mid \langle x,x^*\rangle=\norm{x}^2=\norm{x^*}^2\right\}
\]
is non-empty for any $x\in X$ (which follows from the Hahn-Banach theorem). Many works in the context of the theory of accretive operators in general, and the treatment of semigroups generated by those operators in particular, rely on the use of this mapping and in that way, this section is concerned with a proof-theoretic treatment thereof in the context of the formal systems as discussed previously.

\medskip

As we want to refrain from providing a treatment for both the operator norm on the dual space as well as for the full duality map as a set-valued mapping, we follow the approach initiated by Kohlenbach and Leu\c{s}tean in \cite{KL2012} where the authors handle uses of $J$ by only treating certain selection functionals for $J$ (depending on the situation at hand).

\medskip

Concretely, a selection functional for the duality map $J$ is just a map $j:X\to X^*$ such that $j(x)\in J(x)$ for any $x\in X$. This general property of being a selection map can then be expressed by corresponding axioms formalizing that 
\begin{enumerate}
\item $jx:X\to\mathbb{R}$ is a linear operator for any $x\in X$;
\item $\norm{jx}\leq\norm{x}$ where $\norm{jx}$ means the operator norm;
\item $jxx=\norm{x}^2$ (which, as discussed in \cite{KL2012} already, yields $\norm{jx}=\norm{x}$). 
\end{enumerate}
Given a constant $j$ of type $X\to (X\to \mathbb{N}^\mathbb{N})$, this can then be encapsulated by the following universal axiom introduced in \cite{KL2012}:
\begin{gather*}
\forall x^X,y^X\Big(jxx=_\mathbb{R}\norm{x}_X^2\land\vert jxy\vert\leq_\mathbb{R}\norm{x}_X\norm{y}_X\\
\land \forall\alpha^{\mathbb{N}^\mathbb{N}},\beta^{\mathbb{N}^\mathbb{N}},u^X,v^X\left(jx(\alpha u+_X\beta v)=_\mathbb{R}\alpha jxu+_\mathbb{R}\beta jxv\right)\Big).
\end{gather*}
Notice that the operator norm is here avoided by expressing $\norm{jx}\leq\norm{x}$ via stipulating $\vert jxy\vert\leq_\mathbb{R}\norm{x}_X\norm{y}_X$.

\begin{remark}
As discussed in \cite{KL2012}, the functional $j$ is not provably extensional from the above axiom alone. As indicated by the use of the Dialectica interpretation, if extensionality is to be treated then one has to stipulate an associated modulus of uniform continuity which has been considered in \cite{KL2012}. The applications discussed later actually do not require an extensional or continuous selection map and we therefore do not explicitly discuss this issue any further.
\end{remark}

\subsection{The alternative notion of accretivity}\label{sec:jinterp}

Besides the purely metric notion of accretivity discussed in the preceding sections, which also forms the basis of the systems $\mathcal{V}^\omega_{p}$ and its intuitionistic variant $\mathcal{V}^\omega_{i,p}$, the more common notion of accretivity, especially in the context of nonlinear semigroups generated by such operators, is the notion introduced by Kato in \cite{Kat1967} where one stipulates that $A$ is accretive if
\[
\forall (x,u),(y,v)\in A\exists j\in J(x-y)\left( \langle u-v,j\rangle\geq 0\right).
\]
In the language of the preceding subsection, this can be recognized as stipulating the existence of a family of selection functionals $j_{u,v}$ such that, as before, $j_{u,v}x\in J(x)$ and where now further $\langle u-v,j_{u,v}(x-y)\rangle\geq 0$ for any $u\in Ax$ and $v\in Ay$.

\medskip

Formally, this leads us to the following modification of the previous system: we define $\widehat{\mathcal{V}}^\omega_p$ as the extension of $\mathcal{A}^\omega[X,\norm{\cdot}]$ with the axiom schemes (I), (II), (IV) and (V) as before, now over the language extended with a constant $j$ of type $X\to (X\to (X\to (X\to \mathbb{N}^\mathbb{N})))$ (or, in a more suggestive notation, $(X\times X\times X\times X)\to \mathbb{N}^\mathbb{N}$) together with the axioms
\begin{gather*}
\forall x^X,y^X,u^X,v^X\Big(\langle x,j_{u,v}x\rangle=_\mathbb{R}\norm{x}_X^2\land\vert \langle y,j_{
u,v}x\rangle\vert\leq_\mathbb{R}\norm{x}_X\norm{y}_X\tag{$J$}\\
\land \forall\alpha^{\mathbb{N}^\mathbb{N}},\beta^{\mathbb{N}^\mathbb{N}},z^X,w^X\left(\langle \alpha z+_X\beta w,j_{u,v}x\rangle =_\mathbb{R}\alpha \langle z,j_{u,v}x\rangle+\beta \langle w,j_{u,v}x\rangle\right)\Big)
\end{gather*}
as well as
\[
\forall x^X,y^X,u^X,v^X\left( u\in Ax\land v\in Ay \rightarrow\langle u-_Xv,j_{u,v}(x-_Xy)\rangle\geq_\mathbb{R} 0\right)\tag{$A$}
\]
where we write $j_{u,v}$ for $juv$ as well as $\langle y,j_{u,v}x\rangle$ for $juvxy$.

\medskip

It is rather immediately clear through the considerations made in \cite{KL2012} that the bound extraction theorems contained in Theorem \ref{thm:metatheoremC} and \ref{thm:metatheoremI} extend to the system $\widehat{\mathcal{V}}^\omega_p$, as we will discuss now. For this, we first have to give a suitable interpretation to the constant $j$ in the model $\mathcal{M}^{\omega,X}$ associated with an accretive operator $A$ (see \cite{Pis2022}). For that, note that the function $j$ is defined by contracting the two parameters  besides $u,v$, namely $x$ and $y$, into the one argument of $j$ (which is feasible as the witnessing functionals required by the notion of accretivity only have to satisfy $j\in J(x-y)$). The interpretation of this constant in the model now has to ``unwind'' this contraction (which essentially relies on a choice principle). Concretely, we are lead to the following interpretation of $j$ (writing $\mathcal{M}$ concisely for $\mathcal{M}^{\omega,X}$): given an accretive operator $A\subseteq X\times X$, define $[j]_\mathcal{M}$ by
\[
[j]_\mathcal{M}(u,v,z,w)=\begin{cases}(\langle w,j^A_{u,v}(z)\rangle)_\circ&\text{if }\exists x,y\in X\left( (x,u),(y,v)\in A\land z=_Xx-_Xy\right),\\
(\langle w,\widetilde{j}(z)\rangle)_\circ&\text{otherwise},
\end{cases}
\]
where $\langle\cdot,\cdot\rangle$ is application in the space $X^*$, the functionals $j^A_{u,v}(z)\in J(z)$ are those guaranteed to exist by the definition of accretivity (if such corresponding $x,y$ exist), $\widetilde{j}(z)$ is a generic element of $J(z)$ (which always exists as $J(z)\neq\emptyset$ by the Hahn-Banach theorem) and  $(\cdot)_\circ$ is the obvious extension to $\mathbb{R}$ of the operator $(\cdot)_\circ$ defined in \cite{Koh2005} on $[0,\infty)$ which selects, for a given real number, a canonical representation as a functional of type $\mathbb{N}^\mathbb{N}$. With this interpretation, the previous axioms are naturally satisfied in the model $\mathcal{M}^{\omega,X}$ associated with an accretive operator $A$.

Theorems \ref{thm:metatheoremC} and \ref{thm:metatheoremI} now extend to this setting as all the additional axioms ($J$) and ($A$) are purely universal and since the additional constant $j$ with its interpretation in the model $\mathcal{M}^{\omega,X}$ can be majorized by following the ideas presented in the proof of Theorem 2.2 in \cite{KL2012}: from $\vert\langle y,j_{u,v}x\rangle\vert\leq \norm{x}\norm{y}$, one obtains that $nm\geq\vert\langle y,j_{u,v}x\rangle\vert$ for $n\geq\norm{x}$ and $m\geq\norm{y}$ which immediately yields that the function
\[
(n,m,l,k)\mapsto (mn)_\circ
\]
defined for $n,m,k,l\in\mathbb{N}$ with $\norm{u}\leq k$, $\norm{v}\leq l$, $\norm{z}\leq m$, $\norm{w}\leq n$ is a majorant for $j$. Here $\circ$ is the previous operation, now restricted to $\mathbb{N}$ (which, as discussed in e.g.\ \cite{Koh2008}, can be explicitly calculated). This majorant is in particular actually independent on the arguments induced by the upper bounds on $\norm{u}$ and $\norm{v}$, i.e.\ $k$ and $l$.

\medskip

The question of how this notion of accretivity relates to the previously used notion immediately arises. By formalizing one direction of the proof on the equivalence of the two notions of accretivity (essentially due to Kato \cite{Kat1967}, see also Lemma 3.1 in Chapter II of \cite{Bar1976}), we obtain the following:
\begin{proposition}
The system $\widehat{\mathcal{V}}^\omega_{p}$ proves:
\begin{enumerate}
\item $\forall x^X,y^X,u^X,v^X\left( \langle y,j_{u,v}x\rangle\geq_\mathbb{R}0\rightarrow \forall\lambda^{\mathbb{N}^\mathbb{N}}\left(\norm{x}_X\leq_\mathbb{R}\norm{x+_X\vert\lambda\vert y}_X\right)\right)$.
\item $\begin{cases}\forall x^X,y^X,u^X,v^X,\lambda^{\mathbb{N}^\mathbb{N}}\big( (x,u),(y,v)\in A\\
\qquad\rightarrow\norm{x-_Xy+_X\vert\lambda\vert(u-_Xv)}_X\geq_\mathbb{R}\norm{x-_Xy}_X\big).\end{cases}$
\end{enumerate}
\end{proposition}
\begin{proof}
\begin{enumerate}
\item The conclusion is vacuously true for $x=0$. Thus assume $x\neq 0$ and let $\langle y,j_{u,v}x\rangle\geq 0$. Then we get
\begin{align*}
\norm{x}^2&=\langle x,j_{u,v}x\rangle\\
&=\langle x+\vert\lambda\vert y-\vert\lambda\vert y,j_{u,v}x\rangle \\
&=\langle x+\vert\lambda\vert y,j_{u,v}x\rangle -\vert\lambda\vert\langle y,j_{u,v}x\rangle\\
&\leq\langle x+\vert\lambda\vert y,j_{u,v}x\rangle\leq\norm{x+\vert\lambda\vert y}\norm{x}
\end{align*}
by $(J)$ and the quantifier-free extensionality rule. We have $\norm{x}\leq\norm{x+\vert\lambda\vert y}$ after dividing by $\norm{x}$.
\item By using $(A)$, we have $\langle u-v,j_{u,v}(x-y)\rangle\geq 0$ for $u\in Ax$ and $v\in Ay$. Then, we get $\norm{x-y}\leq\norm{x-y+\vert\lambda\vert(u-v)}$ by (1).
\end{enumerate}
\end{proof}

Therefore, the system $\widehat{\mathcal{V}}^\omega_p$ is an extension of $\mathcal{V}^\omega_p$ as all the axioms of $\mathcal{V}^\omega_p$ are provable in $\widehat{\mathcal{V}}^\omega_p$. In particular, all properties of $A$ and its resolvent exhibited in Lemma \ref{lem:basicProp} are provable in $\widehat{\mathcal{V}}^\omega_p$. Further, the system proves most of the basic facts about such duality selection mappings. One such fact that will be particularly useful later on is the following (proved, in passing, e.g.\ in the proof of Proposition 1.1 in Chapter I of \cite{Bar1976}):
\begin{proposition}\label{pro:DualMapTrick}
The system $\widehat{\mathcal{V}}^\omega_p$ proves:
\[
\forall x^X,y^X,u^X,v^X,t^{\mathbb{N}^\mathbb{N}}\left( t>_\mathbb{R}0\rightarrow \langle y,j_{u,v}x\rangle\leq_\mathbb{R}\norm{x}_X\frac{\norm{x+_Xty}_X-\norm{x}_X}{t}\right).
\]
\end{proposition}
\begin{proof}
We have
\[
\norm{x}^2+t\langle y,j_{u,v}x\rangle=\langle x+ty,j_{u,v}x\rangle\leq\norm{x}\norm{x+ty}
\]
by axiom ($J$). This implies 
\[
\langle y,j_{u,v}x\rangle\leq\norm{x}\frac{\norm{x+ty}-\norm{x}}{t}.\qedhere
\]
\end{proof}

\subsection{The mapping $\langle\cdot,\cdot\rangle_s$}\label{sec:supremumIP}

Of crucial importance in the context of many proofs from the theory of nonlinear semigroups, and in particular in the context of the exemplary applications considered later in this paper, is the use of a function $\langle\cdot,\cdot\rangle_s:X\times X\to\overline{\mathbb{R}}$ defined by
\[
\langle y,x\rangle_s:=\mathrm{sup}\left\{\langle y,j\rangle\mid j\in J(x)\right\}.
\]
As already observed in the early papers \cite{Bre1971,CL1971b}, it is easy to see that $\langle y,x\rangle_s<+\infty$ for all $x,y\in X$ and in fact, since $J(x)$ is weak-star compact in $X^*$, the supremum is actually attained.

\medskip

While $\langle\cdot,\cdot\rangle_s$ is by virtue of its definition via the supremum and the duality map $J$ a complex object in the formal contexts considered in this paper, many proofs only rely on the existence of a mapping which shares some essential properties with $\langle\cdot,\cdot\rangle_s$ and in that case, such a mapping can indeed be treated in the context of the systems discussed above and this is what we want to briefly discuss in the following.

\medskip

Concretely, under the ``essential properties'' mentioned above we will understand the following:
\begin{enumerate}
\item $\langle \alpha y,\beta x\rangle_s=\alpha\beta\langle y,x\rangle_s$ for $x,y\in X$ and $\alpha,\beta\geq 0$;
\item $\langle\alpha x+y,x\rangle_s=\alpha\norm{x}^2+\langle y,x\rangle_s$ for $x,y\in X$ and $\alpha\in\mathbb{R}$;
\item $\vert\langle y,x\rangle_s\vert\leq\norm{y}\norm{x}$ for $x,y\in X$;
\item $\langle y,j_{u,v}x\rangle\leq\langle y,x\rangle_s$ for $x,y\in X$ and $u,v\in X$ where the $j_{u,v}$ are the selection functionals for $J$ guaranteed by accretivity;
\item $\langle\cdot,\cdot\rangle_s$ is upper semicontinuous (in its right argument).
\end{enumerate}
For a proof for the items (1), (2) and (5), see Proposition 1.2 in Chapter I of \cite{Bar1976}. The other items are immediate.

\medskip

If all that is required of $\langle\cdot,\cdot\rangle_s$ in a proof is that it fulfills these properties, then this proof can, under suitable uniformization of these assumptions, be treated in the context of the above systems by adding a further constant $\langle\cdot,\cdot\rangle_s$ of type $X\to (X\to \mathbb{N}^\mathbb{N})$ together with the following axioms: the items (1) -- (4) are readily formulated as
\begin{gather*}
\forall x^X,y^X,\alpha^{\mathbb{N}^\mathbb{N}},\beta^{\mathbb{N}^\mathbb{N}}\left( \langle \vert\alpha\vert y,\vert\beta\vert x\rangle_s=_\mathbb{R}\vert\alpha\vert\vert\beta\vert\langle y,x\rangle_s\right),\tag*{$(+)_1$}\\
\forall x^X,y^X,\alpha^{\mathbb{N}^\mathbb{N}}\left(\langle\alpha x+_Xy,x\rangle_s=_\mathbb{R}\alpha\norm{x}_X^2+\langle y,x\rangle_s\right),\tag*{$(+)_2$}\\
\forall x^X,y^X\left(\vert\langle y,x\rangle_s\vert\leq_\mathbb{R}\norm{y}_X\norm{x}_X\right),\tag*{$(+)_3$}\\
\forall x^X,y^X,u^X,v^X\left(\langle y,j_{u,v}x\rangle\leq_\mathbb{R}\langle y,x\rangle_s\right),\tag*{$(+)_4$}
\end{gather*}
in the underlying language. For a suitable formulation of item (5), note that the logical methodology based on the monotone Dialectica interpretation suggest that the assumption is upgraded to the existence of a ``modulus of uniform upper semicontinuity'' $\omega^+$. Concretely, we will consider an additional constant $\omega^+$ of type $\mathbb{N}\to (\mathbb{N}\to \mathbb{N})$ together with the axiom
\begin{gather*}
\forall x^X,y^X,z^X,b^\mathbb{N},k^\mathbb{N}\Big(\norm{x}_X,\norm{z}_X<_\mathbb{R} b\land\norm{x-_Xy}_X<_\mathbb{R} 2^{-\omega^+(b,k)}\tag*{$(+)_5$}\\
\rightarrow\langle z,y\rangle_s\leq_\mathbb{R}\langle z,x\rangle_s+2^{-k}\Big).
\end{gather*}

Note that by the uniformity on $x$ where the rate only depends on the upper bound $b$, this is actually a full modulus of uniform continuity.

\medskip

The assumption that $\langle\cdot,\cdot\rangle_s$ is uniformly continuous is in particular true if the space is uniformly smooth and will be in particular also be necessary if the proof to be treated in some form uses the extensionality of the functional $\langle\cdot,\cdot\rangle_s$ (in its right argument) as suggested by the logical methodology. However, if that is not the case and the proof can be formalized just using the axioms $(+)_1,\dots,(+)_4$, then the bound extraction theorem established later in particular guarantees a bound which is valid in all Banach spaces.

\medskip

Note also that accretivity is sometimes defined by explicitly using the functional $\langle\cdot,\cdot\rangle_s$ through stating that
\[
\forall (x,u),(y,v)\in A\left( \langle u-_Xv,x-_Xy\rangle_s\geq_\mathbb{R} 0\right).
\]
This version of accretivity is immediately provable in the system $\widehat{\mathcal{V}}^\omega_p+(+)_4$ as, using axioms $(A)$ and $(+)_4$, we have
\[
\langle u-v,x-y\rangle_s\geq\langle u-v,j_{u,v}(x-y)\rangle\geq 0.
\]

We later denote the collection of these five axioms $(+)_1$ -- $(+)_5$ by $(+)$. Now, the bound extraction results contained in Theorems \ref{thm:metatheoremC} and \ref{thm:metatheoremI} also extend to the associated extended system(s) $\widehat{\mathcal{V}}^\omega_p+(+)_1+\dots+(+)_4+((+)_5)$ with the conclusion drawn over any space (or where $\langle\cdot,\cdot\rangle_s$ is additionally uniformly continuous on bounded subsets as above if $(+)_5$ is included). Concretely, this follows as before since, for one, all the axiom schemes are purely universal and, for another, the constant $\langle\cdot,\cdot\rangle_s$ can be immediately majorized: from $\vert\langle y,x\rangle_s\vert\leq\norm{y}\norm{x}$, we as before infer $mn\geq\vert\langle y,x\rangle_s\vert$ for $m\geq\norm{y}$ and $n\geq\norm{x}$. From this, a majorant for the accompanying interpretation (using the extension of $(\cdot)_\circ$ as before) in the model $\mathcal{M}^{\omega,X}$ follows. Further, the additional constant $\omega^+$ is immediately majorized (essentially by itself) as it is of type $\mathbb{N}\to (\mathbb{N}\to\mathbb{N})$ and so, similar to Lemma 17.82 of \cite{Koh2008}, we have that $\omega^{+,M}$ defined by
\[
\omega^{+,M}(b,k):=\max\{\omega^+(a,j)\mid a\leq b,j\leq k\}
\]
is a majorant for $\omega^+$.

\section{Treating nonlinear semigroups generated by the Crandall-Liggett formula}

In this section, we now are concerned with a formal treatment of the semigroup $\mathcal{S}$ generated by the exponential formula as guaranteed by the result of Crandall and Liggett \cite{CL1971b} previously discussed in Theorem \ref{thm:CL}. Before diving into the formal treatment of these semigroups, we however need to consider some preliminary formal results for the treatment of $\overline{\mathrm{dom}A}$ (which features in the premise of the range condition in Theorem \ref{thm:CL}) as well as how $J^A_0$ is to be understood.

\subsection{The treatment of $\overline{\mathrm{dom}A}$}\label{sec:closureOfDomain}

Crucial both for the definition of the semigroup and for the central assumption of Theorem \ref{thm:CL}, i.e.\ the range condition, is the use of the closure of the domain of $A$ and in the following formal investigations, quantification over elements from $\overline{\mathrm{dom}A}$ will therefore be necessary. All the previous systems essentially only considered normed spaces and in that context, we now first have to lift the previous treatment to take the completeness of the underlying Banach space into account. For that, we are following the approach laid out in \cite{Koh2008} by which complete spaces are treated by adding another operator $C$ of type $(\mathbb{N}\to X)\to X$ which is meant to assign to a Cauchy sequence $x^{\mathbb{N}\to X}$ a limit $C(x)$. To discard of the complex premise of Cauchyness in an axiom stating that property, one then restricts oneself to Cauchy sequences with a fixed Cauchy rate (similar to the representation of real numbers in finite type arithmetic, see \cite{Koh2008}). To implicitly quantify only over all such sequences, a term construction $\widehat{x}$ is used on the objects $x^{\mathbb{N}\to X}$. Precisely, $\widehat{x}$ is defined on the level of the representation of the real value of the norm via sequences of rational numbers with fixed Cauchy rate via\footnote{We here follow the notion of \cite{Koh2008} and denote by $[a](k)$ the $k$-th element of the Cauchy sequence representation of the real number $a$.}
\[
\widehat{x}_n=_X\begin{cases}
x_n&\text{if }\forall k<_\mathbb{N} n\left([\norm{x_k-_Xx_{k+1}}_X](k+1)<_\mathbb{Q}6\cdot 2^{-k-1}\right),\\
x_k&\text{for }\min k<_\mathbb{N} n:[\norm{x_k-_Xx_{k+1}}_X](k+1)\geq_\mathbb{Q}6\cdot 2^{-k-1}\text{, otherwise}.
\end{cases}
\]
Then, completeness of the space can be formulated via the universal axiom\footnote{See the discussion in \cite{Koh2008} for the necessity of the additional $+3$ in the formulation.}
\[
\forall x^{\mathbb{N}\to X},k^\mathbb{N}\left(\norm{C(x)-_X\widehat{x}_k}_X\leq_\mathbb{R}2^{-k+3}\right)\tag{$\mathcal{C}$}
\]
which indeed implies completeness of the space in the form that from
\[
\forall k^\mathbb{N}\exists n^\mathbb{N}\forall m,\widetilde{m}\geq_\mathbb{N} n\left( \norm{x_m-_Xx_{\widetilde{m}}}_X<_\mathbb{R}2^{-k}\right)
\]
it follows provably in $\mathcal{A}^\omega[X,\norm{\cdot}]+(\mathcal{C})$ that 
\[
\forall k^\mathbb{N}\exists m^\mathbb{N}\forall l\geq_\mathbb{N} m\left(\norm{C(x)-_Xx_l}_X<_\mathbb{R}2^{-k+1}\right).
\]

As further shown in \cite{Koh2008}, the constant $C$ is majorizable and therefore we find that the bound extraction theorems discussed above immediately extend to $\widehat{\mathcal{V}}^\omega_p+(\mathcal{C})$ or any suitable extension (e.g.\ by $(+)$).

\medskip

Now a statement where one is quantifying over the closure of the domain, i.e.\ a statement of the form
\[
\forall x\in\overline{\mathrm{dom}A}\left( B(x)\right)\tag{$*$}
\]
can, through the use of $C$, be (naively) expressed as
\[
\forall x^{\mathbb{N}\to X}\left(\forall n^\mathbb{N}\exists y^X(y\in A\widehat{x}_n)\rightarrow B(C(x))\right).
\]
The premise that $x^{\mathbb{N}\to X}$ is a Cauchy sequence was removed through the use of $\widehat{x}$ and $C$ but the inclusion of the sequence in the domain, in the form of $\forall n^\mathbb{N}\exists y^X(y\in A\widehat{x}_n)$, remains, which is needed to specify that the limit of $\widehat{x}$, i.e.\ $C(x)$, is indeed an element of $\overline{\mathrm{dom}A}$.

The approach is now to also remove this assumption in a similar style as the $\widehat{\cdot}$-operation by universally quantifying over the potential witnessing sequence $y_n$ and defining a subsequent operation similar to $\widehat{\cdot}$ which potentially alters the sequence such that $x_n\in\mathrm{dom}A$ will always be guaranteed for any $n$. Concretely, for two objects $x,y$ of type $\mathbb{N}\to X$, we define
\[
(x\upharpoonright y)_n=_X\begin{cases}
x_n&\text{if }\forall k\leq_\mathbb{N} n\left( y_k\in Ax_k\right),\\
x_{k\dotdiv 1}&\text{for }\min k\leq_\mathbb{N} n:y_k\not\in Ax_k\text{, otherwise}.
\end{cases}
\]
Note that since inclusions in the graph of $A$ are quantifier-free, the above indeed can be defined by a closed term in the underlying language.

Now, using the operation $\upharpoonright$ in tandem with $\widehat{\cdot}$, we can implicitly quantify over elements from $\overline{\mathrm{dom}A}$ by quantifying over elements of type $\mathbb{N}\to X$ and thus we can express the statement ($*$) equivalently by
\[
\forall x^{\mathbb{N}\to X},y^{\mathbb{N}\to X}\left( y_0\in Ax_0\rightarrow B(C(x\upharpoonright y))\right).
\]
As a feasibility check for using $x\upharpoonright y$ to specify elements in $\overline{\mathrm{dom}A}$, note first that 
\[
\widehat{x\upharpoonright y}=_{\mathbb{N}\to X}\widehat{x}\upharpoonright y.
\]
To see this, one can consider a case distinction on whether $\widehat{x}=x$ holds or not and simultaneously on whether $x\upharpoonright y=x$ holds or not. We only consider the one case out of the four where $\widehat{x}\neq x$ and $x\upharpoonright y\neq x$. By definition, we then have a least $k$ such that $[\norm{x_k-x_{k+1}}](k+1)\geq_\mathbb{Q}6\cdot 2^{-k-1}$ as well as a least $j$ such that $y_j\not\in Ax_j$. Then, it immediately follows by definition of the operations as well as the minimality of $k$ and $j$ that
\begin{align*}
\widehat{x}\upharpoonright y&=(x_0,\dots,x_k,x_k,\dots)\upharpoonright y\\
&=(x_0,\dots,x_{\min\{k,j\dotdiv 1\}},x_{\min\{k,j\dotdiv 1\}},\dots)\\
&=(x_0,\dots,x_{j\dotdiv 1},x_{j\dotdiv 1},\dots)\,\widehat{\;}\\
&=\widehat{x\upharpoonright y}
\end{align*}
where, in the third line, we wrote $(x_0,\dots,x_{j\dotdiv 1},x_{j\dotdiv 1},\dots)\,\widehat{\;}\,$ for the operation $\widehat{\,\cdot\,}$ applied to the respective sequence.

Further, note that the premise $y_0\in Ax_0$ actually guarantees that $(\widehat{x}\upharpoonright y)_n\in\mathrm{dom}A$ for all $n$. For this, define  
\[
(x\upharpoonleft y)_n=_X\begin{cases}
y_n&\text{if }\forall k\leq_{\mathbb{N}} n\left( y_k\in Ax_k\right),\\
y_{k\dotdiv 1}&\text{for }\min k\leq_{\mathbb{N}} n:y_k\not\in Ax_k\text{, otherwise}.
\end{cases}
\]
Then clearly $y_0\in Ax_0$ implies $(\widehat{x}\upharpoonleft y)_n\in A((\widehat{x}\upharpoonright y)_n)$ for any $n$.

\subsection{Range conditions}

A treatment for some variants of the range conditions was already briefly discussed in \cite{Pis2022} where the particular case of
\[
\mathrm{dom}A\subseteq\bigcap_{\lambda>0}\mathrm{ran}(\mathrm{Id}+\lambda A)
\]
was studied. As discussed there, one can provide a formalized version of this range condition by making use of the resolvent in the form of the following sentence:
\[
\forall x^X,\lambda^{\mathbb{N}^\mathbb{N}}\left(x\in\mathrm{dom}A\land \lambda>_{\mathbb{R}}0\rightarrow \lambda^{-1}(x-_XJ^A_\lambda x)\in A(J^A_\lambda x)\right).
\]
This correctly expresses the range condition since stating that $x\in\mathrm{ran}(\mathrm{Id}+\lambda A)$ is equivalent to stating that $x\in\mathrm{dom}J^A_\lambda$ just via the definition of the resolvent. This latter statement is now equivalently formally encapsulated in our systems by stating the inclusion $\lambda^{-1}(x-_XJ^A_\lambda x)\in A(J^A_\lambda x)$. Note also that this axiom is in particular purely universal and thus can be used in the bound extraction theorems.

\medskip

In the following, we want to consider two modifications: (1) we want to specify that the inclusion is valid even for the closure of the domain; (2) we want to restrict the intersection to $\lambda<\lambda_0$ for some real parameter $\lambda_0>0$. The use of such a $\lambda_0$ can be facilitated by adding two further constants and an axiom: $\lambda_0$ of type $\mathbb{N}^\mathbb{N}$ and $m_{\lambda_0}$ of type $\mathbb{N}$ together with the accompanying axiom $\lambda_0\geq_\mathbb{R} 2^{-m_{\lambda_0}}$ providing a verifier to $\lambda_0>0$. Note that the bound extraction results stay valid in the context of such an extension if one additionally requires the parameter $n$ from Theorem \ref{thm:metatheoremC} to satisfy $n\geq\vert\lambda_0\vert,m_{\lambda_0}$.

\medskip

In the context of such additional constants, the above range condition can be immediately modified to represent the restricted range condition
\[
\mathrm{dom}A\subseteq\bigcap_{\lambda_0>\lambda>0}\mathrm{ran}(\mathrm{Id}+\lambda A)
\]
by considering
\[
\forall x^X,\lambda^{\mathbb{N}^\mathbb{N}}\left(x\in\mathrm{dom}A\land \lambda_0>_\mathbb{R}\lambda>_\mathbb{R} 0\rightarrow \lambda^{-1}(x-_XJ^A_\lambda x)\in A(J^A_\lambda x)\right).
\]

Further, in both cases we can now consider the other main modification of stipulating the range condition also for the closure of the domain, i.e.
\[
\overline{\mathrm{dom}A}\subseteq\bigcap_{\lambda_0>\lambda>0}\mathrm{ran}(\mathrm{Id}+\lambda A),
\]
by using the above treatment of quantification over elements in the closure of the domain by quantification over sequences in $X$ together with the operators $C$ and $(\cdot\upharpoonright\cdot)$. Concretely, one rather immediately obtains the following natural extension to the closure of the domain:
\begin{gather*}
\forall x^{\mathbb{N}\to X},v^{\mathbb{N}\to X},\lambda^{\mathbb{N}^\mathbb{N}}\Big(v_0\in Ax_0\land \lambda_0>_\mathbb{R}\lambda>_\mathbb{R} 0\tag*{$(RC)_{\lambda_0}$}\\
\rightarrow \lambda^{-1}(C(x\upharpoonright v)-_XJ^A_\lambda(C(x\upharpoonright v)))\in A(J^A_\lambda(C(x\upharpoonright v)))\Big).
\end{gather*}
Similarly, we could here lift the restriction via $\lambda_0$ again and get a full range condition for the closure of the domain. We denoted this full range condition for the closure of the domain by $(RC)$, but at the same time refrain from spelling this out in any more detail here. Note however that all the other range conditions introduced here are still purely universal and thus are admissible in the context of the bound extraction theorems.\\

Further, note that e.g.\ from $(RC)_{\lambda_0}$, the statement 
\[
\forall x^X,\lambda^{\mathbb{N}^\mathbb{N}}\left(x\in\mathrm{dom}A\land \lambda_0>_\mathbb{R}\lambda>_\mathbb{R} 0\rightarrow \lambda^{-1}(x-_XJ^A_\lambda x)\in A(J^A_\lambda x)\right)
\]
is provable: if $x\in\mathrm{dom}A$ with $v\in Ax$, consider the constant-$x$ and constant-$v$ sequences $\overline{x}$ and $\overline{v}$, respectively. Then clearly $(\overline{x}\upharpoonright \overline{v})_n=_Xx$ for any $n$ and thus provably $C(\overline{x}\upharpoonright \overline{v})=_Xx$ by $(\mathcal{C})$. The statement $(RC)_{\lambda_0}$ yields
\[
\lambda^{-1}(C(\overline{x}\upharpoonright \overline{v})-_X J^A_\lambda(C(\overline{x}\upharpoonright \overline{v}))) \in A(J^A_\lambda(C(\overline{x}\upharpoonright \overline{v})))
\]
for $\lambda_0> \lambda> 0$ and the quantifier-free extensionality rule (as $v\in Ax$ is quantifier-free) yields $\lambda^{-1}(x-_XJ^A_\lambda x)\in A(J^A_\lambda x)$.\\

In the following remark, we lastly collect some subtleties regarding the extension of the metatheorems to systems with these types of axioms.

\begin{remark}\label{rem:modelRem}
The metatheorems exhibited in Theorems \ref{thm:metatheoremC} and \ref{thm:metatheoremI} require as an assumption that $\bigcap_{\lambda>0}\mathrm{dom}J^A_\lambda\neq\emptyset$, a requirement which would be substantiated via a full range condition together with a witness for $\mathrm{dom}A\neq\emptyset$ (which was previously, in some sense but not precisely, represented by $c_X$). In the context of the above restricted range conditions, it is however feasible that $\bigcap_{\lambda>0}\mathrm{dom}J^A_\lambda$ is actually empty while only $\bigcap_{\lambda_0>\lambda>0}\mathrm{dom}J^A_\lambda\neq\emptyset$ holds. It should be noted that in this case, Theorems \ref{thm:metatheoremC} and \ref{thm:metatheoremI} can be modified to stay valid if $c_X$ is interpreted by a point in this restricted intersection. Therefore, if we in the following write $\widehat{\mathcal{V}}^\omega_p+(\mathcal{C})+(RC)_{\lambda_0}$ or consider any extension, we consider the axioms (IV) and (V) to be replaced by
\begin{enumerate}
\item[(IV)$'$] $\lambda_0-2^{m'_{\widetilde{\gamma}}}\geq_\mathbb{R}\widetilde{\gamma}\geq_\mathbb{R}2^{-m_{\widetilde{\gamma}}}$,
\item[(V)$'$] $d_X\in Ac_X$,
\end{enumerate}
where $d_X$ is a new constant of type $X$ and $m'_{\widetilde{\gamma}}$ is a new constant of type $\mathbb{N}$, the latter witnessing that $\lambda_0>\widetilde{\gamma}$. The majorization of all resolvents $J^A_\gamma$ for $\gamma\in (0,\lambda_0)$ is then achieved similar to \cite{Pis2022} via
\begin{align*}
\norm{J^A_\gamma x}&\leq \norm{x}+2\norm{c_X}+\left(2+\frac{\gamma}{\widetilde{\gamma}}\right)\norm{c_X-J^A_{\widetilde{\gamma}}c_X}\\
&\leq \norm{x}+2\norm{c_X}+\left(2\widetilde{\gamma}+\gamma\right)\norm{d_X}.
\end{align*}
In that case however, the interpretation of the resolvent constant $J^{\chi_A}$ in the models $\mathcal{M}^{\omega,X}$ and $\mathcal{S}^{\omega,X}$ has to be modified to set $[J^{\chi_A}]_\mathcal{M}(\gamma,x)=0$ for all $x$ if $\gamma\geq_\mathbb{R}\lambda_0$ (and similar for $\mathcal{S}^{\omega,X}$). Therefore, the extracted bounds only remain meaningful if the theorem does not utilize these resolvents. If it does, further modifications are necessary but we refrain from discussing this here any further as this situation does not arise in this paper.
\end{remark}

\subsection{The resolvent at zero}\label{sec:resAtZero}

Something left open by the axioms characterizing the resolvent, as discussed in the preliminaries, is the behavior of $J^A_0$. This, however, takes a special role in the context of the treatment of nonlinear semigroups $S$ generated by the associated operator $A$ due to the prominent use often made of $S(0)$.

The reason for this previous ambiguity in the treatment of the resolvent at $0$ was the fact that the resolvent does not always behave continuously at $0$ if it is naively defined: while the definition of the resolvent via
\[
J^A_\gamma=(\mathrm{Id}+\gamma A)^{-1}
\]
suggests $J^A_0x=x$, it is well known (see \cite{BC2017}) that already in Hilbert spaces with a maximally monotone operator $A$, one has $J^A_tx\to P_{\overline{\mathrm{dom}A}}x$ for $t\to 0$ and all $x\in\mathrm{dom}J^A_t$. Therefore, extensionality for the constant $J^{\chi_A}$ in its first argument $t$ at $0$ can in general not be expected if $J^A_0$ is defined in this way and the previous axiomatization left the definition of $J^A_0$ open.

\medskip

In the following, we nevertheless consider the set of axioms discussed previously forming $\widehat{\mathcal{V}}^\omega_p$ to actually be extended with the sixth axiom
\begin{enumerate}
\item[(VI)] $\forall x^X\left(J^A_0x=_Xx\right)$,
\end{enumerate}
stating the defining equality $J^A_0=(\mathrm{Id}+0 A)^{-1}=\mathrm{Id}$.

\medskip

Now, the above result that $J^A_tx\to P_{\overline{\mathrm{dom}A}}x$ for $t\to 0$ extends to Banach spaces at least partially in the sense that one can show (see Proposition 3.2 of Chapter II in \cite{Bar1976}) that $J^A_tx\to x$ for $\lambda_0> t\to 0$ and 
\[
x\in\overline{\mathrm{dom}A}\cap\bigcap_{\lambda_0>\lambda>0}\mathrm{dom}J^A_\lambda.
\]
Therefore, in the presence of a range condition, we should at least have a continuous and thus extensional behavior of the resolvent defined in this manner at $t=0$ for all $x\in\overline{\mathrm{dom}A}$ and this can indeed be formally verified in the accompanying system.

\begin{lemma}
$\widehat{\mathcal{V}}^\omega_p+(\mathcal{C})+(RC)_{\lambda_0}$ proves:
\begin{gather*}
\forall x^{\mathbb{N}\to X},v^{\mathbb{N}\to X},\lambda^{\mathbb{N}^\mathbb{N}},k^\mathbb{N}\Bigg(v_0\in Ax_0\land 0<_\mathbb{R}\lambda<_\mathbb{R} \min\left\{\frac{2^{-(k+1)}}{\max\{1,\norm{(\widehat{x}\upharpoonleft v)_{k+5}}_X\}},\lambda_0\right\}\\
\to \norm{C(x\upharpoonright v)-_XJ^A_\lambda C(x\upharpoonright v)}_X\leq_\mathbb{R} 2^{-k}\Bigg).
\end{gather*}
\end{lemma}
\begin{proof}
First, by Lemma \ref{lem:basicProp}, we have
\[
\forall x^X,v^X,\lambda^{\mathbb{N}^\mathbb{N}}\left(0<_\mathbb{R}\lambda<_\mathbb{R}\lambda_0\land v\in Ax\rightarrow \norm{x-J^A_\lambda x}\leq\lambda\norm{v}\right)
\]
as using $(RC)_{\lambda_0}$ and the quantifier-free extensionality rule, we obtain $x\in\mathrm{dom}J^A_\lambda$ for all $\lambda\in (0,\lambda_0)$. So, for $x^{\mathbb{N}\to X}$ and $v^{\mathbb{N}\to X}$ such that $v_0\in Ax_0$, we get $C(x\upharpoonright v)\in\mathrm{dom}J^A_\lambda$ for all $\lambda\in (0,\lambda_0)$ again by $(RC)_{\lambda_0}$. Therefore, using $(\mathcal{C})$ and the nonexpansivity of the resolvent on its domain:
\begin{align*}
\norm{C(x\upharpoonright v)-J^A_\lambda C(x\upharpoonright v)}&\leq \norm{C(x\upharpoonright v)-(\widehat{x}\upharpoonright v)_n}+\norm{(\widehat{x}\upharpoonright v)_n-J^A_\lambda(\widehat{x}\upharpoonright v)_n}\\
&\qquad+\norm{J^A_\lambda(\widehat{x}\upharpoonright v)_n-J^A_\lambda C(x\upharpoonright v)}\\
&\leq 2\norm{C(x\upharpoonright v)-(\widehat{x}\upharpoonright v)_n}+\norm{(\widehat{x}\upharpoonright v)_n-J^A_\lambda(\widehat{x}\upharpoonright v)_n}\\
&\leq 2\cdot 2^{-n+3}+\lambda\norm{(\widehat{x}\upharpoonleft v)_n}.
\end{align*}
Choosing $n=k+5$, we get that for $\lambda\leq 2^{-(k+1)}/\max\{1,\norm{(\widehat{x}\upharpoonleft v)_{k+5}}\}$:
\[
\norm{C(x\upharpoonright v)-J^A_\lambda C(x\upharpoonright v)}\leq 2^{-k}.\qedhere
\]
\end{proof}

This property will be sufficient in the following as the semigroup operates only on $\overline{\mathrm{dom}A}$.

\subsection{The semigroup}

For treating the semigroup on $\overline{\mathrm{dom}A}$ from Theorem \ref{thm:CL}, it is very instructive to first consider the operator $S$ solely on $\mathrm{dom}A$. In that case, we can facilitate a treatment by directly adding a further constant $S$ of type $\mathbb{N}^\mathbb{N}\to (X\to X)$ to the underlying language together with an axiom stating that $S$ on $\mathrm{dom}A$ arises from the Crandall-Liggett formula, i.e.\ that 
\[
S(t)x=\lim_{n\to\infty}\left(\mathrm{Id}+\frac{t}{n}A\right)^{-n}x
\]
for any $x\in\mathrm{dom}A$. This can be achieved by further adding a constant $\omega^S$ of type $\mathbb{N}\to (\mathbb{N}\to (\mathbb{N}\to \mathbb{N}))$ together with the axiom
\begin{gather*}
\forall k^\mathbb{N},b^\mathbb{N},T^\mathbb{N},x^X,v^X,t^{\mathbb{N}^\mathbb{N}}\bigg( v\in Ax\land \norm{x}_X,\norm{v}_X<_\mathbb{R}b\land \vert t\vert<_\mathbb{R}T\tag{S1}\\
\rightarrow\forall n\geq_\mathbb{N}\omega^S(k,b,T)\left(\vert t\vert/n<_\mathbb{R}\lambda_0\rightarrow \norm{S(\vert t\vert)x-_X(J^A_{\vert t\vert/n})^nx}_X\leq_\mathbb{R}2^{-k}\right)\bigg),
\end{gather*}
expressing that $\omega^S$ represents a rate of convergence uniform for elements $x$ from bounded subsets $B_b(0)\cap\mathrm{dom}A$ and uniform in $t$ for bounded intervals $[0,T]$ (where we use the absolute value to disperse of the universal premise $t\geq 0$). The term $(J^A_{\vert t\vert/n})^n$ used here is a shorthand for a term $I(t)(n)(n)$ where $I(t)(m)$ is a closed term of type $\mathbb{N}\to (X\to X)$ defined using the recursors of the underlying language of $\mathcal{A}^\omega[X,\norm{\cdot}]$ (see \cite{Koh2008}) via $I(t)(m)(0)=\lambda x.x$ and $I(t)(m)(n+1)=\lambda x.(J^A_{t/m}(I(t)(m)(n)(x)))$.\footnote{We consider $I(t)(m)$ to be trivially defined at $m=0$} Note also that we in particular treat $S(0)x$ via $J^A_0x$ by using the absolute value $\vert t\vert$ in the above formula to implicitly quantify over non-negative real numbers.

\medskip

Such a use of a rate of convergence is in particular justified by the fact that the proof given in \cite{CL1971b} of the Cauchy-property of the sequence $(J^A_{t/n})^nx$ for given $t>0$ and $x\in\mathrm{dom}A$ can be immediately recognized to be provable in the system $\widehat{\mathcal{V}}^\omega_{i,p}+(\mathcal{C})+(RC)_{\lambda_0}$ (naturally defined similar to $\widehat{\mathcal{V}}^\omega_p+(\mathcal{C})+(RC)_{\lambda_0}$, just over $\mathcal{A}^\omega_i[X,\norm{\cdot}]$ instead of $\mathcal{A}^\omega[X,\norm{\cdot}]$). Therefore, the extension of the semi-constructive metatheorem (Theorem \ref{thm:metatheoremI}) to this system guarantees the existence of a rate of Cauchyness for $(J^A_{t/n})^nx$ and consequently the existence of a modulus $\omega^S$ as characterized by the above axiom which can moreover be extracted from the proof given in \cite{CL1971b} (which is in fact rather immediate and was essentially already observed in \cite{CL1971b}): one can (formally) show that given $x\in\mathrm{dom}A$ with witness $v\in Ax$ and $t\geq 0$, we have
\[
\norm{(J^A_{t/n})^nx-(J^A_{t/m})^mx}\leq 2t\left\vert\frac{1}{m}-\frac{1}{n}\right\vert^{1/2}\norm{v}.
\]
Thus for $T> t$ and $b>\norm{v}$, we have for a given $\varepsilon>0$ that for any $m\geq n\geq\ceil*{\frac{4T^2b^2}{\varepsilon^2}}$:
\begin{align*}
\norm{(J^A_{t/n})^nx-(J^A_{t/m})^mx}&\leq 2Tb\left\vert\frac{1}{m}-\frac{1}{n}\right\vert^{1/2}\\
&\leq 2Tb\frac{1}{\sqrt{n}}\\
&\leq 2Tb\frac{1}{\sqrt{\ceil*{\frac{4T^2b^2}{\varepsilon^2}}}}\\
&\leq \varepsilon.
\end{align*}
Thus the mapping
\[
\omega^S(k,b,T)=2^{2k+2}T^2b^2
\]
is a possible choice for the rate of convergence\footnote{Note that although the function is exponential in $k$, this is just due to requiring an error of the form $2^{-k}$. Abstracting $\varepsilon=2^{-k}$, the rate is actually quadratic in $1/\varepsilon$.} in the exponential formula as derived from the proof and the upper bound $b$ is here actually even independent of $\norm{x}$.

\medskip

Now, the treatment of the extension of $S$ to $\overline{\mathrm{dom}A}$ is best motivated by considering how it is usually defined in the literature: $S(t)$ as a mapping $\mathrm{dom}A\to X$ is nonexpansive and thus (uniformly) continuous. The object $S(t)x$ for $x\in\overline{\mathrm{dom}A}$ is then defined by considering that as $x\in\overline{\mathrm{dom}A}$, there exists a sequence $x_n\to x$ with $x_n\in\mathrm{dom}A$. By convergence, the sequence $x_n$ is Cauchy and by continuity of $S(t)$, the sequence $S(t)x_n$ is Cauchy as well and thus converges in a Banach space by completeness. Then $S(t)x$ is identified with the limit of that sequence. This crucial use of the completeness of the space prompts us to work in the context of the formal treatment of complete spaces and $\overline{\mathrm{dom}A}$ as discussed before.

In that vein, we now want to provide an axiom classifying the behavior of $S(t)$ for elements of $\overline{\mathrm{dom}A}$ by essentially stating that for any $x$ and any Cauchy sequence $x_n\to x$ with $x_n\in\mathrm{dom}A$, $S(t)x_n$ converges to $S(t)x$. The quantification over all elements of $\overline{\mathrm{dom}A}$ together with their generating sequences can now be achieved as discussed in Section \ref{sec:closureOfDomain} and in that way, the axiom stating the resulting behavior for $S(t)x$ then takes the form of the following universal axiom\footnote{Note again that the additional $+3$ is included here as the axiom $(\mathcal{C})$ requires this modification in order to have a model as discussed before and the same rate applies to the semigroup-images here as the semigroup is nonexpansive.}
\begin{gather*}
\forall x^{\mathbb{N}\to X}, y^{\mathbb{N}\to X}, t^{\mathbb{N}^\mathbb{N}}\big(y_0\in Ax_0\tag{S2}\\
\rightarrow \forall n^\mathbb{N}\left(\norm{S(\vert t\vert)(C(x\upharpoonright y))-_XS(\vert t\vert)((\widehat{x}\upharpoonright y)_n)}_X\leq_\mathbb{R} 2^{-n+3}\right)\big).
\end{gather*}
Note again that the behavior of $S(0)$ is implicitly characterized by the above axioms through the use of $\vert t\vert$. We write $(S)$ for $(S1)+(S2)$ as well as $H^\omega_p$ for $\widehat{\mathcal{V}}^\omega_p+(\mathcal{C})+(RC)_{\lambda_0}+(S)$ (noting again the additional axioms from Remark \ref{rem:modelRem} and Section \ref{sec:resAtZero}).

\medskip

Now, the above axioms forming the theory $H^\omega_p$ are suitable for formalizing large portions on the theory of nonlinear semigroups as generated by the Crandal-Liggett formula and as a sort of litmus test, we at least provide here sketches of formal proofs in the resulting system of the other main semigroup properties which arise pretty much directly by formalizing the proofs given in \cite{CL1971b}. For that, however, some careful consideration for iterations of the semigroup map are required here. Concretely, to make expressions like $S(t)S(s)x$ meaningful, we have to consider how $S(s)x\in\overline{\mathrm{dom}A}$ is reflected in the system. Based on the representation of $\overline{\mathrm{dom}A}$ chosen above (which also features in how the extension of $S$ is formally defined by means of the axiom (S2)) we thus first have to see how $S(\vert t\vert)C(x)$ with $x_n\in\mathrm{dom}A$ for all $n$ can be expressed as an element of the form $C(u)$ for $u^{\mathbb{N}\to X}$ such that $u_n\in\mathrm{dom}A$ for all $n$. To find such a $u$, note first that the convergence result encoded by (S1) for elements from $\mathrm{dom}A$ extends by means of (S2) to $\overline{\mathrm{dom}A}$ in the following way: provably in $H^\omega_p$, we have
\begin{gather*}
\forall x^{\mathbb{N}\to X}, y^{\mathbb{N}\to X}, t^{\mathbb{N}^\mathbb{N}}, k^\mathbb{N}\exists N^\mathbb{N}\forall n\geq_\mathbb{N} N\bigg(y_0\in Ax_0\land \vert t\vert/n<_\mathbb{R}\lambda_0\\
\rightarrow \norm{S(\vert t\vert)(C(x\upharpoonright y))-_X\left(J^A_{\vert t\vert/n}\right)^n(C(x\upharpoonright y))}_X\leq_\mathbb{R}2^{-k}\bigg)
\end{gather*}
where moreover (although we avoid spelling this out here) the choice functional for $N$ can be explicitly given by closed terms build up from $\omega^S$ (and the other constants). To see the provability of the above statement, let $k,x,y,t$ be arbitrary with $y_0\in Ax_0$. Then using nonexpansivity of the semigroup and the resolvent (see item (4) of the following Lemma \ref{lem:semigroupprop}\footnote{The first four items of this lemma in particular do not rely on this construction as it will only become necessary in the fifth item. Thus, there is no circularity induced by this construction.}), we have
\begin{align*}
&\norm{S(\vert t\vert)(C(x\upharpoonright y))-\left(J^A_{\vert t\vert/n}\right)^n(C(x\upharpoonright y))}\\
&\qquad\leq\norm{S(\vert t\vert)(C(x\upharpoonright y)) -S(\vert t\vert)((\widehat{x}\upharpoonright y)_{(k+5)})}\\
&\qquad\qquad+\norm{S(\vert t\vert)((\widehat{x}\upharpoonright y)_{(k+5)})-\left(J^A_{\vert t\vert/n}\right)^n((\widehat{x}\upharpoonright y)_{(k+5)})}\\
&\qquad\qquad+\norm{\left(J^A_{\vert t\vert/n}\right)^n((\widehat{x}\upharpoonright y)_{(k+5)})-\left(J^A_{\vert t\vert/n}\right)^n(C(x\upharpoonright y))}\\
&\qquad\leq\norm{C(x\upharpoonright y)-(\widehat{x}\upharpoonright y)_{(k+5)}}\\
&\qquad\qquad+\norm{S(\vert t\vert)((\widehat{x}\upharpoonright y)_{(k+5)})-\left(J^A_{\vert t\vert/n}\right)^n((\widehat{x}\upharpoonright y)_{(k+5)})}\\
&\qquad\qquad+\norm{(\widehat{x}\upharpoonright y)_{(k+5)}-C(x\upharpoonright y)}\\
&\qquad\leq 2^{-k-1}+\norm{S(\vert t\vert)((\widehat{x}\upharpoonright y)_{(k+5)})-\left(J^A_{\vert t\vert/n}\right)^n((\widehat{x}\upharpoonright y)_{(k+5)})}\\
&\qquad\leq 2^{-k}
\end{align*}
for any $n$ large enough such that $\vert t\vert/n<\lambda_0$ as well as 
\[
\norm{S(\vert t\vert)((\widehat{x}\upharpoonright y)_{(k+5)})-\left(J^A_{\vert t\vert/n}\right)^n((\widehat{x}\upharpoonright y)_{(k+5)})}\leq 2^{-(k+1)}
\]
which can be achieved via (S1). In that way, writing $N_{t,x,y}$ also for the choice functionals for the quantifier over $N$ in the above statement, we find that $S(\vert t\vert)C(x\upharpoonright y)$ is provably $=_X$-equal to
\[
C\left(\left(\left( J^A_{\vert t\vert/N_{t,x,y}(k)}\right)^{N_{t,x,y}(k)}C(x\upharpoonright y)\right)_k\right).
\]
We write $\overline{S(\vert t\vert)C(x\upharpoonright y)}$ in the following for this expression (where one should note again that the $N$-functionals can be explicitly computed, albeit being somewhat messy). In particular note that 
\[
\left( J^A_{\vert t\vert/N_{t,x,y}(k)}\right)^{N_{t,x,y}(k)}C(x\upharpoonright y)\in\mathrm{dom}A
\]
with the witnessing terms defined in terms of the Yosida approximates (which follows provably from (RC)$_{\lambda_0}$ if we w.l.o.g.\ assume that the functionals $N$, for a given $t$ as a parameter, are large enough such that $\vert t\vert/N_{t,x,y}(k)<\lambda_0$). In that way, $S(\vert t\vert)S(\vert s\vert )C(x)$ can be meaningfully represented by
\[
S(\vert t\vert)\overline{S(\vert s\vert)C(x\upharpoonright y)}=_X\overline{S(\vert t\vert)\overline{S(\vert s\vert)C(x\upharpoonright y)}}.
\]
Note that the system can nevertheless \emph{not} prove that
\[
S(\vert t\vert)S(\vert s\vert)C(x\upharpoonright y)=_XS(\vert t\vert)\overline{S(\vert s\vert)C(x\upharpoonright y)}
\] 
and so the latter is, in some sense, the only way to talk about iterations meaningfully.\\

We now get to the main properties of nonexpansive semigroups:

\begin{lemma}\label{lem:semigroupprop}
The following are provable in $H^\omega_p$:
\begin{enumerate}
\item $\forall x^X,y^X,t^{\mathbb{N}^\mathbb{N}},s^{\mathbb{N}^\mathbb{N}}\left(y\in Ax\rightarrow \norm{S(\vert t\vert)x-_XS(\vert s\vert)x}_X\leq_\mathbb{R} 2\vert \vert t\vert-\vert s\vert\vert\norm{y}_X\right)$.
\item $\begin{cases}\forall x^X,y^X,t^{\mathbb{N}^\mathbb{N}}\big(x\in\mathrm{dom}A\land y\in\mathrm{dom}A\\
\qquad\rightarrow \norm{S(\vert t\vert)x-_XS(\vert t\vert)y}_X\leq_\mathbb{R}\norm{x-_Xy}_X\big).\end{cases}$
\item $\begin{cases}\forall x^{\mathbb{N}\to X},v^{\mathbb{N}\to X},t^{\mathbb{N}^\mathbb{N}},s^{\mathbb{N}^\mathbb{N}}\Big(v_0\in Ax_0\\
\qquad\land \vert \vert t\vert -\vert s\vert\vert\leq_\mathbb{R} 2^{-(k+2)}/\max\{1,\norm{(\widehat{x}\upharpoonleft v)_{k+5}}\}\\
\qquad\qquad\rightarrow \norm{S(\vert t\vert )C(x\upharpoonright v)-_XS(\vert s\vert)C(x\upharpoonright v)}_X\leq_\mathbb{R} 2^{-k}\Big).\end{cases}$
\item $\begin{cases}\forall x^{\mathbb{N}\to X}, v^{\mathbb{N}\to X}, y^{\mathbb{N}\to X}, w^{\mathbb{N}\to X},t^{\mathbb{N}^\mathbb{N}}\Big(v_0\in Ax_0\land w_0\in Ay_0\\
\qquad\rightarrow \norm{S(\vert t\vert)(C(x\upharpoonright v))-_XS(\vert t\vert)(C(y\upharpoonright w))}_X\\
\qquad\qquad\leq_\mathbb{R}\norm{C(x\upharpoonright v)-_XC(y\upharpoonright w)}_X\Big).\end{cases}$
\item $\begin{cases}\forall x^{\mathbb{N}\to X},v^{\mathbb{N}\to X},t^{\mathbb{N}^\mathbb{N}},s^{\mathbb{N}^\mathbb{N}}\Big(v_0\in Ax_0\\
\qquad\rightarrow S(\vert t\vert +\vert s\vert)(C(x\upharpoonright v))=_XS(\vert t\vert)\overline{S(\vert s\vert)(C(x\upharpoonright v))}\Big).\end{cases}$
\end{enumerate}
\end{lemma}
\begin{proof}
\begin{enumerate}
\item At first, note that provably in $H^\omega_p$, we have
\[
\begin{cases}\forall x^X,y^X,\mu^{\mathbb{N}^\mathbb{N}},\lambda^{\mathbb{N}^\mathbb{N}},n^\mathbb{N},m^\mathbb{N}\bigg(\lambda_0>\vert\lambda\vert\geq\vert\mu\vert\land n\geq m\geq 1\land y\in Ax\\
\qquad\rightarrow \norm{(J^A_{\vert\mu\vert})^nx-(J^A_{\vert\lambda\vert})^mx}\leq\bigg( \left((n\vert\mu\vert-m\vert\lambda\vert)^2+n\vert\mu\vert(\vert\lambda\vert-\vert\mu\vert)\right)^{1/2}\\
\qquad\qquad+\left(m\vert\lambda\vert(\vert\lambda\vert-\vert\mu\vert)+(m\vert\lambda\vert-n\vert\mu\vert)^2\right)^{1/2}\bigg)\norm{y}\bigg)\end{cases}
\]
which can be shown by formalizing the proof given in \cite{CL1971b} (note for this Lemma \ref{lem:basicProp}\footnote{Note that the theorems in Lemma \ref{lem:basicProp} remain valid for the system $H^\omega_p$ if the indices of the resolvent are restricted to be $<\lambda_0$.}). Instantiating this with $m=n$, $\mu=\vert t\vert/n$ and $\lambda=\vert s\vert/n$ for $t,s$ of type ${\mathbb{N}^\mathbb{N}}$, where w.l.o.g.\ $\vert s\vert\geq \vert t\vert$,  and where $n$ is large enough that $\vert t\vert/n,\vert s\vert/n<\lambda_0$, we obtain
\begin{align*}
\norm{(J^A_{\vert t\vert/n})^nx-(J^A_{\vert s\vert/n})^nx}\leq&\bigg( \left((\vert t\vert-\vert s\vert)^2+\vert t\vert(\vert s\vert/n-\vert t\vert/n)\right)^{1/2}\\
&+\left(\vert s\vert(\vert s\vert/n-\vert t\vert/n)+(\vert s\vert-\vert t\vert)^2\right)^{1/2}\bigg)\norm{y}
\end{align*}
for any $x,y$ with $y\in Ax$. Let $k$ be arbitrary. Using the axioms $(S)$, we get
\begin{align*}
\norm{S(\vert t\vert)x-S(\vert s\vert)x}&\leq \norm{S(\vert t\vert)x-(J^A_{\vert t\vert/n})^nx}+\norm{(J^A_{\vert t\vert/n})^nx-(J^A_{\vert s\vert/n})^nx}\\
&\qquad+\norm{S(\vert s\vert)x-(J^A_{\vert s\vert/n})^nx}\\
&\leq 2\cdot 2^{-k}+\norm{(J^A_{\vert t\vert/n})^nx-(J^A_{\vert s\vert/n})^nx}\\
&\leq 2\cdot 2^{-k}+\Big( \left((\vert t\vert-\vert s\vert)^2+\vert t\vert(\vert s\vert/n-\vert t\vert/n)\right)^{1/2}\\
&\qquad+\left(\vert s\vert(\vert s\vert/n-\vert t\vert/n)+(\vert s\vert-\vert t\vert)^2\right)^{1/2}\Big)\norm{y}
\end{align*}
for any $n$ additionally satisfying $n\geq\omega^S(k,b,T)$ with $b>\norm{x},\norm{v}$ and $T>\vert t\vert,\vert s\vert$. This implies
\[
\norm{S(\vert t\vert)x-S(\vert s\vert)x}\leq 2\cdot 2^{-k}+2\vert \vert t\vert-\vert s\vert\vert\norm{y}
\]
and the claim follows as $k$ was arbitrary.
\item By (essentially) Lemma \ref{lem:basicProp}, we have provably that
\[
\norm{J^A_{\vert\lambda\vert} x-J^A_{\vert\lambda\vert} y}\leq\norm{x-y}
\]
for any $\lambda_0>\lambda$ of type ${\mathbb{N}^\mathbb{N}}$ and any $x,y$ of type $X$. By induction, we get
\[
\norm{(J^A_{\vert t\vert/n})^n x-(J^A_{\vert t\vert/n})^n y}\leq\norm{x-y}
\]
for any $t$ of type ${\mathbb{N}^\mathbb{N}}$, any $x,y$ of type $X$ and any $n$ large enough such that $\vert t\vert/n<\lambda_0$. Now, let $k$ be arbitrary. Then we get
\begin{align*}
\norm{S(\vert t\vert)x-S(\vert t\vert)y}&\leq\norm{S(\vert t\vert)x-(J^A_{\vert t\vert/n})^nx}+\norm{(J^A_{\vert t\vert/n})^n x-(J^A_{\vert t\vert/n})^n y}\\
&\qquad+\norm{S(\vert t\vert)y-(J^A_{\vert t\vert/n})^ny}\\
&\leq 2\cdot 2^{-k}+\norm{(J^A_{\vert t\vert/n})^n x-(J^A_{\vert t\vert/n})^n y}\\
&\leq 2\cdot 2^{-k}+\norm{x-y}
\end{align*}
for any $n$ additionally satisfying $n\geq\omega^S(k,b,T)$ with $b>\norm{x},\norm{y},\norm{v},\norm{w}$ with $v\in Ax$ and $w\in Ay$ as well as $T>\vert t\vert$ using $(S)$. As $k$ was arbitrary, we get the claim.
\item Using item (1) and axiom (S2), if $v_0\in Ax_0$, we have
\begin{align*}
&\norm{S(\vert t\vert )C(x\upharpoonright v)-S(\vert s\vert) C(x\upharpoonright v)}\\
&\qquad\leq \norm{S(\vert t\vert)C(x\upharpoonright v)-S(\vert t\vert)(\widehat{x}\upharpoonright v)_n}\\
&\qquad\qquad+\norm{S(\vert t\vert)(\widehat{x}\upharpoonright v)_n-S(\vert s\vert)(\widehat{x}\upharpoonright v)_n}\\
&\qquad\qquad+\norm{S(\vert s\vert)(\widehat{x}\upharpoonright v)_n-S(\vert s\vert) C(x\upharpoonright v)}\\
&\qquad\leq 2\cdot 2^{-n+3}+\norm{S(\vert t\vert)(\widehat{x}\upharpoonright v)_n-S(\vert s\vert)(\widehat{x}\upharpoonright v)_n}\\
&\qquad\leq 2\cdot 2^{-n+3}+2\vert\vert t\vert-\vert s\vert\vert\norm{(\widehat{x}\upharpoonleft v)_n}.
\end{align*}
Choosing $n=k+5$, we get the claim for $\vert\vert t\vert-\vert s\vert\vert\leq 2^{-(k+2)}/\max\{1,\norm{(\widehat{x}\upharpoonleft v)_{k+5}}\}$.
\item Using item (2), axiom (S2) as well as ($\mathcal{C}$), if $v_0\in Ax_0$ and $w_0\in Ay_0$, we have
\begin{align*}
&\norm{S(\vert t\vert)(C(x\upharpoonright v))-S(\vert t\vert)(C(y\upharpoonright w))}\\
&\qquad\leq\norm{S(\vert t\vert)(C(x\upharpoonright v))-S(\vert t\vert)((\widehat x\upharpoonright v)_k)}\\
&\qquad\qquad+\norm{S(\vert t\vert)((\widehat x\upharpoonright v)_k)-S(\vert t\vert)((\widehat y\upharpoonright w)_k)}\\
&\qquad\qquad+\norm{S(\vert t\vert)(C(y\upharpoonright w))-S(\vert t\vert)((\widehat y\upharpoonright w)_k)}\\
&\qquad\leq 2\cdot 2^{-k+3}+\norm{S(\vert t\vert)((\widehat x\upharpoonright v)_k)-S(\vert t\vert)((\widehat y\upharpoonright w)_k}\\
&\qquad\leq 2\cdot 2^{-k+3}+\norm{(\widehat x\upharpoonright v)_k-(\widehat y\upharpoonright w)_k}\\
&\qquad\leq 2\cdot 2^{-k+3}+\norm{(\widehat x\upharpoonright v)_k-C(x\upharpoonright v)}\\
&\qquad\qquad+\norm{C(x\upharpoonright v)-C(y\upharpoonright w)}\\
&\qquad\qquad+\norm{C(y\upharpoonright w)-(\widehat y\upharpoonright w)_k}\\
&\qquad\leq 2\cdot 2^{-k+3}+2\cdot 2^{-k+3}+\norm{C(x\upharpoonright v)-C(y\upharpoonright w)}.
\end{align*}
As this holds for arbitrary $k$, we get the claim.
\item Let $x\in\mathrm{dom}A$. Using the previously introduced notation of $\overline{\,\cdot\,}$, we write
\[
\begin{cases}
[S(\vert t\vert)]^1x=\overline{S(\vert t\vert)x},\\
[S(\vert t\vert)]^{m+1}x=\overline{S(\vert t\vert)\left([S(\vert t\vert)]^m x\right)}.
\end{cases}
\]
Note that provably 
\[
[S(\vert t\vert)]^{m+1}x=S(\vert t\vert)\left([S(\vert t\vert)]^mx\right)
\]
which follows as in the discussion previous to this lemma. We now show by induction on $m$ that provably
\[
\forall k\exists N_{m}\forall n\geq N_{m}\left(\vert t\vert/n<\lambda_0\rightarrow \norm{[S(\vert t\vert)]^mx-\left(\left(J^A_{\vert t\vert/n}\right)^m\right)^{n}x}\leq2^{-k}\right).
\]
The induction base follows from (S1) as was already discussed above. For the induction step, let $N_m(k)$ be the choice function of the above statement. Then for arbitrary $k$, we get (using extensionality, see Remark \ref{rem:extSemi}, and nonexpansivity of $S(\vert t\vert)$ on the closure of the domain) that
\begin{align*}
&\norm{[S(\vert t\vert)]^{m+1}x- \left(\left(J^A_{\vert t\vert/n}\right)^{m+1}\right)^nx}\\
&\qquad= \norm{S(\vert t\vert)[S(\vert t\vert)]^mx-\left(J^A_{\vert t\vert/n}\right)^n\left(\left(J^A_{\vert t\vert/n}\right)^m\right)^nx}\\
&\qquad\leq \norm{S(\vert t\vert)[S(\vert t\vert)]^mx-\left(J^A_{\vert t\vert/n}\right)^n[S(\vert t\vert)]^mx}\\
&\qquad\qquad+\norm{\left(J^A_{\vert t\vert/n}\right)^n[S(\vert t\vert)]^mx-\left(J^A_{\vert t\vert/n}\right)^n\left(\left(J^A_{\vert t\vert/n}\right)^m\right)^nx}\\
&\qquad\leq \norm{S(\vert t\vert)[S(\vert t\vert)]^mx-\left(J^A_{\vert t\vert/n}\right)^n[S(\vert t\vert)]^mx}\\
&\qquad\qquad+\norm{[S(\vert t\vert)]^mx-\left(\left(J^A_{\vert t\vert/n}\right)^m\right)^nx}\\
&\qquad\leq 2^{-k}
\end{align*}
for all $n$ such that $\vert t\vert /n<\lambda_0$, $n\geq N_m(k+1)$ and such that $n$ is large enough for
\[ 
\norm{S(\vert t\vert)[S(\vert t\vert)]^mx-\left(J^A_{\vert t\vert/n}\right)^n[S(\vert t\vert)]^mx}\leq 2^{-(k+1)}
\]
which can be constructed as in the discussion previous to this lemma. Therefore, given $k$, we in particular get
\begin{align*}
&\norm{[S(\vert t\vert)]^mx-S(m\vert t\vert)x}\\
&\qquad\leq\norm{[S(\vert t\vert)]^mx-\left(\left(J^A_{\vert t\vert/n}\right)^m\right)^nx}+\norm{\left(\left(J^A_{\vert t\vert/n}\right)^m\right)^nx-S(m\vert t\vert)x}\\
&\qquad= \norm{[S(\vert t\vert)]^mx-\left(\left(J^A_{\vert t\vert/n}\right)^m\right)^nx}+\norm{\left(J^A_{m\vert t\vert/mn}\right)^{mn}x-S(m\vert t\vert)x}\\
&\qquad\leq 2^{-(k+1)}+2^{-(k+1)}\\
&\qquad\leq 2^{-k}
\end{align*}
for any $n$ such that $m\vert t\vert/n<\lambda_0$ and
\[
n\geq\max\{N_{m}(k+1),\omega^S(k+1,b,mT)\}
\]
for $b>\norm{x},\norm{v}$ and $T>\vert t\vert$ using (S1) and the previous result. As $k$ was arbitrary, we get $[S(\vert t\vert)]^mx=S(m\vert t\vert)x$. Using this, we provably get
\begin{align*}
S\left(\frac{l}{k}+\frac{r}{s}\right)x&=S\left(\frac{ls+rk}{ks}\right)x\\
&=\left[S\left(\frac{1}{ks}\right)\right]^{ls+rk}x\\
&=\left[S\left(\frac{1}{ks}\right)\right]^{ls}\left[S\left(\frac{1}{ks}\right)\right]^{rk}x\\
&=\left[S\left(\frac{1}{ks}\right)\right]^{ls}\overline{S\left(\frac{rk}{ks}\right)x}\\
&=S\left(\frac{ls}{ks}\right)\overline{S\left(\frac{rk}{ks}\right)x}\\
&=S\left(\frac{l}{k}\right)\overline{S\left(\frac{r}{s}\right)x}
\end{align*}
where we have used the above items for extensionality of $S$ (see again Remark \ref{rem:extSemi}). A continuity argument using item (3) now yields the claim for arbitrary reals $\vert t\vert$ and $\vert s\vert$. Further, the claim extends to the closure of the domain via another usual continuity argument. Both we do not spell out here.
\end{enumerate}
\end{proof}

\begin{remark}\label{rem:extSemi}
The constant $S(t)x$ is provably extensional in $x\in\overline{\mathrm{dom}A}$ for any $t\geq 0$ by (4) as well as in $t\geq 0$ for any $x\in\overline{\mathrm{dom}A}$ by (3).
\end{remark}

\begin{remark}\label{rem:unifEqui}
Note that by the proof of the above item (3), we have that if the operator $A$ is majorizable in the sense of \cite{Pis2022}, i.e.\ if there exists a function $A^*:\mathbb{N}\to\mathbb{N}$ such that
\[
\forall b\in\mathbb{N}\forall x\in\mathrm{dom}A\cap\overline{B}_b(0)\exists y\in X\left(\norm{y}\leq A^*b\land y\in Ax\right),
\]
then the semigroup $\mathcal{S}$ generated by $A$ through the Crandall-Liggett formula is uniformly equicontinuous in the sense of \cite{KKA2016}, i.e.\ there exists a function $\omega:\mathbb{N}\times\mathbb{N}\times\mathbb{N}\to\mathbb{N}$ such that 
\begin{gather*}
\forall b\in\mathbb{N}\forall q\in \overline{\mathrm{dom}A}\cap B_b(0)\forall m\in\mathbb{N}\forall K\in\mathbb{N}\forall t,t'\in [0,K]\\
\left( \vert t-t'\vert< 2^{-\omega_{K,b}(m)}\to \norm{S(t)q-S(t')q}< 2^{-m}\right).
\end{gather*}
Concretely, assuming w.l.o.g.\ that $A^*$ is nondecreasing, this so-called \emph{modulus of uniform equicontinuity} for $\mathcal{S}$ can be given by
\[
\omega_{K,b}(m)=(m+2)\max\{1,A^*(b+1)\}.
\]
Note in particular that this modulus is independent of the parameter $K$.
\end{remark}

We now come to the main theoretical result of this paper which comprises a proof-theoretic bound extraction theorem for the system $H^\omega_p$ akin to the usual metatheorems of proof mining. As discussed before, the proof of this metatheorem follows the general outline of most other proofs in the literature and, since the proof is very much standard in this way, we hence omit most of the details and in the following mainly just sketch the majorizability of the new constant $S$.

\begin{theorem}\label{thm:metatheoremSC}
Let $\tau$ be admissible, $\delta$ be of degree $1$ and $s$ be a closed term of $H^\omega_p$ of type $\delta\to \sigma$ for admissible $\sigma$. Let $B_\forall(x,y,z,u)$/$C_\exists(x,y,z,v)$ be $\forall$-/$\exists$-formulas of $H^\omega_p$ with only $x,y,z,u$/$x,y,z,v$ free. Let $\Delta$ be a set of formulas of the form $\forall\underline{a}^{\underline{\alpha}}\exists\underline{b}\preceq_{\underline{\beta}}\underline{r}\underline{a}\forall\underline{c}^{\underline{\zeta}}F_{qf}(\underline{a},\underline{b},\underline{c})$ where $F_{qf}$ is quantifier-free, the types in $\underline{\alpha}$, $\underline{\beta}$ and $\underline{\zeta}$ are admissible and $\underline{r}$ is a tuple of closed terms of appropriate type. If
\[
H^\omega_p+\Delta\vdash\forall x^\delta\forall y\preceq_\sigma s(x)\forall z^\tau\left(\forall u^\mathbb{N} B_\forall(x,y,z,u)\rightarrow\exists v^\mathbb{N} C_\exists(x,y,z,v)\right),
\]
then one can extract a partial functional $\Phi:\mathcal{S}_{\delta}\times \mathcal{S}_{\widehat{\tau}}\times\mathbb{N}\times\mathbb{N}^\mathbb{N}\rightharpoonup\mathbb{N}$ which is total and (bar-recursively) computable on $\mathcal{M}_\delta\times\mathcal{M}_{\widehat{\tau}}\times\mathbb{N}\times\mathbb{N}^\mathbb{N}$ and such that for all $x\in \mathcal{S}_\delta$, $z\in \mathcal{S}_\tau$, $z^*\in \mathcal{S}_{\widehat{\tau}}$ and all $n\in\mathbb{N}$ and $\omega\in\mathbb{N}^\mathbb{N}$, if $z^*\gtrsim z$ and $\omega\gtrsim\omega^S$ as well as $n\geq_\mathbb{R} m_{\widetilde{\gamma}},\vert\widetilde{\gamma}\vert,\norm{c_X}_X$, $\norm{d_X}_X$, $\vert\lambda_0\vert$, $m_{\lambda_0}$, $m'_{\widetilde{\gamma}}$, then
\begin{gather*}
\mathcal{S}^{\omega,X}\models\forall y\preceq_\sigma s(x)\big(\forall u\leq_\mathbb{N}\Phi(x,z^*,n,\omega) B_\forall(x,y,z,u)\\
\rightarrow\exists v\leq_\mathbb{N}\Phi(x,z^*,n,\omega)C_\exists(x,y,z,v)\big)
\end{gather*}
holds for $\mathcal{S}^{\omega,X}$ whenever $\mathcal{S}^{\omega,X}\models\Delta$ where $\mathcal{S}^{\omega,X}$ is defined via any (nontrivial) Banach space $(X,\norm{\cdot})$ with 
\begin{enumerate}
\item $\chi_A$ interpreted by the characteristic function of an accretive operator $A$ satisfying the range condition $\overline{\mathrm{dom}A}\subseteq \bigcap_{\lambda_0>\gamma>0}\mathrm{dom}J^A_\gamma$,
\item $J^{\chi_A}$ interpreted by the corresponding resolvents $J^A_\gamma x$ for $\lambda_0>\gamma\geq 0$  and $x\in\mathrm{dom}J^A_\gamma$, and by $0$ otherwise,
\item $j$ interpreted as discussed in Section \ref{sec:jinterp}, 
\item $S$ interpreted by the semigroup generated by $A$ via the Crandall-Liggett formula on $[0,\infty)\times \overline{\mathrm{dom}A}$, and $0$ otherwise,
\item $d_X$, $c_X$ interpreted by a pair $(c,d)\in A$ witnessing $A\neq\emptyset$,
\item $\omega^S$ interpreted by a rate of convergence for the limit generating the semigroup on $\mathrm{dom}A$,
\end{enumerate}
and with the other constants naturally interpreted so that the respective axioms are satisfied.

Further: If $\widehat{\tau}$ is of degree $1$, then $\Phi$ is a total computable functional. If the claim is proved without $\mathrm{DC}$, then $\tau$ may be arbitrary and $\Phi$ will be a total functional on $\mathcal{S}_\delta\times \mathcal{S}_{\widehat{\tau}}\times\mathbb{N}\times \mathbb{N}^\mathbb{N}$ which is primitive recursive in the sense of G\"odel's T. In that latter case, also plain majorization can be used instead of strong majorization.
\end{theorem}
\begin{proof}
The proof given in \cite{Pis2022} immediately extends to this system, noticing the additional considerations on the model of majorizable functionals discussed in the context of $j$ as well as Remark \ref{rem:modelRem}. In particular, note also that all axioms added to $H^\omega_p$ are purely universal and that the new constants other than $S$ can be majorized as discussed throughout the previous sections. For the last constant $S$, we can argue for the majorizability as follows: In the context of the axiom (V)$'$, stating that $\mathrm{dom}A$ is not empty using the constants $c_X$ and $d_X$, majorization of the constant $S$ on $t\geq 0$ and $x\in\mathrm{dom}A$ follows rather immediately. It is straightforward to obtain that
\[
\widehat{\mathcal{V}}^\omega_p\vdash\forall x^X,\lambda^{\mathbb{N}^\mathbb{N}},n^\mathbb{N}\left( x\in\mathrm{dom}J^A_{\vert\lambda\vert}\rightarrow \norm{(J^A_{\vert\lambda\vert})^nx-_Xx}_X\leq_\mathbb{R} n\norm{J^A_{\vert\lambda\vert} x-_Xx}_X\right).
\]
Therefore, we have for $x\in\mathrm{dom}A$ with $v\in Ax$ and $b>\norm{x},\norm{v}$ and for $t\geq 0$ with $T>t$ that for $n\geq (\omega^S(0,b,T)+\ceil*{T/\lambda_0})$:\footnote{We can choose e.g.\ $n=\omega^S(0,b,T)+[T/\lambda_0](0)+1$ which can be represented through a closed term.}
\begin{align*}
\norm{S(t)x}&\leq\norm{S(t)x-(J^A_{t/n})^{n}x}+\norm{(J^A_{t/n})^{n}x}\\
&\leq 1+\norm{(J^A_{t/n})^{n}x-(J^A_{t/n})^{n}c_X}+\norm{(J^A_{t/n})^{n}c_X}\\
&\leq 1+\norm{x-c_X}+\norm{c_X}+n\norm{J^A_{t/n}c_X-c_X}\\
&\leq 1+\norm{x}+2\norm{c_X}+T\norm{d_X}
\end{align*}
which follows from the axioms $(S)$ and $(RC)_{\lambda_0}$. This extends to $\overline{\mathrm{dom}A}$ as follows: For $x\in\overline{\mathrm{dom}A}$ and $x_n\to x$ with rate of convergence $2^{-n}$ and where $x_n\in\mathrm{dom}A$, we have $\norm{x_0-x}\leq 1$ and $\norm{S(t)x-S(t)x_0}\leq 1$ and thus
\begin{align*}
\norm{S(t)x}&\leq 1+\norm{S(t)x_0}\\
&\leq 2+\norm{x_0}+2\norm{c_X}+T\norm{d_X}
\\
&\leq 3+\norm{x}+2\norm{c_X}+T\norm{d_X}.
\end{align*}
\end{proof}
Also Theorem \ref{thm:metatheoremI} extends to an intuitionistic version $H^\omega_{i,p}$ of the system $H^\omega_p$ in that fashion. Concretely, let $H^\omega_{i,p}$ be defined as the extension/modification of $\mathcal{V}^\omega_{i,p}$ with the same constants and axioms as were added to/modified in $\mathcal{V}^\omega_p$ to form $H^\omega_p$. Then the following semi-constructive bound extraction theorem holds:
\begin{theorem}\label{thm:metatheoremSCI}
Let $\delta$ be of degree $1$ and $\sigma,\tau$ be arbitrary, $s$ be a closed term of suitable type. Let $\Gamma_\neg$ be a set of sentences of the form $\forall\underline{a}^{\underline{\alpha}}(E(\underline{a})\rightarrow \exists\underline{b}\preceq_{\underline{\beta}}\underline{t}\underline{a}\neg F(\underline{a},\underline{b}))$ with $\underline{\alpha},\underline{\beta}$ and $E,F$ arbitrary types and formulas respectively and where $\underline{t}$ is a tuple of closed terms. Let $B(x,y,z)$/$C(x,y,z,u)$ be arbitrary formulas of $H^\omega_{i,p}$ with only $x,y,z$/$x,y,z,u$ free. If
\[
H^\omega_{i,p}+\Gamma_\neg\vdash\forall x^\delta\,\forall y\preceq_\sigma s(x)\,\forall z^\tau\,(\neg B(x,y,z)\rightarrow\exists u^\mathbb{N} C(x,y,z,u)), 
\]
one can extract a $\Phi:\mathcal{S}_\delta\times\mathcal{S}_{\widehat{\tau}}\times\mathbb{N}\times\mathbb{N}^\mathbb{N}\to\mathbb{N}$ which is primitive recursive in the sense of G\"odel's T such that for any $x\in\mathcal{S}_\delta$, any $y\in\mathcal{S}_\sigma$ with $y\preceq_\sigma s(x)$, any $z\in\mathcal{S}_{\tau}$ and $z^*\in\mathcal{S}_{\widehat{\tau}}$ with $z^*\gtrsim z$ and any $n\in\mathbb{N}$ and $\omega\in\mathbb{N}^\mathbb{N}$ with $\omega\gtrsim\omega^S$ as well as $n\geq_\mathbb{R}m_{\widetilde{\gamma}},\vert\widetilde{\gamma}\vert,\norm{c_X}_X, \norm{d_X}_X, \vert\lambda_0\vert, m_{\lambda_0}, m'_{\widetilde{\gamma}}$, we have that
\[
\mathcal{S}^{\omega,X}\models\exists u\leq_\mathbb{N}\Phi(x,z^*,n,\omega)\,(\neg B(x,y,z)\rightarrow C(x,y,z,u))
\]
holds for $\mathcal{S}^{\omega,X}$ whenever $\mathcal{S}^{\omega,X}\models\Gamma_\neg$ where $\mathcal{S}^{\omega,X}$ is defined via any (nontrivial) Banach space $(X,\norm{\cdot})$ with the constants interpreted as in Theorem \ref{thm:metatheoremSC}.
\end{theorem}
Using the previous arguments regarding the majorizability of the (new) constants, the proof is a straightforward adaptation of the proof for the central semi-constructive metatheorem for the system $\mathcal{A}^\omega_i[X,\norm{\cdot}]+\mathrm{IP}_\neg+\mathrm{CA}_\neg$ given in \cite{GeK2006} and we thus omit any further details.

\section{Two applications}

The current section now employs the formal systems and the bound extraction theorems introduced above to provide quantitative information on the asymptotic behavior of these semigroups in the form of two case studies. Concretely, motivated by the results of Pazy \cite{Paz1971} for iterations of nonexpansive mappings, the following two asymptotic results on the resolvents and the semigroups generated by the underlying operator were established:\footnote{To simplify the notation in the following, we drop the superscript of the operator $A$ from the resolvent.}

\begin{theorem}[Plant \cite{Pla1981}]\label{thm:Plant}
Let $X$ be uniformly convex, $A$ be an accretive operator that satisfies the range condition $(RC)_{\lambda_0}$ and let $x\in\mathrm{dom}A$. Then
\[
\lim_{\lambda_0> t\to 0^+}\frac{\norm{J_t x-S(t)x}}{t}= 0.
\]
\end{theorem}

\begin{theorem}[Reich \cite{Rei1981}]\label{thm:Reich}
Let $X$ be uniformly convex, $A$ be an accretive operator that satisfies the range condition $(RC)$ and let $x\in\mathrm{dom}A$. Then
\[
\lim_{t\to\infty}\frac{\norm{J_t x-S(t)x}}{t}= 0.
\]
\end{theorem}

A usual application of negative translation and monotone functional interpretation as used by the metatheorems suggest the extractability of ``metastability-like'' rates here (provided that the proof formalizes in the underlying systems). However, as we will see, classical logic features in these proofs only in two ways: at first, it features in some of the basic underlying convergence results in which case the limits are decreasing and a rate of convergence can thus nevertheless be obtained using the metatheorem from Theorem \ref{thm:metatheoremSC}. For both results, the proof then proceeds via a case distinction on real numbers between $=0$ and $>0$. In both results, the proofs for the ``$=0$''-cases are trivial and rates of convergence can be immediately extracted. While the proofs for the ``$>0$''-cases are nontrivial, they are nevertheless essentially constructive which allows, through the use of the semi-constructive metatheorem of Theorem \ref{thm:metatheoremSCI}, for the extraction of full rates of convergence for both limits exhibited above, under the appropriate quantitative reformulations of the ``$>0$''-assumption, respectively. So, a rate of convergence can be obtained in either case, for both results. Only in the combination of these rates to a rate for the full result, the issues from the use of classical logic could feature but as we will see, in both cases the rates can be smoothed to be combined into a full rate of convergence for the whole result.

\medskip

The next two subsections now present the extractions of the quantitative results and in that context do not explicitly focus on the logical particularities of the extraction which will only be discussed in the last subsection. In that way, to present these results in a way more amenable to the usual literature of the theory of semigroups, we also move to using $\varepsilon$'s for the errors instead of $2^{-k}$ or similar constructions using natural numbers.

\medskip

In general, the main assumption featuring in both results is the uniform convexity of the underlying space which can be treated, as extensively discussed in the proof mining literature starting from the earliest works on the treatment of abstract spaces (see \cite{Koh2005}), by a so-called modulus of uniform convexity:

\begin{definition}
A modulus of uniform convexity for a space $X$ is a mapping $\eta:(0,2]\to (0,1]$ such that 
\[
\forall\varepsilon\in (0,2]\forall x,y\in X\left( \norm{x},\norm{y}\leq 1\land \norm{x-y}\geq\varepsilon\rightarrow\norm{\frac{x+y}{2}}\leq 1-\eta(\varepsilon)\right).
\]
\end{definition}

Of course, the rates in general will then depend on such a modulus. It should be further noted that this modulus is conceptually related to the common analytic notion of a modulus of convexity $\delta:[0,2]\to [0,1]$ (implicit already used in e.g.\ \cite{Cla1936}) defined as
\[
\delta(\varepsilon):=\inf\left\{1-\norm{x+y}/2\mid \norm{x}=\norm{y}=1,\norm{x-y}=\varepsilon\right\}.
\]
In fact, as is well-known, uniformly convex spaces are characterized by the property that $\delta(\varepsilon)>0$ whenever $\varepsilon>0$ and the modulus of uniform convexity $\eta$ effectively provides a witness for this inequality in the form of a lower bound, i.e.\ that $\delta(\varepsilon)\geq\eta(\varepsilon)>0$ for $\varepsilon\in (0,2]$.

The proofs of the results of both Plant and Reich make an essential use of $\delta$ but closer inspection reveals that they only rely on a lower bound on $\delta(\varepsilon)$ greater than $0$ which therefore can be substituted by the modulus of uniform convexity $\eta$. Note that $\eta$ can be assumed to be nondecreasing which we will do w.l.o.g.\ in the following. In that case, one in particular has that $\eta(\varepsilon)<\eta(\delta)$ implies $\varepsilon\leq\delta$.

\subsection{An analysis of Plant's result}

In this subsection, if not said otherwise, let $X$ be a fixed Banach space, $A$ be a fixed accretive operator that satisfies the range condition $(RC)_{\lambda_0}$ and let $S$ be the semigroup on $\overline{\mathrm{dom}A}$ generated by $A$ using the Crandall-Liggett formula. The proof of Plant's result now proceeds by establishing that the sequence 
\[
\frac{x-J_tx}{t},\quad (t\to 0^+)
\]
is Cauchy and that we have the limit
\[
\lim_{t,s/t\to 0^+}\norm{\frac{x-J_tx}{t}-\frac{x-S(s)x}{s}}= 0.
\]
Both results rely crucially on the existence and equality of the limits
\[
\lim_{t\to 0^+}\frac{\norm{x-J_tx}}{t}\text{ and }\lim_{t\to 0^+}\frac{\norm{x-S(t)x}}{t}.
\]
The first sequence is nondecreasing for $t\to 0^+$ (see e.g.\ \cite{CP1972}) and bounded by $\norm{v}$ for $v\in Ax$ witnessing $x\in\mathrm{dom}A$ (see e.g.\ \cite{Bar1976}). Following \cite{Cra1973}, we denote the first limit by $\vert Ax\vert$ which naturally satisfies $\vert Ax\vert\leq\norm{v}$. The second limit was shown to coincide with $\vert Ax\vert$ in \cite{Cra1973}.

\medskip

Now, the proof given in \cite{Pla1981} crucially relies on the use of the limit operator $\vert Ax\vert$ and some elementary properties thereof. For the following, we denote the expression $(x-J_tx)/t$ (which is just the Yosida approximate) by $A_tx$, in contrast to Plant's notation.

\medskip

As discussed in \cite{Pis2022}, one central theoretical obstacle in treating accretive and monotone operators is the use of extensionality for such operators. While this will be discussed in more detail in the later logical remarks, we also find here that the main convergence principle
\[
\norm{A_tx}\to\vert Ax\vert\text{ for }t\to 0^+\text{ with }x\in\mathrm{dom}A,
\]
on which the proof of Plant relies, can be recognized as a particular version of such a kind of extensionality statement, namely it can be shown that it is provably equivalent to the lower semicontinuity on $\mathrm{dom}A$ of the operator $\vert A\cdot\vert$ associated with $A$ (see Proposition \ref{pro:ROCYosida} later on).

As in the case of the functional $\langle\cdot,\cdot\rangle_s$, the logical methodology based on the monotone Dialectica interpretation now implies the following quantitative version of this statement: under this interpretation, the statement is upgraded to the existence of a ``modulus of uniform lower semicontinuity'' $\varphi:\mathbb{R}_{>0}\times \mathbb{N}\to\mathbb{R}_{>0}$, i.e.\ 
\begin{gather*}
\forall b\in\mathbb{N},\varepsilon\in\mathbb{R}_{>0},(x,u),(y,v)\in A\\
\left(\norm{x},\norm{u},\norm{y},\norm{v}\leq b\land\norm{x-y}\leq \varphi(\varepsilon,b)\rightarrow \vert Ax\vert-\vert Ay\vert\leq\varepsilon\right),
\end{gather*} 
which, as discussed already in the context of $\langle\cdot,\cdot\rangle_s$, is essentially a modulus of uniform continuity.

Based on the above mentioned equivalence, this modulus can then be used to derive a rate of convergence for the Yosida approximates towards $\vert Ax\vert$.

\begin{lemma}\label{lem:resolventROC}
Let $\varphi$ be a modulus of uniform continuity for $\vert A\cdot\vert$ and let $n$ satisfy $n\geq\norm{c},\norm{d},\lambda_0,\widetilde{\gamma}$ for $(c,d)\in A$ and $0<\widetilde{\gamma}<\lambda_0$. Then for $x\in\mathrm{dom}A$ with $v\in Ax$ and $b\in\mathbb{N}^*$\footnote{Throughout, we write $\mathbb{N}^*$ for the natural numbers without $0$.} with $b\geq\norm{x},\norm{v}$, we have
\[
\forall \varepsilon>0\forall t\in \left(0, \varphi_1(\varepsilon,b,n,\varphi)\right]\left(\vert Ax\vert-\frac{\norm{x-J^A_tx}}{t}\leq \varepsilon\right)
\]
where
\[
\varphi_1(\varepsilon,b,n,\varphi):=\min\{\varphi(\varepsilon,b+2n+3n^2)/b,\lambda_0/2\}.
\]
\end{lemma}
\begin{proof}
Let $\varepsilon$ be given and let $t\leq\varphi_1(\varepsilon,b,n,\varphi)$. We at first have $\norm{A_tx}\leq\norm{v}\leq b$ as well as
\begin{align*}
\norm{J_tx}&\leq \norm{x}+2\norm{c}+\left(2\widetilde{\gamma}+t\right)\norm{d}\\
&\leq \norm{x}+2\norm{c}+\left(2\widetilde{\gamma}+\lambda_0\right)\norm{d}\\
&\leq \norm{x}+2n+\left(2n+n\right)n\\
&\leq \norm{x}+2n+3n^2
\end{align*}
using Lemma \ref{lem:basicProp} (as $t<\lambda_0$). Now as $A_tx\in AJ_tx$, we have $\vert AJ_tx\vert\leq\norm{A_tx}$ and thus
\[
\vert Ax\vert -\norm{A_tx}\leq\vert Ax\vert -\vert AJ_tx\vert.
\]
Now, we get
\[
\norm{x-J_tx}\leq t\norm{v} \leq \varphi_1(\varepsilon,b,n,\varphi)b\leq \varphi(\varepsilon,b+2n+3n^2)
\]
and thus, as $v\in Ax$ and $A_t x\in AJ_tx$ with $\norm{x},\norm{v},\norm{J_tx},\norm{A_tx}\leq b+2n+3n^2$, we have
\[
\vert Ax\vert -\norm{A_tx}\leq\vert Ax\vert -\vert AJ_tx\vert\leq \varepsilon
\]
which is the claim.
\end{proof}

As mentioned before, the fact that 
\[
\lim_{t\to 0^+}\frac{\norm{x-S(t)x}}{t}=\vert Ax\vert
\]
was proved by Crandall in \cite{Cra1973} and the proof proceeds by establishing that
\[
\frac{\norm{x-S(t)x}}{t}\leq\vert Ax\vert
\]
for any $t>0$ as well as
\[
\liminf_{t\to 0^+}\frac{\norm{S(t)x-x}}{t}\geq\vert Ax\vert
\]
and in that way crucially relies on the limit operator $\vert A\cdot\vert$ as well. The latter of these results relies on a result established by Miyadera in \cite{Miy1971}\footnote{The result goes back to earlier work by Brezis \cite{Bre1971} with a special case already contained in \cite{CL1971b} and more general results proved in \cite{CP1972}.} that
\[
\limsup_{t\to 0^+}\left\langle\frac{S(t)x-x}{t},\zeta^*\right\rangle\leq\langle y_0,x_0-x\rangle_s
\]
for $y_0\in Ax_0$, $x\in\overline{\mathrm{dom}A}$ and $\zeta^*\in J(x-x_0)$.

The proof given by Crandall actually only invokes this result for $x\in\mathrm{dom}A$ and, for the proof of Plant's result, it is further sufficient to obtain it only for \emph{some} $\zeta^*\in J(x-x_0)$. Lastly, the proof relies crucially on the use of the functional $\langle\cdot,\cdot\rangle_s$ and in particular on the upper semicontinuity of this functional. In that way, based on the logical methodology that upgrades this upper semicontinuity to a modulus of uniform continuity, we extract the following quantitative version of the above fragment of Miyadera's result:

\begin{lemma}\label{lem:MiyadersQuant}
Let $\omega$ be such that
\[
\forall x,y,z\in X,b\in\mathbb{N},\varepsilon>0\left(\norm{x},\norm{z}\leq b\land\norm{x-y}\leq \omega(b,\varepsilon)\rightarrow\langle z,y\rangle_s\leq\langle z,x\rangle_s+\varepsilon\right).
\]
For $\zeta^*\in J(x-x_0)$ where $x\in\mathrm{dom}A$ with $v\in Ax$ and $y_0\in Ax_0$ where $\norm{x},\norm{v},\norm{x_0},\norm{y_0}\leq b$ for $b\in\mathbb{N}^*$:
\[
\forall \varepsilon>0\forall t\in \left(0,\psi(\varepsilon,b,\omega)\right]\left(\left\langle \frac{S(t)x-x}{t},\zeta^*\right\rangle\leq\langle y_0,x_0-x\rangle_s+\varepsilon\right)
\]
where
\[
\psi(\varepsilon,b,\omega):=\frac{\omega(2b,\varepsilon)}{2b}.
\]
\end{lemma}
\begin{proof}
At first, given an $\varepsilon$, we get for any $t\in \left(0,\frac{\varepsilon}{2b}\right]$ and for all $v\in Ax$ with $\norm{x},\norm{v}\leq b$ that
\[
\norm{x-S(t)x}=\norm{S(0)x-S(t)x}\leq 2\norm{v}t\leq 2b \frac{\varepsilon}{2b}\leq \varepsilon
\]
by Lemma \ref{lem:semigroupprop}, (1). Now, as in Miyadera's proof from \cite{Miy1971}, we get
\[
\langle S(t)x-x,\zeta^*\rangle\leq\int^t_0\langle y_0,x_0-S(\tau)x\rangle_s\mathrm{d}\tau.
\]
Then for $b\geq\norm{x},\norm{v},\norm{x_0},\norm{y_0}$, we get
\[
\langle y_0,x_0-S(t)x\rangle_s\leq\langle y_0,x_0-x\rangle_s +\varepsilon
\]
for any $t\in \left(0,\psi(\varepsilon,b,\omega)\right]$ as by the above, we have
\[
\norm{S(t)x-x}\leq \omega(2b,\varepsilon)
\]
for all such $t$ by assumption on $\omega$ and since we trivially have $\norm{x-x_0}\leq 2b$. Thus in particular we have
\begin{align*}
\langle S(t)x-x,\zeta^*\rangle &\leq\int^t_0\langle y_0,x_0-S(\tau)x\rangle_s\mathrm{d}\tau\\
&\leq t\left( \langle y_0,x_0-x\rangle_s+\varepsilon\right)
\end{align*}
which gives the claim.
\end{proof}
Then, by following the proof given in \cite{Cra1973}, we obtain a quantitative version of the crucial direction
\[
\liminf_{t\to 0^+}\frac{\norm{S(t)x-x}}{t}\geq\vert Ax\vert
\]
of Crandall's proof. Now, already here, a case distinction on whether $\vert Ax\vert=0$ or $\vert Ax\vert >0$ features in the proof of Crandall and the following result first provides a quantitative result on the latter case.

\begin{lemma}\label{lem:semigroupROCpart}
Let $\omega$ be such that
\[
\forall x,y,z\in X,b\in\mathbb{N},\varepsilon>0\left(\norm{x},\norm{z}\leq b\land\norm{x-y}\leq \omega(b,\varepsilon)\rightarrow\langle z,y\rangle_s\leq\langle z,x\rangle_s+\varepsilon\right).
\]
Let further $\varphi$ be a modulus of uniform continuity for $\vert A\cdot\vert$ and let $n$ be as in Lemma \ref{lem:resolventROC}. Then for $x\in\mathrm{dom}A$ with $v\in Ax$ and $b\in\mathbb{N}^*$ with $b\geq\norm{x},\norm{v}$ and where $\vert Ax\vert\geq c$ for $c\in\mathbb{R}_{>0}$, we have
\[
\forall \varepsilon>0\forall t\in \left(0, \varphi_2'(\varepsilon,b,c,n,\varphi,\omega)\right]\left(\vert Ax\vert-\frac{\norm{x-S(t)x}}{t}\leq \varepsilon\right)
\]
where
\[
\varphi_2'(\varepsilon,b,c,n,\varphi,\omega):=\psi(\varepsilon c\min\{\varphi_1(\min\{\varepsilon/2,c/2\},b,n,\varphi),\lambda_0/2\}/4,b+2n+3n^2,\omega)
\]
with $\psi$ as in Lemma \ref{lem:MiyadersQuant} and $\varphi_1$ as in Lemma \ref{lem:resolventROC}.
\end{lemma}
\begin{proof}
Using Lemma \ref{lem:MiyadersQuant}, we get 
\[
\forall \varepsilon>0\forall t\in \left(0,\psi(\varepsilon,b,\omega)\right]\left(\left\langle \frac{S(t)x-x}{t},\zeta^*\right\rangle\leq\langle y_0,x_0-x\rangle_s+\varepsilon\right)
\]
for $\norm{x}, \norm{v}, \norm{x_0}, \norm{y_0}\leq b$ and $\zeta^*\in J(x-x_0)$. Now, for $y_0=A_\lambda x$ and $x_0=J_\lambda x$ with $\lambda< \lambda_0$, we have
\[
\langle y_0,x_0-x\rangle_s=-\lambda\norm{A_\lambda x}^2
\]
as well as
\begin{align*}
\left\langle\frac{S(t)x-x}{t},\zeta^*\right\rangle&\geq -\norm{\frac{S(t)x-x}{t}}\norm{x-J_\lambda x}\\
&=-\norm{\frac{S(t)x-x}{t}}\lambda\norm{A_\lambda x}.
\end{align*}
Therefore, we obtain
\[
\forall \varepsilon>0\forall t\in \left(0,\psi(\varepsilon,b+2n+3n^2,\omega)\right]\left(\norm{\frac{S(t)x-x}{t}}\norm{A_\lambda x}\geq\norm{A_\lambda x}^2-\frac{\varepsilon}{\lambda}\right)
\]
for all such $\lambda$ since $b+2n+3n^2\geq\norm{J_\lambda x}$ and $b\geq\norm{A_\lambda x}$ as before. Since $\vert Ax\vert\geq c$, we have that for $\lambda\leq\min\{\varphi_1(c/2,b,n,\varphi),\lambda_0/2\}$ that 
\[
c/2=c-c/2\leq\vert Ax\vert -c/2\leq\norm{A_\lambda x}
\]
by Lemma \ref{lem:resolventROC}. Therefore, we have that
\[
\forall \varepsilon>0\forall t\in \left(0,\psi(\varepsilon,b+2n+3n^2,\omega)\right]\left(\norm{\frac{S(t)x-x}{t}}\geq\norm{A_\lambda x}-\frac{\varepsilon}{\lambda c/2}\right)
\]
for all $\lambda\leq\min\{\varphi_1(c/2,b,n,\varphi),\lambda_0/2\}$ and thus in particular
\begin{align*}
\vert Ax\vert-\norm{\frac{S(t)x-x}{t}}&\leq\vert Ax\vert-\norm{A_\lambda x}+\frac{\varepsilon}{\lambda c/2}\\
&\leq \delta/2 +\frac{\varepsilon}{\lambda c/2}
\end{align*}
for all $t\leq\psi(\varepsilon,b+2n+3n^2,\omega)$ and for all $\lambda\leq\min\{\varphi_1(\min\{\delta/2,c/2\},b,n,\varphi),\lambda_0/2\}$ and $\delta>0$. Thus, lastly, for 
\[
t\leq \psi(\varepsilon c\min\{\varphi_1(\min\{\varepsilon/2,c/2\},b,n,\varphi),\lambda_0/2\}/4,b+2n+3n^2,\omega)
\]
we have 
\[
\vert Ax\vert-\norm{\frac{S(t)x-x}{t}}\leq\varepsilon.\qedhere
\]
\end{proof}

For the other case, i.e.\ where $\vert Ax\vert=0$, it is immediately clear that for $\vert Ax\vert \leq\varepsilon$, we get 
\[
\vert Ax\vert-\frac{\norm{x-S(t)x}}{t}\leq\vert Ax\vert\leq\varepsilon
\]
for all $t$. However, this allows for a smoothening of the above case distinction (see the later logical remarks for further discussions of this) in the form of the following lemma:

\begin{lemma}\label{lem:semigroupROC}
Let $\omega$ be such that
\[
\forall x,y,z\in X,b\in\mathbb{N},\varepsilon>0\left(\norm{x},\norm{z}\leq b\land\norm{x-y}\leq \omega(b,\varepsilon)\rightarrow\langle z,y\rangle_s\leq\langle z,x\rangle_s+\varepsilon\right).
\]
Let further $\varphi$ be a modulus of uniform continuity for $\vert A\cdot\vert$ and let $n$ be as in Lemma \ref{lem:resolventROC}. Then for $x\in\mathrm{dom}A$ with $v\in Ax$ and $b\in\mathbb{N}^*$ with $b\geq\norm{x},\norm{v}$, we have
\[
\forall \varepsilon>0\forall t\in \left(0, \varphi_2(\varepsilon,b,n,\omega,\varphi)\right]\left(\vert Ax\vert-\frac{\norm{x-S(t)x}}{t}\leq \varepsilon\right)
\]
where
\[
\varphi_2(\varepsilon,b,n,\omega,\varphi):=\psi(\varepsilon^2\min\{\varphi_1(\varepsilon/2,b,n,\varphi),\lambda_0/2\}/4,b+2n+3n^2,\omega).
\]
with $\psi$ as in Lemma \ref{lem:MiyadersQuant} and $\varphi_1$ as in Lemma \ref{lem:resolventROC}.
\end{lemma}
\begin{proof}
Let $\varepsilon$ be given. Then either $\vert Ax\vert\leq\varepsilon$ whereas 
\[
\vert Ax\vert-\frac{\norm{x-S(t)x}}{t}\leq \varepsilon
\]
for any $t$. Otherwise, we have $\vert Ax\vert\geq\varepsilon$ and thus from Lemma \ref{lem:semigroupROCpart}, with $c=\varepsilon$, it follows that
\[
\vert Ax\vert-\frac{\norm{x-S(t)x}}{t}\leq \varepsilon
\]
for all
\[
t\leq \psi(\varepsilon^2\min\{\varphi_1(\varepsilon/2,b,n,\varphi),\lambda_0/2\}/4,b+2n+3n^2,\omega).\qedhere
\]
\end{proof}

Using those two results, we can then give a quantitative version of the partial results on the way to Plants results discussed above, in the form of a rate of Cauchyness and a rate of convergence, respectively.

\medskip

In that context, we follow the notation used in \cite{Pla1981} and write
\[
\alpha(a,b)=\norm{\frac{a}{\norm{a}}-\frac{b}{\norm{b}}}\leq 2
\]
where $a,b\neq 0$ for the generalized angle of Clarkson \cite{Cla1936}. Similar to the proof given in \cite{Pla1981}, we rely on two fundamental inequalities of $\alpha$:
\begin{lemma}[essentially \cite{Cla1936}]\label{lem:ClarkLemma}
Let $a,b\neq 0$. Then
\[
\vert\norm{a}\alpha(a,b)-\norm{a-b}\vert\leq\vert\norm{a}-\norm{b}\vert.
\]
If further $a+b\neq 0$, then
\[
\norm{a+b}\leq (1-2\eta(\alpha(a+b,a)))\norm{a}+\norm{b}
\]
where $\eta$ is a modulus of uniform convexity for the space $X$.
\end{lemma}

\begin{lemma}\label{lem:resCauchypart}
Let $X$ be a uniformly convex Banach space with a modulus of uniform convexity $\eta:(0,2]\to (0,1]$. Let further $\varphi$ be a modulus of uniform continuity for $\vert A\cdot\vert$ and let $n$ be as in Lemma \ref{lem:resolventROC}. Let further $x\in\mathrm{dom}A$ with $v\in Ax$ and $b\in\mathbb{N}^*$ with $b\geq\norm{x},\norm{v}$. Suppose $\vert Ax\vert\geq c$ for $c\in\mathbb{R}_{>0}$. Then
\[
\forall \varepsilon>0\forall t\in \left(0,\varphi_3'(\varepsilon,b,c,\eta,n,\varphi)\right]\forall s\in (0,t)\left(\norm{\frac{x-J_tx}{t}-\frac{x-J_sx}{s}}\leq \varepsilon\right)
\]
where
\begin{align*}
\varphi_3'(\varepsilon,b,c,\eta,n,\varphi):=\min\{&\varphi_1(\varepsilon/2,b,n,\varphi),\varphi_1(\eta(\min\{\varepsilon/2b,2\}) c/2,b,n,\varphi),\\
&\varphi_1(c/2,b,n,\varphi),\lambda_0/2\}
\end{align*}
with $\varphi_1$ as in Lemma \ref{lem:resolventROC}.
\end{lemma}
\begin{proof}
If $x=J_tx$ or $x=J_sx$, then $0\in Ax$ and thus $\vert Ax\vert=0$. As we have assumed $\vert Ax\vert\geq c>0$, we get $x\neq J_tx$ and $x\neq J_sx$. We write $\alpha_{s,t}=\alpha(x-J_sx,x-J_tx)$ where $s\in (0,t)$ and $t\leq\lambda_0/2<\lambda_0$. Using Lemma \ref{lem:ClarkLemma} with $a=x-J_sx$ and $b=J_sx-J_tx$, we have 
\[
\norm{x-J_tx}\leq(1-2\eta(\alpha_{s,t}))\norm{x-J_sx}+\norm{J_tx-J_sx}.
\]
Using Lemma \ref{lem:basicProp}, items (3) and (5), we get
\begin{align*}
\norm{J_tx-J_sx}&=\norm{J_s\left(\frac{s}{t}x+\frac{t-s}{t}J_tx\right)-J_sx}\\
&\leq\left( 1-\frac{s}{t}\right)\norm{x-J_tx}
\end{align*}
and thus we have
\[
s\norm{A_tx}\leq (1-2\eta(\alpha_{s,t}))\norm{x-J_sx},
\]
i.e.
\[
2\eta(\alpha_{s,t})\norm{A_sx}\leq \norm{A_sx}-\norm{A_tx}.
\]
Therefore, we have for $0<t\leq \varphi_1(c/2,b,n,\varphi)$ that
\[
c-\norm{A_tx}\leq \vert Ax\vert-\norm{A_tx}\leq c/2
\]
so that $c/2\leq \norm{A_tx}$ and for $s\in (0,t)$, we get that
\[
\eta(\alpha_{s,t})c\leq 2\eta(\alpha_{s,t})\norm{A_tx}\leq 2\eta(\alpha_{s,t})\norm{A_sx}\leq\norm{A_sx}-\norm{A_tx}\leq\vert Ax\vert-\norm{A_tx}.
\]
By Lemma \ref{lem:resolventROC}, we have for any $\varepsilon$ that 
\[
\forall t\in \left(0, \min\left\{\varphi_1(\varepsilon,b,n,\varphi),\varphi_1(c/2,b,n,\varphi)\right\}\right]\forall s\in (0,t)\left(\eta(\alpha_{s,t})c\leq \varepsilon\right)
\]
which, in particular, implies
\begin{gather*}
\forall t\in \left(0,\min\left\{\varphi_1(\eta(\min\{\varepsilon/2b,2\}) c/2,b,n,\varphi),\varphi_1(c/2,b,n,\varphi)\right\}\right]\forall s\in (0,t)\\\left(\eta(\alpha_{s,t})\leq \eta(\min\{\varepsilon/2b,2\})/2\right)
\end{gather*}
and using that $\eta$ is nondecreasing, we get
\begin{gather*}
\forall t\in \left(0,\min\left\{\varphi_1(\eta(\min\{\varepsilon/2b,2\}) c/2,b,n,\varphi),\varphi_1(c/2,b,n,\varphi),\lambda_0\right\}\right]\forall s\in (0,t)\\
\left(\alpha_{s,t}\leq \varepsilon/2b\right).
\end{gather*}
Using Lemma \ref{lem:ClarkLemma} with $a=x-J_tx/t$ and $b=x-J_sx/s$ (noting that $\alpha_{s,t}=\alpha(a,b)$ for these $a,b$) together with $s<t$ as well as the triangle inequality, we now have
\begin{align*}
\norm{\frac{x-J_tx}{t}-\frac{x-J_sx}{s}}&\leq \left\vert\frac{\norm{x-J_tx}}{t}\alpha_{s,t}-\norm{\frac{x-J_tx}{t}-\frac{x-J_sx}{s}}\right\vert+\frac{\norm{x-J_tx}}{t}\alpha_{s,t}\\
&\leq\left\vert\norm{\frac{x-J_tx}{t}}-\norm{\frac{x-J_sx}{s}}\right\vert+\frac{\norm{x-J_tx}}{t}\alpha_{s,t}\\
&\leq\left(\vert Ax\vert-\norm{\frac{x-J_tx}{t}}\right)+b\alpha_{s,t}.
\end{align*}
Thus for $0<t\leq\varphi_3'(\varepsilon,b,c,\eta,n,\varphi)$ and for $s\in (0,t)$, we have
\[
\norm{\frac{x-J_tx}{t}-\frac{x-J_sx}{s}}\leq \varepsilon/2 +b\varepsilon/2b\leq \varepsilon
\]
by Lemma \ref{lem:resolventROC}.
\end{proof}

Again, the case for $\vert Ax\vert=0$ is trivial and yields the following quantitative version: if $\vert Ax\vert\leq\varepsilon/2$, then in particular
\[
\norm{\frac{x-J^A_tx}{t}-\frac{x-J^A_sx}{s}}\leq \norm{A_tx}+\norm{A_sx}\leq\vert Ax\vert+\vert Ax\vert\leq \varepsilon.
\]
In that way, we get the following smoothening for both results combined.

\begin{lemma}\label{lem:resCauchy}
Let $X$ be a uniformly convex Banach space with a modulus of uniform convexity $\eta:(0,2]\to (0,1]$. Let further $\varphi$ be a modulus of uniform continuity for $\vert A\cdot\vert$ and let $n$ be as in Lemma \ref{lem:resolventROC}. Let further $x\in\mathrm{dom}A$ with $v\in Ax$ and $b\in\mathbb{N}^*$ with $b\geq\norm{x},\norm{v}$. Then
\[
\forall \varepsilon>0\forall t\in \left(0,\varphi_3(\varepsilon,b,\eta,n,\varphi)\right]\forall s\in (0,t)\left(\norm{\frac{x-J^A_tx}{t}-\frac{x-J^A_sx}{s}}\leq \varepsilon\right)
\]
where
\begin{align*}
\varphi_3(\varepsilon,b,\eta,n,\varphi):=\min\{&\varphi_1(\varepsilon/2,b,n,\varphi),\varphi_1(\eta(\min\{\varepsilon/2b,2\}) \varepsilon/4,b,n,\varphi),\\
&\varphi_1(\varepsilon/4,b,n,\varphi),\lambda_0/2\}
\end{align*}
with $\varphi_1(\varepsilon,b,n)$ as in Lemma \ref{lem:resolventROC}.
\end{lemma}
\begin{proof}
Let $\varepsilon$ be given. Then either $\vert Ax\vert\leq\varepsilon/2$ whereas
\[
\norm{\frac{x-J^A_tx}{t}-\frac{x-J^A_sx}{s}}\leq \varepsilon
\]
for any $t,s<\lambda_0$ as discussed above. Otherwise we have $\vert Ax\vert\geq \varepsilon/2$ and thus by Lemma \ref{lem:resCauchypart}, with $c=\varepsilon/2$, it follows that
\[
\norm{\frac{x-J^A_tx}{t}-\frac{x-J^A_sx}{s}}\leq \varepsilon
\]
for $s\in (0,t)$ and 
\[
t\leq \min\{\varphi_1(\varepsilon/2,b,n,\varphi),\varphi_1(\eta(\min\{\varepsilon/2b,2\}) \varepsilon/4,b,n,\varphi),\varphi_1(\varepsilon/4,b,n,\varphi),\lambda_0/2\}.\qedhere
\]
\end{proof}

\begin{lemma}[Plant \cite{Pla1981}, Eq. (2.10)]\label{lem:ClarkIntLemma}
Let $x\in\mathrm{dom}A$ and $t,\lambda>0$. Then
\[
\norm{J_\lambda x-S(t)x}\leq \left(1 -\frac{t}{\lambda}\right)\norm{x-J_\lambda x}+\frac{2}{\lambda}\int^t_0\norm{x-S(s)x}\mathrm{ds}.
\]
\end{lemma}

\begin{lemma}\label{lem:resSemiCombpart}
Let $X$ be a uniformly convex Banach space with a modulus of uniform convexity $\eta:(0,2]\to (0,1]$ and let $\omega$ be such that
\[
\forall x,y,z\in X,b\in\mathbb{N},\varepsilon>0\left(\norm{x},\norm{z}\leq b\land\norm{x-y}\leq \omega(b,\varepsilon)\rightarrow\langle z,y\rangle_s\leq\langle z,x\rangle_s+\varepsilon\right).
\]
Let further $\varphi$ be a modulus of uniform continuity for $\vert A\cdot\vert$ and let $n$ be as in Lemma \ref{lem:resolventROC}. Let further $x\in\mathrm{dom}A$ with $v\in Ax$ and $b\in\mathbb{N}^*$ with $b\geq\norm{x},\norm{v}$. Suppose that $\vert Ax\vert\geq c$ for $c\in\mathbb{R}_{>0}$. Then
\[
\forall \varepsilon>0\forall t,\frac{s}{t}\in \left(0,\varphi_4'(\varepsilon,b,c,\eta,n,\omega,\varphi)\right]\left( \norm{\frac{x-J^A_tx}{t}-\frac{x-S(s)x}{s}}\leq \varepsilon\right)
\]
where
\begin{align*}
\varphi_4'(\varepsilon,b,c,\eta,n,\omega,\varphi):=\min\{&\varphi_1(\varepsilon/3,b,n,\varphi),\varphi_2(\varepsilon/3,b,n,\omega,\varphi),\\
&\varphi_1(\eta(\min\{\varepsilon/3b,2\})c/4,b,n,\varphi),\sqrt{\varphi_2(c/2,b,n,\omega,\varphi)},\\
&\eta(\min\{\varepsilon/3b,2\})c/8b,1,\lambda_0/2\}
\end{align*}
with $\varphi_1,\varphi_2$ as in Lemmas \ref{lem:resolventROC}, \ref{lem:semigroupROC}, respectively.
\end{lemma}
\begin{proof}
As before, $x\neq S(s)x$ and $x\neq J_tx$ as $\vert Ax\vert\geq c>0$. We write $\alpha'_{s,t}=\alpha(x-S(s)x,x-J_tx)$ for $t,s<\lambda_0$. Using Lemma \ref{lem:ClarkLemma}, we again obtain
\[
\norm{x-J_tx}\leq \left( 1-2\eta(\alpha'_{s,t})\right)\norm{x-S(s)x}+\norm{J_tx-S(s)x}.
\]
Using Lemma \ref{lem:ClarkIntLemma}, we get for $t,s\leq \min\{\lambda_0/2,1\}$:
\begin{align*}
&\norm{x-J_tx} \\
&\qquad\leq \left( 1-2\eta(\alpha'_{s,t})\right)\norm{x-S(s)x}+\left(1 -\frac{s}{t}\right)\norm{x-J_t x}+\frac{2}{t}\int^s_0\norm{x-S(\tau)x}\mathrm{d\tau}\\
&\qquad\leq \left( 1-2\eta(\alpha'_{s,t})\right)\norm{x-S(s)x}+\left(1 -\frac{s}{t}\right)\norm{x-J_t x}+\frac{2}{t}\int^s_0s\frac{\norm{x-S(\tau)x}}{\tau}\mathrm{d\tau}\\
&\qquad\leq \left( 1-2\eta(\alpha'_{s,t})\right)\norm{x-S(s)x}+\left(1 -\frac{s}{t}\right)\norm{x-J_t x}+\frac{s^2}{t}2b
\end{align*}
which implies that
\[
2\eta(\alpha'_{s,t})\frac{\norm{x-S(s)x}}{s}\leq \vert Ax\vert-\frac{\norm{x-J_t x}}{t}+\frac{s}{t}2b.
\]
Now for
\[
t\leq \min\{\varphi_1(\varepsilon/2,b,n,\varphi),\sqrt{\varphi_2(c/2,b,n,\omega,\varphi)}\}
\]
and 
\[
\frac{s}{t}\leq \min\{\varepsilon/4b,\sqrt{\varphi_2(c/2,b,n,\omega,\varphi)}\}
\]
we obtain that 
\[
s\leq t\frac{s}{t}\leq \varphi_2(c/2,b,n,\omega,\varphi)
\]
and thus (using Lemma \ref{lem:semigroupROC}), we obtain
\begin{align*}
\eta(\alpha'_{s,t})c&\leq 2\eta(\alpha'_{s,t})\left(\vert Ax\vert - c/2\right)\\
&\leq 2\eta(\alpha'_{s,t})\frac{\norm{x-S(s)x}}{s}\\
&\leq \vert Ax\vert-\frac{\norm{x-J_t x}}{t}+\frac{s}{t}2b\\
&\leq \varepsilon/2+\frac{s}{t}2b\\
&\leq \varepsilon/2+2b(\varepsilon/4b)\\
&\leq \varepsilon.
\end{align*}
Dividing by $c$, we get $\eta(\alpha'_{s,t})\leq \frac{\varepsilon}{c}$ for all such $t,s$. Thus, using that $\eta$ is nondecreasing, we have $\alpha'_{s,t}\leq \varepsilon$ for
\[
t \leq \min\{\varphi_1(\eta(\min\{\varepsilon,2\})c/4,b,n,\varphi),\sqrt{\varphi_2(c/2,b,n,\omega,\varphi)}\}
\]
and
\[
\frac{s}{t}\leq \min\{\eta(\min\{\varepsilon,2\})c/8b,\sqrt{\varphi_2(c/2,b,n,\omega,\varphi)}\}.
\]
Using Lemma \ref{lem:ClarkLemma} and triangle inequality again, we now have similarly to before
\begin{align*}
&\norm{\frac{x-J_tx}{t}-\frac{x-S(s)x}{s}}\\
&\qquad\leq \left\vert\frac{\norm{x-J_tx}}{t}\alpha'_{s,t}-\norm{\frac{x-J_tx}{t}-\frac{x-S(s)x}{s}}\right\vert+\frac{\norm{x-J_tx}}{t}\alpha'_{s,t}\\
&\qquad\leq\left\vert\norm{\frac{x-J_tx}{t}}-\norm{\frac{x-S(s)x}{s}}\right\vert+\frac{\norm{x-J_tx}}{t}\alpha'_{s,t}\\
&\qquad\leq\left(\vert Ax\vert-\norm{\frac{x-J_tx}{t}}\right)+\left(\vert Ax\vert-\norm{\frac{x-S(s)x}{s}}\right)+b\alpha'_{s,t}.
\end{align*}
Thus for $0<t,\frac{s}{t}\leq \varphi_4'(\varepsilon,b,c,\eta,n,\omega,\varphi)$, we have
\[
\norm{\frac{x-J_tx}{t}-\frac{x-S(s)x}{s}}\leq \varepsilon/3 + \varepsilon/3 +b(\varepsilon/3b)\leq \varepsilon
\]
by Lemma \ref{lem:resolventROC} and Lemma \ref{lem:semigroupROC}.
\end{proof}

As before, a smoothening of this result can be achieved by extracting from the proof for the case of $\vert Ax\vert =0$ the following quantitative version: if $\vert Ax\vert\leq\varepsilon/2$, then
\[
\norm{\frac{x-J^A_tx}{t}-\frac{x-S(s)x}{s}}\leq\frac{\norm{x-J^A_tx}}{t}+\frac{\norm{x-S(s)x}}{s}\leq\vert Ax\vert+\vert Ax\vert\leq\varepsilon.
\]
Therefore, we obtain the following result:

\begin{lemma}\label{lem:resSemiComb}
Let $X$ be a uniformly convex Banach space with a modulus of uniform convexity $\eta:(0,2]\to (0,1]$ and let $\omega$ be such that
\[
\forall x,y,z\in X,b\in\mathbb{N},\varepsilon>0\left(\norm{x},\norm{z}\leq b\land\norm{x-y}\leq \omega(b,\varepsilon)\rightarrow\langle z,y\rangle_s\leq\langle z,x\rangle_s+\varepsilon\right).
\]
Let further $\varphi$ be a modulus of uniform continuity for $\vert A\cdot\vert$ and let $n$ be as in Lemma \ref{lem:resolventROC}. Let further $x\in\mathrm{dom}A$ with $v\in Ax$ and $b\in\mathbb{N}^*$ with $b\geq\norm{x},\norm{v}$. Then
\[
\forall \varepsilon>0\forall t,\frac{s}{t}\in \left(0,\varphi_4(\varepsilon,b,\eta,n,\omega,\varphi)\right]\left( \norm{\frac{x-J^A_tx}{t}-\frac{x-S(s)x}{s}}\leq \varepsilon\right)
\]
where
\begin{align*}
\varphi_4(\varepsilon,b,\eta,n,\omega,\varphi):=\min\{&\varphi_1(\varepsilon/3,b,n,\varphi),\varphi_2(\varepsilon/3,b,n,\omega,\varphi),\\
&\varphi_1(\eta(\min\{\varepsilon/3b,2\})\varepsilon/8,b,n,\varphi),\sqrt{\varphi_2(\varepsilon/4,b,n,\omega,\varphi)},\\
&\eta(\min\{\varepsilon/3b,2\})\varepsilon/16b,1,\lambda_0/2\}
\end{align*}
with $\varphi_1,\varphi_2$ as in Lemmas \ref{lem:resolventROC}, \ref{lem:semigroupROC}, respectively.
\end{lemma}
\begin{proof}
Let $\varepsilon$ be given. Then either $\vert Ax\vert\leq\varepsilon/2$ which implies
\[
\norm{\frac{x-J^A_tx}{t}-\frac{x-S(s)x}{s}}\leq \varepsilon
\]
as above for any such $t$ and $s$ or $\vert Ax\vert\geq\varepsilon/2$ where now the result is implied for any 
\[
t,\frac{s}{t}\in \left(0,\varphi_4'(\varepsilon,b,\varepsilon/2,\eta,n,\omega,\varphi)\right]
\]
by Lemma \ref{lem:resSemiCombpart}, with $c=\varepsilon/2$.
\end{proof}

Finally, a combination of these two quantitative results yields a quantitative version of the theorem of Plant.

\begin{theorem}
Let $X$ be a uniformly convex Banach space with a modulus of uniform convexity $\eta:(0,2]\to (0,1]$ and let $\omega$ be such that
\[
\forall x,y,z\in X,b\in\mathbb{N},\varepsilon>0\left(\norm{x},\norm{z}\leq b\land\norm{x-y}\leq \omega(b,\varepsilon)\rightarrow\langle z,y\rangle_s\leq\langle z,x\rangle_s+\varepsilon\right).
\]
Let further $\varphi$ be a modulus of uniform continuity for $\vert A\cdot\vert$ and let $n$ be as in Lemma \ref{lem:resolventROC}. Let further $x\in\mathrm{dom}A$ with $v\in Ax$ and $b\in\mathbb{N}^*$ with $b\geq\norm{x},\norm{v}$. Then
\[
\forall \varepsilon>0\forall t\in \left(0,\Phi(\varepsilon,b,\eta,\omega,\varphi,n)\right]\left(\frac{\norm{J^A_t x-S(t)x}}{t}\leq \varepsilon\right)
\]
where
\[
\Phi(\varepsilon,b,\eta,\omega,\varphi,n):=\left(\min\{\varphi_3(\varepsilon/2,b,\eta,n,\varphi),\varphi_4(\varepsilon/2,b,\eta,n,\omega,\varphi)\}\right)^2
\]
with $\varphi_1$ -- $\varphi_4$ as well as $\psi$ defined by
\begin{align*}
\varphi_1(\varepsilon,b,n,\varphi):=&\min\{\varphi(\varepsilon,b+2n+3n^2)/b,\lambda_0/2\},\\
\psi(\varepsilon,b,\omega):=&\frac{\omega(2b,\varepsilon)}{2b},\\
\varphi_2(\varepsilon,b,n,\omega,\varphi):=&\psi(\varepsilon^2\min\{\varphi_1(\varepsilon/2,b,n,\varphi),\lambda_0/2\}/4,b+2n+3n^2,\omega),\\
\varphi_3(\varepsilon,b,\eta,n,\varphi):=&\min\{\varphi_1(\varepsilon/2,b,n,\varphi),\varphi_1(\eta(\min\{\varepsilon/2b,2\}) \varepsilon/4,b,n,\varphi),\\
&\hphantom{\min\{\,}\varphi_1(\varepsilon/4,b,n,\varphi),\lambda_0/2\},\\
\varphi_4(\varepsilon,b,\eta,n,\omega,\varphi):=&\min\{\varphi_1(\varepsilon/3,b,n,\varphi),\varphi_2(\varepsilon/3,b,n,\omega,\varphi),\\
&\hphantom{\min\{\,}\varphi_1(\eta(\min\{\varepsilon/3b,2\})\varepsilon/8,b,n,\varphi),\sqrt{\varphi_2(\varepsilon/4,b,n,\omega,\varphi)},\\
&\hphantom{\min\{\,}\eta(\min\{\varepsilon/3b,2\})\varepsilon/16b,1,\lambda_0/2\}.
\end{align*}
\end{theorem}
\begin{proof}
Using the triangle inequality, we have
\[
\frac{\norm{J_tx-S(t)x}}{t}\leq\norm{\frac{x-J_tx}{t}-\frac{x-J_{\sqrt{t}}x}{\sqrt{t}}}+\norm{\frac{x-J_{\sqrt{t}}x}{\sqrt{t}}-\frac{x-S(t)x}{t}}.
\]
Then, for $t\leq 1$, we have $t\leq\sqrt{t}$ and $t/\sqrt{t}=\sqrt{t}$ so that for $t\leq \Phi(\varepsilon,b,\eta,\omega,\varphi,n)$, we obtain
\[
\frac{\norm{J_tx-S(t)x}}{t}\leq \varepsilon
\]
using Lemmas \ref{lem:resCauchy} and \ref{lem:resSemiComb}.
\end{proof}

\begin{remark}
While the above result uses the construction of $\varphi_1$ from $\varphi$ exhibited in Lemma \ref{lem:resolventROC}, it is clear that if $\varphi_1$ is any other rate of convergence for $\norm{A_tx}$ to $\vert Ax\vert$ as $t\to 0$, the above result nevertheless remains valid.
\end{remark}

\subsection{An analysis of Reich's result}

Similar as in the context of Plant's result, in this subsection we fix a Banach space $X$ and an accretive operator $A$ that now satisfies the range condition $(RC)$. As before, let $S$ be the semigroup on $\overline{\mathrm{dom}A}$ generated by $A$ using the Crandall-Liggett formula. The proof for Reich's result now proceeds by establishing
\[
\lim_{t\to\infty}\frac{\norm{J_tx}}{t}=d(0,\mathrm{ran}A)
\]
and concluding from this that $J_tx/t$ is Cauchy for $t\to\infty$. This result is then in turn used to conclude the claim. While Reich actually establishes his result even for $x\in\overline{\mathrm{dom}A}$, we here focus for simplicity on the case where $x\in\mathrm{dom}A$.

\medskip

The main object used in these proofs is the concrete value
\[
d:=d(0,\mathrm{ran}A)=\inf\{\norm{y}\mid y\in \mathrm{ran}A\}
\]
and for the quantitative results, the logical methodology implies (see the later logical remarks for a discussion of this) a dependence on a function $f$ witnessing the above infimum quantitatively in the sense that $f:\mathbb{R}_{>0}\to\mathbb{N}$ satisfies 
\[
\forall \varepsilon>0\exists (y,z)\in A \left( \norm{y},\norm{z}\leq f(\varepsilon)\land \norm{z}-d\leq\varepsilon\right).
\]
The proof of Reich's result then proceeds by a case distinction on whether $d>0$ or $d=0$ but, as before with the quantitative analysis of Plant's result, this case distinction can be smoothed as will be exhibited later. We at first begin with the following result which provides a rate of convergence for the limit $\norm{J_tx}/t\to d$ for $t\to\infty$ (which can be obtained as the sequence is monotone).
\begin{lemma}\label{lem:resolventROCInf}
Let $x\in\mathrm{dom}A$ with $v\in Ax$ and $b\in\mathbb{N}^*$ where $b\geq\norm{x},\norm{v}$. Suppose that $f:\mathbb{R}_{>0}\to\mathbb{N}$ satisfies $f(\varepsilon)\geq f(\delta)$ for $\varepsilon\leq \delta$ and 
\[
\forall \varepsilon>0\exists (y,z)\in A \left( \norm{y},\norm{z}\leq f(\varepsilon)\land \norm{z}-d\leq\varepsilon\right).
\]
Then we have
\[
\forall \varepsilon>0\forall t\geq\varphi(\varepsilon,b,f)\left(\left\vert\frac{\norm{J_tx}}{t}-d\right\vert\leq \varepsilon\right)
\]
where
\[
\varphi(\varepsilon,b,f):=\frac{8(b+f(\varepsilon/2))}{\varepsilon}.
\]
\end{lemma}
\begin{proof}
As $A_t x\in AJ_tx$ for any $t>0$, we have $d\leq\norm{A_tx}$. Let $\varepsilon$ be given and let $z\in Ay$ such that $\norm{z}-d\leq \varepsilon/2$ and $\norm{y},\norm{z}\leq f(\varepsilon/2)$. Then using Lemma \ref{lem:basicProp}, (7), we have
\begin{align*}
\norm{A_tx}&\leq\norm{A_tx-A_ty}+\norm{A_ty}\\
&\leq \frac{2}{t}\norm{x-y}+\norm{z}\\
&\leq \frac{2(b+f(\varepsilon/2))}{t}+d+\varepsilon/2.
\end{align*}
Thus, for $t\geq (\varepsilon/4(b+f(\varepsilon/2)))^{-1}$, we have
\[
\norm{A_tx}\leq \frac{2(b+f(\varepsilon/2))}{(\varepsilon/4(b+f(\varepsilon/2)))^{-1}}+d+\varepsilon/2\leq d+\varepsilon.
\]
Now, for $t\geq (\varepsilon/8(b+f(\varepsilon/2)))^{-1}$, we obtain
\begin{align*}
\left\vert\frac{\norm{J_tx}}{t}-d\right\vert&\leq\left\vert\frac{\norm{J_tx}}{t}-\frac{\norm{x-J_tx}}{t}\right\vert+\left\vert\norm{A_tx}-d\right\vert\\
&\leq\frac{\norm{x}}{t}+\left\vert\norm{A_tx}-d\right\vert\\
&\leq \varepsilon
\end{align*}
as $t\geq (\varepsilon/8(b+f(\varepsilon/2)))^{-1}$ and thus $\norm{A_tx}-d\leq \varepsilon/2$ as well as $t\geq (\varepsilon/2b)^{-1}$ and thus $\norm{x}/t\leq \varepsilon/2$.
\end{proof}

The following result is a quantitative version of the well-known result due to Reich \cite{Rei1980} that $d>0$ implies that $\norm{J_tx}\to\infty$ for $t\to\infty$ and $x\in\mathrm{dom}A$.
\begin{lemma}\label{lem:quantReichLem}
Assume that $d\geq D$ for $D\in\mathbb{R}_{>0}$. Let $x\in\mathrm{dom }A$ with $v\in Ax$ and $b\in\mathbb{N}^*$ where $b\geq \norm{x},\norm{v}$. Then we have
\[
\forall K>0\forall t\geq \psi(K,b,D)\left(\norm{J_tx}\geq K\right)
\]
where
\[
\psi(K,b,D):=\frac{b+K}{D}.
\]
\end{lemma}
\begin{proof}
Suppose the claim is false, i.e.\ there is a $K$ and a $t\geq\psi(K,b,D)$ such that $\norm{J_tx}<K$. Then, we have
\begin{align*}
\norm{J_tx-J_1J_tx}&\leq \vert A J_tx\vert\\
&\leq \norm{x-J_tx}/t\\
&<(b+K)/(D^{-1}(b+K))\\
&\leq D.
\end{align*}
Thus $\norm{A_1J_tx}<D\leq d$ which is a contradiction as $A_1J_tx\in\mathrm{ran}A$.
\end{proof}

\begin{lemma}[essentially Reich \cite{Rei1981}]\label{lem:reichLem}
Let $X$ be uniformly convex with a modulus of uniform convexity $\eta$. Then, for $\varepsilon\in (0,2]$, we have $2\eta(\varepsilon)\leq 1-\langle y,j\rangle$ for all $j\in Jx$ with $\norm{x}=\norm{y}=1$ and $\norm{x-y}\geq\varepsilon$.
\end{lemma}
\begin{proof}
Let $x,y$ and $j\in Jx$ be given with $\norm{x}=\norm{y}=1$ and $\norm{x-y}\geq\varepsilon$. Then
\[
\frac{\norm{x+y}}{2}\leq 1-\eta(\varepsilon)
\]
by definition of $\eta$. Thus as $\langle x,j\rangle =\norm{x}^2=1$ and $\norm{j}=\norm{x}=1$, we have
\[
1/2+1/2\langle y,j\rangle=\langle (x+y)/2,j\rangle\leq 1-\eta(\varepsilon)
\]
which yields the claim.
\end{proof}

\begin{lemma}\label{lem:CauchyIntRes}
Let $X$ be a uniformly convex Banach space with a modulus of uniform convexity $\eta:(0,2]\to (0,1]$ with (w.l.o.g.) $\eta(\varepsilon)\leq\varepsilon$. Suppose that $f:\mathbb{R}_{>0}\to\mathbb{N}$ satisfies 
\[
\forall \varepsilon>0\exists (y,z)\in A \left( \norm{y},\norm{z}\leq f(\varepsilon)\land \norm{z}-d\leq\varepsilon\right).
\]
Let $\varepsilon>0$ be given, assume that $D\in\mathbb{R}_{>0}$ with $d\geq D$ and let $z\in Ay$ be such that
\[
\norm{z}\leq d+2d\eta(\min\{\varepsilon/2,2\})
\]
as well as $c\geq\norm{y},\norm{z}$. Let $x\in\mathrm{dom }A$ with $v\in Ax$ and $b\in\mathbb{N}^*$ with $b\geq \norm{x},\norm{v}$. Then, for any $t\geq\varphi_1(\varepsilon,b,D,c,\eta,f)$:
\[
\norm{\frac{z}{\norm{z}}+\frac{J_tx}{\norm{J_tx}}}\leq \varepsilon
\]
where
\begin{align*}
\varphi_1(\varepsilon,b,D,c,\eta,f):=\max\Bigg\{&\psi(c+1,b,D),\psi\left(\left(\frac{4}{\varepsilon}+1\right)c,b,D\right),\\
&\left(\frac{D\left(2\eta(\min\{\varepsilon/2,2\})\right)^2}{18}\right)^{-1}(c+b),\\
&\varphi\left(\frac{D\left(2\eta(\min\{\varepsilon/2,2\})\right)^2}{18},b,f\right)\Bigg\}
\end{align*}
with $\varphi$ defined as in Lemma \ref{lem:resolventROCInf} and $\psi$ as in Lemma \ref{lem:quantReichLem}.
\end{lemma}
\begin{proof}
As $A_tx\in AJ_tx$ and as $A$ is accretive, there is a $j_t\in J(y-J_tx)$ such that $\langle z-A_tx,j_t\rangle\geq 0$. Therefore we have
\[
\left\langle\frac{z}{\norm{z}},\frac{j_t}{\norm{y-J_tx}}\right\rangle\geq\left\langle\frac{A_tx}{\norm{z}},\frac{j_t}{\norm{y-J_tx}}\right\rangle
\]
for any $t$ such that $t\geq \psi(c+1,b,D)$ as then $\norm{J_tx}>c$ which implies $y\neq J_tx$ and $z\neq 0$ follows by $\norm{z}\geq d\geq D>0$. Then further
\[
\left\langle J_tx-x,\frac{j_t}{\norm{y-J_tx}}\right\rangle\leq\norm{y-x}-\norm{y-J_tx},
\]
by (an argument similar to the proof of) Proposition \ref{pro:DualMapTrick} and we thus obtain
\begin{align*}
\left\langle A_tx,\frac{j_t}{\norm{y-J_tx}}\right\rangle&\geq\norm{\frac{y}{t}-\frac{J_tx}{t}}-\frac{\norm{y-x}}{t}\\
&\geq\frac{\norm{J_tx}}{t}-\frac{\norm{y}}{t}-\frac{\norm{y-x}}{t}.
\end{align*}
Thus for any $\delta\in\mathbb{R}_{>0}$ and any
\[
t\geq\max\left\{(\delta/3)^{-1}(c+b),\varphi(\delta/3,b,f)\right\},
\]
we obtain from Lemma \ref{lem:resolventROCInf} that 
\begin{align*}
\left\langle\frac{z}{\norm{z}},\frac{j_t}{\norm{y-J_tx}}\right\rangle&\geq \frac{d}{\norm{z}}-\frac{\delta}{\norm{z}}\\
&\geq\frac{1}{1+2\eta(\min\{\varepsilon/2,2\})}-\frac{\delta}{\norm{z}}.
\end{align*}
Now we get $1-\left(2\eta(\min\{\varepsilon/2,2\})\right)^2+\left(2\eta(\min\{\varepsilon/2,2\})\right)^2=1$ and therefore 
\[
1= (1+2\eta(\min\{\varepsilon/2,2\}))(1-2\eta(\min\{\varepsilon/2,2\}))+\left(2\eta(\min\{\varepsilon/2,2\})\right)^2
\]
which yields
\begin{align*}
\frac{1}{1+2\eta(\min\{\varepsilon/2,2\})}&= 1-2\eta(\min\{\varepsilon/2,2\})+\frac{\left(2\eta(\min\{\varepsilon/2,2\})\right)^2}{1+2\eta(\min\{\varepsilon/2,2\})}\\
&\geq 1-2\eta(\min\{\varepsilon/2,2\})+\frac{\left(2\eta(\min\{\varepsilon/2,2\})\right)^2}{3}.
\end{align*}
Thus for
\[
t\geq\max\left\{\left(\frac{D\left(2\eta(\min\{\varepsilon/2,2\})\right)^2}{18}\right)^{-1}(c+b),\varphi\left(\frac{D\left(2\eta(\min\{\varepsilon/2,2\})\right)^2}{18},b,f\right)\right\},
\]
we obtain (using $z\in\mathrm{ran}A$) that
\begin{align*}
\left\langle\frac{z}{\norm{z}},\frac{j_t}{\norm{y-J_tx}}\right\rangle&\geq 1-2\eta(\min\{\varepsilon/2,2\})+\frac{\left(2\eta(\min\{\varepsilon/2,2\})\right)^2}{3}\\
&\qquad-\frac{D\left(2\eta(\min\{\varepsilon/2,2\})\right)^2}{6\norm{z}}\\
&\geq 1-2\eta(\min\{\varepsilon/2,2\})+\frac{\left(2\eta(\min\{\varepsilon/2,2\})\right)^2}{6}\\
&> 1-2\eta(\min\{\varepsilon/2,2\}).
\end{align*}
Then in particular
\[
\norm{\frac{z}{\norm{z}}-\frac{y-J_tx}{\norm{y-J_tx}}}\leq\frac{\varepsilon}{2}
\]
by Lemma \ref{lem:reichLem} for all such $t$.\\

Now, secondly:
\[
\norm{\frac{y-J_tx}{\norm{y-J_tx}}+\frac{J_tx}{\norm{J_tx}}}\leq\frac{\norm{y}}{\left\vert\norm{y}-\norm{J_tx}\right\vert}+\left\vert 1-\frac{\norm{J_tx}}{\norm{y-J_tx}}\right\vert
\]
For $\delta>0$ and $t\geq\psi((\delta^{-1}+1)c,b,D)$, we immediately have
\[
\frac{\norm{y}}{\left\vert\norm{y}-\norm{J_tx}\right\vert}\leq\frac{c}{(\delta^{-1}+1)c-\norm{y}} \leq\delta
\]
by Lemma \ref{lem:quantReichLem}. Similarly, we get for $t\geq\psi((\delta^{-1}+1)c,b,D)$, as $(\delta^{-1}+1)c\geq \delta^{-1}c$, that
\[
\frac{\norm{y}}{\norm{y}+\norm{J_tx}}\leq \frac{c}{\norm{y}+\delta^{-1}c}\leq\frac{c}{\delta^{-1}c}=\delta.
\]
Further, we have 
\[
1-\frac{\norm{y}}{\norm{y}+\norm{J_tx}}\leq\frac{\norm{J_tx}}{\norm{y-J_tx}}\leq 1+\frac{\norm{y}}{\vert\norm{y}-\norm{J_tx}\vert}
\]
and thus for $t\geq\psi((\delta^{-1}+1)c,b,D)$, we get 
\[
\left\vert 1-\frac{\norm{J_tx}}{\norm{y-J_tx}}\right\vert\leq \delta.
\]
Combining the above, we have that for any $t\geq\psi(((\varepsilon/4)^{-1}+1)c,b,D)$:
\[
\norm{\frac{y-J_tx}{\norm{y-J_tx}}+\frac{J_tx}{\norm{J_tx}}}\leq\frac{\varepsilon}{2}.
\]
Thus, finally for $t\geq\varphi_1(\varepsilon,b,D,c,\eta,f)$ we obtain the desired result by triangle inequality.
\end{proof}

\begin{lemma}\label{lem:CauchyRespart}
Let $X$ be a uniformly convex Banach space with a modulus of uniform convexity $\eta:(0,2]\to (0,1]$ with (w.l.o.g.) $\eta(\varepsilon)\leq\varepsilon$. Suppose that $f:\mathbb{R}_{>0}\to\mathbb{N}$ satisfies $f(\varepsilon)\geq f(\delta)$ for $\varepsilon\leq \delta$ and 
\[
\forall \varepsilon>0\exists (y,z)\in A \left( \norm{y},\norm{z}\leq f(\varepsilon)\land \norm{z}-d\leq\varepsilon\right).
\]
Let $\varepsilon>0$ be given, assume that $E\in\mathbb{N}^*$, $D\in\mathbb{R}_{>0}$ where $E\geq d\geq D$. Let $x\in\mathrm{dom }A$ with $v\in Ax$ and $b\in\mathbb{N}^*$ where $b\geq \norm{x},\norm{v}$. Then, for any $t,s\geq\varphi_2'(\varepsilon,b,D,\eta,E,f)$:
\[
\norm{\frac{J_sx}{s}-\frac{J_tx}{t}}\leq \varepsilon
\]
where
\[
\varphi_2'(\varepsilon,b,D,\eta,E,f):=\max\{\varphi(\varepsilon/3,b,f),\varphi_1(\varepsilon/6E,b,D,f(2D\eta(\min\{\varepsilon/12E,2\})),\eta,f)\}
\]
with $\varphi$ as in Lemma \ref{lem:resolventROCInf} and $\varphi_1$ as in Lemma \ref{lem:CauchyIntRes}.
\end{lemma}
\begin{proof}
We have that there exits $z\in Ay$ such that $\norm{z}\leq d+2d\eta(\min\{\varepsilon/4,2\})$ with $\norm{y},\norm{z}\leq f(2D\eta(\min\{\varepsilon/4,2\}))$. Thus, using Lemma \ref{lem:CauchyIntRes}, we have for
\[
t,s\geq\varphi_1(\varepsilon/2,b,D,f(2D\eta(\min\{\varepsilon/4,2\})),\eta,f)
\]
that it holds that
\[
\norm{\frac{J_sx}{\norm{J_sx}}-\frac{J_tx}{\norm{J_tx}}}\leq\norm{\frac{z}{\norm{z}}+\frac{J_sx}{\norm{J_sx}}}+\norm{-\frac{z}{\norm{z}}-\frac{J_tx}{\norm{J_tx}}}\leq \varepsilon.
\]
Therefore, we in particular have that
\begin{align*}
\norm{\frac{J_sx}{s}-\frac{J_tx}{t}}&=\norm{\frac{J_sx}{\norm{J_sx}}\frac{\norm{J_sx}}{s}-\frac{J_tx}{\norm{J_tx}}\frac{\norm{J_tx}}{t}}\\
&\leq\norm{\frac{J_sx}{\norm{J_sx}}\frac{\norm{J_sx}}{s}-\frac{J_sx}{\norm{J_sx}}d}+\norm{\frac{J_sx}{\norm{J_sx}}d-\frac{J_tx}{\norm{J_tx}}d}\\
&\qquad+\norm{\frac{J_tx}{\norm{J_tx}}d-\frac{J_tx}{\norm{J_tx}}\frac{\norm{J_tx}}{t}}\\
&\leq\left\vert\frac{\norm{J_sx}}{s}-d\right\vert+d\norm{\frac{J_sx}{\norm{J_sx}}-\frac{J_tx}{\norm{J_tx}}}+\left\vert d-\frac{\norm{J_tx}}{t}\right\vert.
\end{align*}
Thus, for $t,s\geq\varphi_2'(\varepsilon,b,D,\eta,E,f)$, we get the claim by Lemma \ref{lem:resolventROCInf} together with the above.
\end{proof}
This result, which presents the quantitative version of the Cauchyness of $J_tx/t$ in the case that $d>0$, can now be smoothed to omit this assumption. For this, note that through the trivial proof of the case of $d=0$, one obtains the following quantitative version of the full result:
\begin{lemma}\label{lem:CauchyRes}
Let $X$ be a uniformly convex Banach space with a modulus of uniform convexity $\eta:(0,2]\to (0,1]$ with (w.l.o.g.) $\eta(\varepsilon)\leq\varepsilon$. Suppose that $f:\mathbb{R}_{>0}\to\mathbb{N}$ satisfies $f(\varepsilon)\geq f(\delta)$ for $\varepsilon\leq \delta$ and 
\[
\forall \varepsilon>0\exists (y,z)\in A \left( \norm{y},\norm{z}\leq f(\varepsilon)\land \norm{z}-d\leq\varepsilon\right).
\]
Let $\varepsilon>0$ be given, assume that $E\in\mathbb{N}^*$ where $E\geq d$. Let $x\in\mathrm{dom }A$ with $v\in Ax$ and $b\in\mathbb{N}^*$ where $b\geq \norm{x},\norm{v}$. Then, for any $t,s\geq\varphi_2(\varepsilon,b,\eta,E,f)$:
\[
\norm{\frac{J_sx}{s}-\frac{J_tx}{t}}\leq \varepsilon
\]
where
\begin{align*}
\varphi_2(\varepsilon,b,\eta,E,f):=\max\{&\varphi(\varepsilon/4,b,f),\varphi(\varepsilon/3,b,f),\\
&\varphi_1(\varepsilon/6E,b,\varepsilon/4,f(\varepsilon\eta(\min\{\varepsilon/12E,2\})/2),\eta,f)\}
\end{align*}
with $\varphi$ as in Lemma \ref{lem:resolventROCInf} and $\varphi_1$ as in Lemma \ref{lem:CauchyIntRes}.
\end{lemma}
\begin{proof}
Suppose that $d\leq\varepsilon/4$. By Lemma \ref{lem:resolventROCInf}, we have that
\[
\left\vert\frac{\norm{J_tx}}{t}-d\right\vert\leq\varepsilon/4
\]
for any $t\geq\varphi(\varepsilon/4,b,f)$. Thus in particular we have that $\norm{J_tx}/t\leq\varepsilon/2$ for all such $t$ and thus
\[
\norm{\frac{J_sx}{s}-\frac{J_tx}{t}}\leq \norm{\frac{J_sx}{s}}+\norm{\frac{J_tx}{t}}\leq\varepsilon
\]
for all $t,s\geq \varphi(\varepsilon/4,b,f)$ in that case. Otherwise $d\geq\varepsilon/4$ and thus the above result holds for $t,s\geq\varphi_2'(\varepsilon,b,\varepsilon/4,\eta,E,f)$ by Lemma \ref{lem:CauchyRespart}, with $D=\varepsilon/4$.
\end{proof}

The rest of the proof given in \cite{Rei1981} now relies on the use of the limit $-v_x$ of $J_tx/t$ for $t\to\infty$. By the above lemma, this limit exists as $X$ is complete. While we emphasized that this limit a priori depends on the starting point $x$, the following lemma (which provides a concrete quantitative version of Lemma 3.2 given in \cite{Rei1981}) shows that this limit is actually unique, i.e all the $v_x$ coincide.

\begin{lemma}\label{lem:Reich32quantpart}
Let $X$ be a uniformly convex Banach space with a modulus of uniform convexity $\eta:(0,2]\to (0,1]$ with (w.l.o.g.) $\eta(\varepsilon)\leq\varepsilon$. Suppose that $f:\mathbb{R}_{>0}\to\mathbb{N}$ satisfies $f(\varepsilon)\geq f(\delta)$ for $\varepsilon\leq \delta$ and 
\[
\forall \varepsilon>0\exists (y,z)\in A \left( \norm{y},\norm{z}\leq f(\varepsilon)\land \norm{z}-d\leq\varepsilon\right).
\]
Let $\varepsilon>0$ be given, assume that $E\in\mathbb{N}^*$ and $D\in\mathbb{R}_{>0}$ where $E\geq d\geq D$ and let $z\in Ay$ be such that
\[
\norm{z}\leq d+2d\eta\left(\min\left\{\frac{\varepsilon}{16(E+1)},2\right\}\right).
\]
If $x\in\mathrm{dom }A$, then $\norm{z-v_x}\leq \varepsilon$.
\end{lemma}
\begin{proof}
We write $\delta_\varepsilon=2d\eta(\min\{\varepsilon/16(E+1),2\})$. Then, for $\norm{z}\leq d+\delta_\varepsilon$, we have
\begin{align*}
\norm{z+\frac{J_tx}{t}}&=\norm{z+\frac{J_tx}{\norm{J_tx}}\frac{\norm{J_tx}}{t}}\\
&\leq\norm{z-\frac{d}{\norm{z}}z}+\norm{\frac{d}{\norm{z}}z+\frac{J_tx}{\norm{J_tx}}\frac{\norm{J_tx}}{t}}\\
&\leq\norm{z}-d+\norm{\frac{d}{\norm{z}}z+\frac{J_tx}{\norm{J_tx}}\frac{\norm{J_tx}}{t}}\\
&\leq \delta_\varepsilon+\norm{\frac{d}{\norm{z}}z+\frac{J_tx}{\norm{J_tx}}\frac{\norm{J_tx}}{t}}.
\end{align*}
Similar to before, we have
\begin{align*}
\norm{\frac{d}{\norm{z}}z+\frac{J_tx}{\norm{J_tx}}\frac{\norm{J_tx}}{t}}&\leq\norm{\frac{d}{\norm{z}}z-\frac{z}{\norm{z}}\frac{\norm{J_tx}}{t}}+\norm{\frac{z}{\norm{z}}\frac{\norm{J_tx}}{t}+\frac{J_tx}{\norm{J_tx}}\frac{\norm{J_tx}}{t}}\\
&=\left\vert d-\frac{\norm{J_tx}}{t}\right\vert+\frac{\norm{J_tx}}{t}\norm{\frac{z}{\norm{z}}+\frac{J_tx}{\norm{J_tx}}}.
\end{align*}
From this we obtain that
\[
\norm{z+\frac{J_tx}{t}}\leq\delta_\varepsilon + \frac{\varepsilon}{4}
\]
for all
\[
t\geq\max\{\varphi(\min\{\varepsilon/8,1\},b,f),\varphi_1(\varepsilon/8(E+1),b,D,c,\eta,f)\}
\]
where $c,b\in\mathbb{N}^*$ are such that $c\geq\norm{y},\norm{z}$ and $b\geq\norm{x},\norm{v}$ for $v\in Ax$ as, for one, $t\geq\varphi(\min\{\varepsilon/8,1\},b,f)$ and thus
\[
\left\vert d-\frac{\norm{J_tx}}{t}\right\vert\leq \min\{\varepsilon/8,1\}
\]
by Lemma \ref{lem:resolventROCInf} as well as $\frac{\norm{J_tx}}{t}\leq d+1\leq E+1$ and, for another, $t\geq\varphi_1(\varepsilon/8(E+1),b,D,c,\eta,f)$ and thus
\[
\frac{\norm{J_tx}}{t}\norm{\frac{z}{\norm{z}}+\frac{J_tx}{\norm{J_tx}}}\leq \varepsilon/8
\]
by Lemma \ref{lem:CauchyIntRes}. Then the properties of $\eta$ imply that $\delta_\varepsilon\leq \varepsilon/4$ and thus 
\[
\norm{z+\frac{J_tx}{t}}\leq \varepsilon/2
\]
for all such $t$. Then
\[
\norm{z-v_x}\leq\norm{z+\frac{J_tx}{t}}+\norm{v_x+\frac{J_tx}{t}}
\]
for all $t$ and thus choosing
\[
t=\max\{\varphi(\min\{\varepsilon/8,1\},b,f),\varphi_1(\varepsilon/8(E+1),b,D,c,\eta,f),\varphi_2(\varepsilon/2,b,\eta,E,f)\}
\]
implies $\norm{z-v_x}\leq \varepsilon$ by definition of $v_x$ (which yields that $\varphi_2$ is a rate of convergence for $J_tx/t$ towards $-v_x$) and Lemma \ref{lem:CauchyRes}.
\end{proof}

\begin{lemma}\label{lem:Reich32quant}
Let $X$ be a uniformly convex Banach space with a modulus of uniform convexity $\eta:(0,2]\to (0,1]$ with (w.l.o.g.) $\eta(\varepsilon)\leq\varepsilon$. Suppose that $f:\mathbb{R}_{>0}\to\mathbb{N}$ satisfies $f(\varepsilon)\geq f(\delta)$ for $\varepsilon\leq \delta$ and 
\[
\forall \varepsilon>0\exists (y,z)\in A \left( \norm{y},\norm{z}\leq f(\varepsilon)\land \norm{z}-d\leq\varepsilon\right).
\]
Let $\varepsilon>0$ be given, assume that $E\in\mathbb{N}^*$ where $E\geq d$ and let $z\in Ay$ be such that
\[
\norm{z}\leq d+\min\left\{\varepsilon\eta\left(\min\left\{\frac{\varepsilon}{16(E+1)},2\right\}\right)/4,\varepsilon/8\right\}.
\]
If $x\in\mathrm{dom }A$, then $\norm{z-v_x}\leq \varepsilon$.
\end{lemma}
\begin{proof}
Let $\varepsilon$ be given. Then either $d\leq\varepsilon/8$ which, since $\norm{z}\leq d+\varepsilon/8$, implies $\norm{z}\leq\varepsilon/4$. For
\[
t\geq\max\{\varphi(\varepsilon/4,b,f),\varphi_2(\varepsilon/4,b,\eta,E,f)\},
\]
we then get
\begin{align*}
\norm{z-v_x}&\leq\norm{z}+\norm{v_x}\\
&\leq\norm{z}+d+\left\vert d-\frac{\norm{J_tx}}{t}\right\vert+\norm{v_x+\frac{J_tx}{t}}\\
&\leq\varepsilon.
\end{align*}
Otherwise we have $d\geq\varepsilon/8$ and thus, we get the same result for
\[
\norm{z}\leq d+\varepsilon\eta\left(\min\left\{\frac{\varepsilon}{16(E+1)},2\right\}\right)/4\leq d+2d\eta\left(\min\left\{\frac{\varepsilon}{16(E+1)},2\right\}\right)
\]
by Lemma \ref{lem:Reich32quantpart}, with $D=\varepsilon/8$.
\end{proof}

\begin{theorem}
Let $X$ be a uniformly convex Banach space with a modulus of uniform convexity $\eta:(0,2]\to (0,1]$ with (w.l.o.g.) $\eta(\varepsilon)\leq\varepsilon$. Suppose that $f:\mathbb{R}_{>0}\to\mathbb{N}$ satisfies $f(\varepsilon)\geq f(\delta)$ for $\varepsilon\leq \delta$ and 
\[
\forall \varepsilon>0\exists (y,z)\in A \left( \norm{y},\norm{z}\leq f(\varepsilon)\land \norm{z}-d\leq\varepsilon\right).
\]
Assume that $E\in\mathbb{N}^*$ where $E\geq d$ and further that $x\in\mathrm{dom}A$ with $v\in Ax$ and $b\in\mathbb{N}^*$ with $b\geq\norm{x},\norm{v}$. Then
\[
\forall\varepsilon>0\forall t\geq \Phi(\varepsilon,b,\eta,E,f)\left(\frac{\norm{J^A_t x-S(t)x}}{t}\leq \varepsilon\right)
\]
where
\begin{align*}
\Phi(\varepsilon,b,\eta,E,f):=\max\Bigg\{&\frac{4}{\varepsilon}\left(b+f\left(\min\left\{\varepsilon\eta\left(\min\left\{\frac{\varepsilon/8}{16(E+1)},2\right\}\right)/32,\varepsilon/64\right\}\right)\right),\\
&\frac{8}{\varepsilon}f\left(\min\left\{\varepsilon\eta\left(\min\left\{\frac{\varepsilon/8}{16(E+1)},2\right\}\right)/32,\varepsilon/64\right\}\right),\\
&\varphi_2(\varepsilon/2,b,\eta,E,f)\Bigg\}
\end{align*}
with
\begin{align*}
\varphi(\varepsilon,b,f):=&\frac{8(b+f(\varepsilon/2))}{\varepsilon},\\
\psi(K,b,D):=&\frac{b+K}{D},\\
\varphi_1(\varepsilon,b,D,c,\eta,f):=&\max\Bigg\{\psi(c+1,b,D),\psi\left(\left(\frac{4}{\varepsilon}+1\right)c,b,D\right),\\
&\hphantom{\max\Bigg\{\,}\left(\frac{D\left(2\eta(\min\{\varepsilon/2,2\})\right)^2}{18}\right)^{-1}(c+b),\\
&\hphantom{\max\Bigg\{\,}\varphi\left(\frac{D\left(2\eta(\min\{\varepsilon/2,2\})\right)^2}{18},b,f\right)\Bigg\},\\
\varphi_2(\varepsilon,b,\eta,E,f):=&\max\{\varphi(\varepsilon/4,b,f),\varphi(\varepsilon/3,b,f),\\
&\hphantom{\max\{\,}\varphi_1(\varepsilon/6E,b,\varepsilon/4,f(\varepsilon\eta(\min\{\varepsilon/12E,2\})/2),\eta,f)\}.
\end{align*}
\end{theorem}
\begin{proof}
Given $\varepsilon$, there are $z\in Ay$ such that
\[
\norm{z}\leq d+\min\left\{\varepsilon\eta\left(\min\left\{\frac{\varepsilon/8}{16(E+1)},2\right\}\right)/32,\varepsilon/64\right\}
\]
and such that $\norm{y},\norm{z}\leq f\left(\min\left\{\varepsilon\eta\left(\min\left\{\frac{\varepsilon/8}{16(E+1)},2\right\}\right)/32,\varepsilon/64\right\}\right)$. Now, we in particular have $\norm{A_rJ_ra}\leq\vert AJ_ra\vert\leq\norm{A_ra}$ for all $a\in\mathrm{dom}A$ and thus
\[
\norm{A_{t/n}J^i_{t/n}y}\leq \norm{A_{t/n}J^{i-1}_{t/n}y}.
\]
Iterating this gives
\[
\norm{A_{t/n}J^{i}_{t/n}y}\leq \norm{A_{t/n}y}\leq\norm{z}\leq d+\min\left\{\varepsilon\eta\left(\min\left\{\frac{\varepsilon/8}{16(E+1)},2\right\}\right)/32,\varepsilon/64\right\}
\]
for all $i\in [0;n-1]$. Now we get
\[
\norm{A_{t/n}J^i_{t/n}y-v_x}\leq \varepsilon/8
\]
for any $i\in [0;n-1]$ by Lemma \ref{lem:Reich32quant} which implies $\norm{(y-J^n_{t/n}y)/t-v_x}\leq \varepsilon/8$ for all $t$ and all $n$ as
\begin{align*}
\norm{\frac{y-J^n_{t/n}y}{t}-v_x}&=\norm{\frac{\sum_{i=0}^{n-1}(J^i_{t/n}y-J^{i+1}_{t/n}y)}{nt/n}-\frac{\sum_{i=0}^{n-1}v_x}{n}}\\
&=\norm{\frac{\sum_{i=0}^{n-1}\left(\frac{J^i_{t/n}y-J^{i+1}_{t/n}y}{t/n}-v_x\right)}{n}}\\
&\leq \frac{\sum_{i=0}^{n-1}\norm{A_{t/n}J^i_{t/n}y-v_x}}{n}.
\end{align*}
Thus
\[
\norm{\frac{y-S(t)y}{t}-v_x}\leq \varepsilon/8
\]
for all $t$. Then in particular
\[
\norm{\frac{S(t)y}{t}+v_x}\leq\frac{\norm{y}}{t}+\norm{\frac{y-S(t)y}{t}-v_x}\leq \frac{\norm{y}}{t}+\varepsilon/8
\]
for all $t$. In particular, for 
\[
t\geq (\varepsilon/8)^{-1}f\left(\min\left\{\varepsilon\eta\left(\min\left\{\frac{\varepsilon/8}{16(E+1)},2\right\}\right)/32,\varepsilon/64\right\}\right), 
\]
we have
\[
\norm{\frac{S(t)y}{t}+v_x}\leq \varepsilon/4.
\]
Continuing, we obtain
\[
\norm{\frac{S(t)x}{t}+v_x}\leq\norm{\frac{S(t)y}{t}+v_x}+\frac{\norm{x-y}}{t}
\]
which implies 
\[
\norm{\frac{S(t)x}{t}+v_x}\leq \varepsilon/2
\]
for all 
\begin{align*}
t\geq\max\Bigg\{&(\varepsilon/8)^{-1}f\left(\min\left\{\varepsilon\eta\left(\min\left\{\frac{\varepsilon/8}{16(E+1)},2\right\}\right)/32,\varepsilon/64\right\}\right),\\
&(\varepsilon/4)^{-1}\left(b+f\left(\min\left\{\varepsilon\eta\left(\min\left\{\frac{\varepsilon/8}{16(E+1)},2\right\}\right)/32,\varepsilon/64\right\}\right)\right)\Bigg\}.
\end{align*}
Finally, we get 
\[
\norm{\frac{S(t)x-J_tx}{t}}\leq\norm{\frac{S(t)x}{t}+v_x}+\norm{v_x+\frac{J_tx}{t}}
\]
and thus 
\[
\norm{\frac{S(t)x-J_tx}{t}}\leq \varepsilon
\]
for all $t\geq\Phi(\varepsilon,b,\eta,E,f)$ by Lemma \ref{lem:CauchyRes} and the definition of $v_x$ (which yields that $\varphi_2$ is a rate of convergence as before).
\end{proof}

\subsection{Logical remarks on the above results}

Lastly, we want to outline the additional modifications to $H^\omega_p$ necessary for formalizing the proofs of the theorems of Plant and Reich. These modifications in that way give rise to the systems and bound extraction results underlying the extractions outlined above. In that context, we here in particular move away from the use of arbitrary real errors $\varepsilon$ and again consider representations of errors via natural numbers through $2^{-k}$.

\medskip

At first, both results are formulated for points $x\in\mathrm{dom}A$ and by the logical methodology, this stands for the existential assumption $\exists y^X(y\in Ax)$ which yields that at least a priori the extracted rates will in particular depend on an upper bound on the norm of this witness which is also the case for the above rates.

\medskip

The second prominent assumption in both results is that of uniform convexity which was quantitatively treated above via the modulus of uniform convexity $\eta$. Formally, this can be achieved by adding an additional constant $\eta$ of type $\mathbb{N}\to\mathbb{N}$ together with a corresponding axiom stating that it represents a modulus of uniform convexity for $X$ (see \cite{Koh2008} for more details):
\begin{gather*}
\forall x^X,y^X,k^\mathbb{N}\bigg(\norm{x}_X,\norm{y}_X<_\mathbb{R}1\land\norm{\frac{x+_Xy}{2}}_X>_\mathbb{R}1-2^{-\eta(k)}\\
\to\norm{x-_Xy}_X\leq_\mathbb{R}2^{-k}\bigg).
\end{gather*}
To really formally encapsulate the previous proof where $\eta$ was applied to various reals, one would first have to extend $\eta$ to $\mathbb{Q}\cap (0,2]$ via
\[
\eta(\varepsilon):=2^{-\eta(\min k[2^{-k}\leq\varepsilon])}
\]
for $\varepsilon\in\mathbb{Q}\cap (0,2]$ and then move to rational approximations of the reals in question. We avoid spelling this out any further.

\medskip

We now first focus on the theorem of Plant. The main object featuring in Plant's proof (and consequently in the above results as well) is the limit functional $\vert Ax\vert$. The use of this functional can now be emulated in the context of $H^\omega_p$ by extending the underlying language with a further constant of type $X\to \mathbb{N}^\mathbb{N}$ which we denote by $\vert A \cdot\vert$ (where we correspondingly denote $\vert A\cdot\vert x$ by $\vert Ax\vert$ for simplicity). One first natural axiom for this constant is induced by the natural bound on $\vert Ax\vert$ by $\norm{v}$ for $v\in Ax$ witnessing $x\in\mathrm{dom}A$:
\[
\forall x^X,v^X,\lambda^{\mathbb{N}^\mathbb{N}}\left(v\in Ax\land 0<_\mathbb{R}\lambda<_\mathbb{R}\lambda_0\rightarrow \norm{A_\lambda x}_X\leq_\mathbb{R} \vert Ax\vert\leq_\mathbb{R}\norm{v}_X\right).\tag{$L1$}
\]
As shortly mentioned in the above quantitative results, the convergence of $\norm{A_\lambda x}$ to $\vert Ax\vert$ for $x\in\mathrm{dom}A$ as $\lambda\to 0$ is ``equivalent'' to the lower semicontinuity of $\vert Ax\vert$ on $\mathrm{dom}A$. This vague ``equivalence'' can now be made precise through the system $H^\omega_p+(L1)$ in the following sense:
\begin{proposition}\label{pro:ROCYosida}
Over $H^\omega_p+(L1)$, the following are equivalent:
\begin{enumerate}
\item $\norm{A_t x}\to\vert Ax\vert$ while $t\to 0^+$ for all $x\in\mathrm{dom}A$, i.e.
\[
\forall x^X, k^\mathbb{N}\exists n^\mathbb{N}\left(x\in\mathrm{dom}A\rightarrow \vert Ax\vert-\norm{A_{2^{-n}}x}_X\leq_\mathbb{R}2^{-k}\right);
\]
\item lower semicontinuity for $\vert Ax\vert$ for all $x\in\mathrm{dom}A$, i.e.
\begin{gather*}
\forall k^\mathbb{N},x^X\exists m^\mathbb{N}\forall y^X(x\in\mathrm{dom}A\land y\in\mathrm{dom}A \land\norm{x-_Xy}_X\leq_{\mathbb{R}}2^{-m}\\
\rightarrow \vert Ax\vert-\vert Ay\vert\leq_\mathbb{R} 2^{-k}).
\end{gather*}
\end{enumerate}
\end{proposition}
\begin{proof}
From (1) to (2), let $x\in\mathrm{dom}A$ and $k$ be given. For $y\in\mathrm{dom}A$, we have
\begin{align*}
\vert Ax\vert-\vert Ay\vert&\leq\vert Ax\vert-\norm{A_\lambda y}\\
&\leq\vert Ax\vert-\norm{A_\lambda x}+\vert\norm{A_\lambda x}-\norm{A_\lambda y}\vert\\
&\leq\vert Ax\vert-\norm{A_\lambda x}+2/\lambda\norm{x-y}
\end{align*}
for any $\lambda\in (0,\lambda_0)$. Now, using (1) we pick $n$ such that $\vert Ax\vert-\norm{A_{2^{-n}} x}\leq 2^{-(k+1)}$ and then pick $m=n+k+2$ such that $2^{n+1}\norm{x-y}\leq 2^{-(k+1)}$ which yields $\vert Ax\vert-\vert Ay\vert\leq 2^{-k}$.\\

From (2) to (1), let $x\in\mathrm{dom}A$ and $k$ be given. Using (2), we pick an $m$ such that $\vert Ax\vert-\vert Ay\vert\leq 2^{-k}$ for all $y\in\mathrm{dom}A$ such that $\norm{x-y}\leq 2^{-m}$. Now, for $n\geq b+m$ for $b\geq\norm{v}$ for some $v\in Ax$, we then get
\[
\norm{x-J^A_{2^{-n}}x}\leq 2^{-n}\norm{v}\leq 2^{-m}
\]
which in particular implies
\[
\vert Ax\vert-\norm{A_{2^{-n}}x}\leq \vert Ax\vert-\vert AJ^A_{2^{-n}}x\vert\leq 2^{-k}
\]
using $A_{2^{-n}}x\in AJ^A_{2^{-n}}x$.
\end{proof}
In that way, the convergence of $\norm{A_\lambda x}$ to $\vert Ax\vert$ on $\mathrm{dom}A$ relates to an extensionality principle of $\vert A\cdot\vert$. Now, in the context of set-valued operators, these continuity and extensionality principles can be logically complicated and intricate (see \cite{Pis2022} as well as the forthcoming \cite{Pis2022b}). In any way, as in the case of the functional $\langle\cdot,\cdot\rangle_s$, the logical methodology based on the monotone Dialectica interpretation now implies the following quantitative version of the statement of item (2): under this interpretation, the statement (2) is upgraded to the existence of a ``modulus of uniform lower semicontinuity'' which, as with $\langle\cdot,\cdot\rangle_s$, by the uniformity on $x$ induced by majorization, is essentially a modulus of uniform continuity. Concretely, this uniformized version of item (2) can be formally hardwired into the system by extending it with an additional constant $\varphi$ of type $\mathbb{N}\to (\mathbb{N}\to\mathbb{N})$ together with the axiom
\begin{gather*}
\forall b^\mathbb{N},k^\mathbb{N},x^X,y^X,u^X,v^X\Big(\Big(\norm{x}_X,\norm{y}_X,\norm{u}_X,\norm{v}_X<_\mathbb{R}b\tag{L2}\\
\land u\in Ax\land v\in Ay\land \norm{x-_Xy}_X<_{\mathbb{R}}2^{-\varphi(k,b)}\Big)\rightarrow \vert Ax\vert-\vert Ay\vert\leq_\mathbb{R} 2^{-k}\Big).
\end{gather*}
Under this extension, Lemma \ref{lem:resolventROC} is then the natural extraction of a corresponding rate of convergence from the above equivalence proof, under this (therefore) necessary assumption of a modulus of uniform continuity for $\vert A\cdot\vert$, following the previous metatheorems. Note however that these metatheorems in general, through this treatment of $\vert A\cdot\vert$, imply a dependence of the extracted bounds on a majorant for the constant $\vert A\cdot\vert$, i.e.\ on a function $f:\mathbb{N}\to\mathbb{N}$ such that
\[
\norm{x}\leq b\to \vert Ax\vert\leq f(b)\text{ for all }x\in\mathrm{dom}A.
\]
Only under this additional dependence on a majorant for $\vert A\cdot\vert$ do the previous metatheorems contained in Theorems \ref{thm:metatheoremSC} and \ref{thm:metatheoremSCI} extend to $H^\omega_p+(L1)+(L2)$.

\medskip

Here, we shortly want to make a note on the strength of the existence of such a majorant. For this, we first remind on the notion of a majorizable operator introduced in \cite{Pis2022} (see also the previous Remark \ref{rem:unifEqui}): an operator $A:X\to 2^X$ is called majorizable if there exists a function $f:\mathbb{N}\to\mathbb{N}$ such that
\[
\forall x\in\mathrm{dom}A,b\in\mathbb{N}\left( \norm{x}\leq b\to \exists y\in Ax\left(\norm{y}\leq f(b)\right)\right).
\]
As discussed in \cite{Pis2022}, there are non-majorizable operators and so the assumption that $A$ is majorizable is a proper restriction. In particular, note now that if $A$ is such that the minimal selection $A^\circ x=\mathrm{argmin}\{\norm{y}\mid y\in Ax\}$ exists, then $\vert Ax\vert=\norm{A^\circ x}$ and thus majorizability of $A$ is equivalent to majorizability of $\vert A\cdot\vert$. Thus, while in general potentially a bit weaker, the assumption of majorizability of $\vert A\cdot\vert$ in particular also seems to carry additional strength similar to that of majorizability of $A$ in most cases.

\medskip

However, as apparent from the result in Lemma \ref{lem:resolventROC}, such a majorant however does not feature in the extracted bounds and we actually find that such a majorant also does not feature in any of the other quantitative results in the context of Plant's theorem. While this is seems to be a particular coincidence in the context of Lemma \ref{lem:resolventROC}, there is actually a logical reason which guarantees this ``non-dependence'' a priori for all the other results. Concretely, the reason is that all the other proofs analyzed have the two crucial properties that, for one, they can be formalized under the assumption of a rate of convergence for $\norm{A_\lambda x}$ toward $\vert Ax\vert$ which can similarly be added to the system and that, for another, they are ``pointwise'' results in $x$ in the sense that they do not require knowledge of $\vert A\cdot\vert$ for any point other than $x$. In that way, instead of following the above route of formalizing the whole functional $\vert A\cdot\vert$, one can add two constant representing this ``particular '' $x$ and a witness $v\in Ax$ as well as a constant $\vert Ax\vert$ of type $\mathbb{N}^\mathbb{N}$ for this single value of $\vert A\cdot\vert$ at the constant $x$ and a constant $\varphi$ representing a rate of convergence for $\norm{A_\lambda x}$ to $\vert Ax\vert$ for this single constant $x$. Then, the other proofs still formalize and in particular depend only on majorants for $\varphi$, $x$, $v$ and $\vert Ax\vert$ and the one for the latter three can be assumed to coincide and to be represented in the above results by $b$. In particular, the strong assumption of majorizability of $\vert A\cdot\vert$ can be avoided a priori in that way. That the extracted rates are true for all $x$ then is drawn as a conclusion on the metalevel as the additional constants were generic. In this way, this also provides a logical insight on why all the other results in the context of Plant's theorem remain true if $\varphi_1$ represents any other rate of convergence besides the one constructed from the modulus of uniform continuity for $\vert A\cdot\vert$ as commented on before.

\medskip

As a last comment on the logical particularities of the proofs towards Plant's theorem, we want to note in the context of Miyadera's lemma from \cite{Miy1971} that the only properties of $\langle\cdot,\cdot\rangle_s$ required in the proof given in \cite{Miy1971} are the properties discussed in Section \ref{sec:supremumIP}. Further, by the fact that the proof given by Crandall in \cite{Cra1973} of his respective result actually only invokes Miyadera's lemma for $x\in\mathrm{dom}A$ and for \emph{some} $\zeta^*\in J(x-x_0)$, this $\zeta^*$ can thus for simplicity be assumed to coincide with $j_{v,y_0}$ for $v\in Ax$ witnessing $x\in\mathrm{dom}A$ and $y_0\in Ax_0$ as in Miyadera's lemma. So, combined we have that this use of Miyadera's lemma in the context of the proof of Plant's result immediately formalizes in the system $H^\omega_p+(+)$.

\medskip

We now consider the theorem of Reich (which features less logical subtleties). The main object featuring in Reich's proof is the value $d$, the infimum over norms of all elements in the range of the operator. Internally in $H^\omega_p$, this value can be represented by adding a further constant of type $\mathbb{N}^\mathbb{N}$ which we, for simplicity, also denote by $d$ together with a further constant $f^{\mathbb{N}\to\mathbb{N}}$ representing the witness for the monotone Dialectica interpretation of the property
\[
\forall k\exists y, z\left( z\in Ay\land \norm{z}-d\leq 2^{-k}\right)
\]
expressing that $d$ indeed is the said infimum. So, we can concretely facilitate the use of $d$ by adding the following two axioms for $d$:
\[
\begin{cases}
\forall y^X,z^X\left( z\in Ay\rightarrow d\leq_\mathbb{R}\norm{z}_X\right),\\
\forall k^\mathbb{N}\exists y, z\preceq_X f(k)1_X\left( z\in Ay\land \norm{z}_X-d\leq_\mathbb{R}2^{-k}\right).
\end{cases}\tag{$d$}
\]
The additional constants are immediately majorizable: $f$ is majorized by $f^M$ as it is of type $\mathbb{N}^\mathbb{N}$ and $d$ is just majorized by $(n)_\circ$ for $n\geq \norm{d_X}$. Therefore the bound extraction theorems extend to this augmentation of $H^\omega_p$ in an immediate way where, in particular, the extracted bounds will in general depend on an upper bound on $d$ as can be seen from some of the bounds extracted in the context of Reich's theorem.

\medskip

The second particularity of the formalization of the proof of Reich's result is that one actually needs to work with the limit of $J_tx/t$, called $-v_x$ in the above, as a concrete object. In the context of the limit operator $C$ however, we can formally deal with this object in the context of the formal systems underlying this extraction rather immediately by utilizing the previously extracted rate $\varphi_2(k)$ (where we for simplicity omit the other parameters for now and switched back to a representation of errors via $2^{-k}$) to then address the limit $v_x$ in the system by considering
\[
v_x=-C\left(\left(\frac{J^A_{\varphi_2(k)}x}{\varphi_2(k)}\right)_k\right).
\]
In particular, with this definition of $v_x$, the other proofs in the context of Reich's theorem immediately formalize.

\medskip

As a last logical comment, we shortly want to discuss on the particular use of the law of excluded middle (and thus of classical logic) in the proofs for the results of Reich and Plant and how this features in the extractions, considering the fact that rates of convergence were nevertheless extracted in the absence of monotonicity. This in fact relates to the circumstances of the (previously called ``smoothable'') case distinctions. Namely, as can be observed by closer inspection of the corresponding proofs, the only part where classical logic actually features in the proofs of Reich and Plant is trough the use of multiple case distinctions which, in the case of Reich's result, take the form of dividing the proof between whether
\[
d=0\text{ or }d>0
\]
and, in the case of Plant's result, take the form of dividing the proof between whether
\[
\vert Ax\vert =0\text{ or }\vert Ax\vert >0.
\]
The deductions of the main results from both parts of this case distinction are essentially constructive (where the $=0$-case is almost trivial in both cases) and in that way, the constructive metatheorems actually allow for the extraction of a rate of convergence from the $>0$-cases as the corresponding results are of the form
\begin{gather*}
d>_\mathbb{R}0\to C\equiv \forall c^\mathbb{N}\left( d\geq_\mathbb{R} 2^{-c}\to C\right)\\
\text{ and }\\
\vert Ax\vert>_\mathbb{R}0\to C\equiv \forall c^\mathbb{N}\left( \vert Ax\vert\geq_\mathbb{R} 2^{-c}\to C\right) 
\end{gather*}
where $C$ is any of the respective convergence statements. These rates will moreover depend on the parameter $c$. For the $=0$-cases, being of the form
\[
d=_\mathbb{R}0\to C\text{ and }\vert Ax\vert =_\mathbb{R}0\to C
\]
where $C$ is any of the respective convergence statements, we find that in these cases one can actually find different constructive proofs (compared to the ones given in the literature) of the classically equivalent but constructively stronger statements
\[
\exists {c'}^\mathbb{N}\left( d\leq_\mathbb{R}2^{-{c'}}\to C\right)\text{ and } \exists {c'}^\mathbb{N}\left( \vert Ax\vert \leq_\mathbb{R}2^{-{c'}}\to C\right).
\]
These new proofs of said statements (which were presented and analyzed in the previous section) are again essentially constructive so that the constructive metatheorems guarantee the extraction of a rate again, now together with the extraction of an upper bound on (and thus a realizer of) the value $c'$. The previously mentioned ``smoothening'' is now just a combination of these two cases by instantiating the former rate with $c=c'$ and combining the two resulting rates.

\medskip

\noindent
{\bf Acknowledgments:}  I want to thank Ulrich Kohlenbach for many insightful and detailed comments on various drafts of this paper. The results of this paper form the main part of Chapters 4 and 7 of my doctoral dissertation \cite{Pis2024b} written under his supervision. I also want to thank Pedro Pinto for the collaboration on the work \cite{PP2022} which served as a strong source of inspiration for the logical systems presented here.

The author was supported by the `Deutsche Forschungs\-gemein\-schaft' Project DFG KO 1737/6-2.

\end{document}